\documentclass[a4paper,12pt,leqno]{amsart}
\usepackage{latexsym}
\usepackage[all]{xy}

\usepackage{amsmath, amsfonts, amssymb,amscd,graphics,graphicx,color,eucal,setspace,xypic} 
\usepackage{comment}
\usepackage{mathtools}

\usepackage{graphicx}
\usepackage{stmaryrd}
\usepackage{bigdelim, multirow}%for underbrace
\usepackage[all]{xy}%for commutative diagram

\definecolor{gray}{gray}{0.7}
\definecolor{Gray}{gray}{0.3}

\usepackage{amscd}

\numberwithin{equation}{section}

\theoremstyle{break}
 \newtheorem{theorem}{Theorem}[section]
 \newtheorem{proposition}[theorem]{Proposition}
 
 \newtheorem{lemma}[theorem]{Lemma}

 \theoremstyle{definition}
 \newtheorem{definition}[theorem]{Definition}
 \newtheorem{remark}[theorem]{Remark}
 \newtheorem{example}[theorem]{Example}

\allowdisplaybreaks[4]

\def\C{\mathbb C}
\def\R{\mathbb R}
\def\Q{\mathbb Q}
\def\Z{\mathbb Z}

\def\x{x}

\def\a{a}
\def\b{b}

\def\j{r}

\DeclareMathOperator{\Reduce}{Red}
\DeclareMathOperator{\Root}{Root}
\DeclareMathOperator{\wt}{wt}
\DeclareMathOperator{\pt}{pt}
\DeclareMathOperator{\pr}{pr}
\DeclareMathOperator{\diag}{diag}
\DeclareMathOperator{\rank}{rank}

\DeclareMathOperator{\Sym}{Sym}
\DeclareMathOperator{\Vol}{Vol}
\DeclareMathOperator{\GL}{GL}

\DeclareMathOperator{\Pet}{Pet}
\DeclareMathOperator{\Perm}{X}
\DeclareMathOperator{\X}{X}
\DeclareMathOperator{\Hess}{Hess}
\DeclareMathOperator{\Conv}{ConvexHull}
\newcommand{\Flag}{Fl}

\begin{document}
  
\title[Mixed Eulerian numbers and Peterson Schubert calculus]{Mixed Eulerian numbers and Peterson Schubert calculus}
\author[T. Horiguchi]{Tatsuya Horiguchi}
\address{National Institute of Technology, Ube College, 2-14-1, Tokiwadai, Ube, Yamaguchi, Japan 755-8555}
\email{tatsuya.horiguchi0103@gmail.com}



\begin{abstract}
Let $\Phi$ be a root system. Postnikov introduced and studied the mixed $\Phi$-Eulerian numbers.
These numbers indicate the mixed volumes of $\Phi$-hypersimplices. 
As specializations of these numbers, one can obtain the usual Eulerian numbers, the Catalan numbers, and the binomial coefficients.
Recent work of Berget--Spink--Tseng gave a simple computation for the mixed $\Phi$-Eulerian numbers when $\Phi$ is of type $A$. 
In this paper we connect a relation between mixed $\Phi$-Eulerian numbers and Peterson Schubert calculus. 
By using the connection, we provide a combinatorial model for the computation of Berget--Spink--Tseng in terms of left-right diagrams which were introduced by Abe--Horiguchi--Kuwata--Zeng for the purpose of Peterson Schubert calculus.
We also derive a simple computation for the mixed $\Phi$-Eulerian numbers in arbitrary Lie types from Peterson Schubert calculus.
\end{abstract}

\maketitle

\setcounter{tocdepth}{1}

\tableofcontents

%%%%%%%%%%%%%%%%%%%%%%%%%%%%%%%%%%
\section{Introduction}
\label{sect:Intro}
%%%%%%%%%%%%%%%%%%%%%%%%%%%%%%%%%%

Let $\Phi$ be a crystallographic root system of rank $n$. 
Let $\Lambda$ be the associated integer weight lattice and $\Lambda_{\R}=\Lambda \otimes \R$ the weight space.
The associated Weyl group $W$ acts on the weight space $\Lambda_{\R}$.
For $\chi \in \Lambda_{\R}$, the \emph{weight polytope} $P_\Phi(\chi)$ is defined to be the convex hull
of the Weyl group orbit of $\chi$:
\begin{align*}
P_\Phi(\chi):=\Conv\{ w(\chi) \in \Lambda_{\R} \mid w \in W \}.
\end{align*}
In type $A$ this polytope is called a \emph{permutohedron}.
Fix a set of simple roots $\Sigma \coloneqq \{\alpha_1, \ldots , \alpha_n\} \subset \Phi$.
Postnikov gave a formula of the volume of the weight polytope $P_\Phi(\chi)$ in \cite{Pos}.
Here, the volume form on $\Lambda_{\R}$ is normalized so that the volume of the parallelepiped generated by the simple roots $\alpha_1, \ldots , \alpha_n$ is $1$.

Let $\varpi_1,\ldots,\varpi_n$ be the fundamental weights.
We take $\chi=u_1\varpi_1+\cdots+u_n\varpi_n$ and consider the associated weight polytope $P_\Phi(\chi)$.
Its volume is a homogeneous polynomial $V_\Phi$ of degree $n$ in the variables $u_1,\ldots,u_n$:
\begin{align*}
V_\Phi(u_1,\ldots,u_n) \coloneqq {\rm volume \ of \ } P_\Phi(u_1\varpi_1+\cdots+u_n\varpi_n).
\end{align*}
The \emph{mixed $\Phi$-Eulerian numbers} $A^{\Phi}_{c_1,\ldots,c_n}$, for $c_1, \ldots , c_n \geq 0$ with $c_1 +\cdots+c_n =n$, are defined in \cite{Pos} as the coefficients of the polynomial 
\begin{align*} 
V_\Phi(u_1,\ldots,u_n) = \sum_{c_1,\ldots,c_n} A^{\Phi}_{c_1,\ldots,c_n} \frac{u_1^{c_1}}{c_1!} \cdots \frac{u_{n}^{c_{n}}}{c_{n}!}.
\end{align*}
When $\Phi$ is of type $A$, we call simply these numbers \emph{mixed Eulerian numbers}.
The mixed $\Phi$-Eulerian number means the mixed volume of the $\Phi$-hypersimplices multiplied by $n!$ (see \cite[Sections~16 and 18]{Pos}).
These numbers include many classical combinatorial numbers such as the Catalan numbers, binomial coefficients and Eulerian numbers (\cite{Pos}).
Also, there are various combinatorial formulas for $A^{\Phi}_{c_1,\ldots,c_n}$ (\cite{Cro,Liu,Pos}).
Recently, Nadeau and Tewari gave a beautiful relation between mixed $\Phi$-Eulerian numbers and the intersection number of the toric variety associated with a weight polytope and a Schubert variety (\cite{NT2}). 
Also, Berget, Spink, and Tseng gave a simple computation for the mixed Eulerian numbers (in type $A$) in the context of matroids (\cite{BST}).
In this paper we first see a combinatorial interpretation for their computation in terms of \emph{left-right diagrams} which are introduced in \cite{AHKZ} to compute structure constants of Schubert divisors in Peterson variety.
We then generalize the simple computation of \cite{BST} to other Lie types.
More specifically, we derive a simple computation for mixed $\Phi$-Eulerian numbers from a geometry of Peterson variety in all Lie types. 

Let $G$ be a simply connected semisimple algebraic group over $\C$ and fix a Borel subgroup $B$ of $G$. 
The Lie algebras of $G$ and $B$ are denoted by $\mathfrak{g}$ and $\mathfrak{b}$, respectively.
We also denote by $\Phi$ the associated root system. 
Let $N \in \mathfrak{g}$ be a regular nilpotent element. 
The \emph{Peterson variety} $\Pet_{\Phi}$ is defined to be the following subvariety of the flag variety $G/B$: 
\begin{align*}
\Pet_{\Phi} \coloneqq \left\{gB \in G/B \mid \mbox{Ad}(g^{-1})(N) \in \mathfrak{b} \oplus \bigoplus_{i=1}^{n} \mathfrak{g}_{-\alpha_i} \right\}
\end{align*}
where $\mathfrak{g}_{-\alpha_i}$ denotes the root space associated to the negative simple root $-\alpha_i$.
This variety arises in the study of the quantum cohomology of flag varieties (\cite{Kos, Rie}).
The geometry and topology of the Peterson varieties have been much studied.
It is known that $\Pet_\Phi$ is irreducible and its complex dimension is equal to $n=\rank(\Phi)$ (\cite{Pre18}).
The Peterson variety is singular in general (\cite[Theorem~4]{IY}, \cite[Theorem~6]{Kos}), while its cohomology ring $H^*(\Pet_\Phi)$ is a Poincar\'e duality algebra (\cite[Corollary~1.2]{AHMMS}).

Here and below, we assume that $\Phi$ is an irreducible root system.
Recall that the set of vertices of the Dynkin diagram of $\Phi$ is in one-to-one correspondence with the simple system $\Sigma = \{\alpha_1, \ldots , \alpha_n\}$. 
We assume to be fixed an ordering of the simple roots as given in \cite{Hum1}.
A subset of simple roots $K \subset \Sigma$ is called \emph{connected} if the induced
Dynkin diagram with the set of vertices $K$ is a connected subgraph of the Dynkin diagram of $\Phi$.

We now explain our computation for the mixed $\Phi$-Eulerian numbers.
Let $\sigma_{s_i}$ denote the Schubert class in the cohomology\footnote{In this paper, all cohomology groups will be taken with real coefficients.} $H^2(G/B)$ associated with $s_i \in W$ where $s_i$ is the simple reflection associated to the simple root $\alpha_i$.
By abuse of notation, we denote by $\varpi_i \in H^2(\Pet_{\Phi})$ the image of the Schubert class $\sigma_{s_i} \in H^2(G/B)$ under the restriction map $H^*(G/B) \to H^*(\Pet_{\Phi})$.
The main theorem is as follows.

\begin{theorem} \label{theorem:intro}
Let $\Phi$ be an irreducible root system.
Let $c_1,\ldots, c_n$ be non-negative integers with $c_1 +\cdots+c_n = n$.
\begin{enumerate}
\item The mixed $\Phi$-Eulerian number $A^{\Phi}_{c_1,\ldots,c_n}$ is equal to
\begin{align*} 
A^{\Phi}_{c_1,\ldots,c_n} = \int_{\Pet_{\Phi}} \varpi_1^{c_1}\varpi_2^{c_2}\cdots\varpi_n^{c_n}. 
\end{align*}
\item If $K \subset \Sigma = \{\alpha_1, \ldots , \alpha_n\}$ is connected and $K \ni \alpha_i$, then we have 
\begin{align*}
\varpi_i \cdot \prod_{\alpha_k \in K} \varpi_{k} = \sum_{J \subset \Sigma: \, {\rm connected} \atop J \supset K \, {\rm and } \, |J|=|K|+1} m_{i,K}^J \prod_{\alpha_j \in J} \varpi_{j}. 
\end{align*}
The coefficients $m_{i,K}^J$ are explicitly given. $($See Table~$\ref{tab:A list of values}$ in Section~$\ref{sect: Mixed Eulerian numbers and Peterson Schubert calculus}$.$)$
\item We have
\begin{align*} 
\int_{\Pet_{\Phi}} \varpi_1\varpi_2\cdots\varpi_n = \frac{|W|}{\det (C_{\Phi})}, 
\end{align*}
where $C_{\Phi}$ is the associated Cartan matrix.
Note that the right hand side can be explicitly computed. $($See Table~$\ref{tab:A list of values mPhi}$ in Section~$\ref{sect:Any_Lie_types}$.$)$
\end{enumerate}
\end{theorem}

Theorem~\ref{theorem:intro} yields a simple computation for the mixed $\Phi$-Eulerian numbers. 
In fact, by using $(2)$ repeatedly, a product $\varpi_1^{c_1}\varpi_2^{c_2}\cdots\varpi_n^{c_n}$  ($c_1, \ldots , c_n \geq 0, c_1 +\cdots+c_n =n$) can be easily expressed as some monomial $q \cdot \varpi_1\varpi_2\cdots\varpi_n$ $(q \in \Q)$.
Note that we can explicitly compute the rational number $q$ by using Table~$\ref{tab:A list of values}$ in Section~\ref{sect: Mixed Eulerian numbers and Peterson Schubert calculus}. 
Taking the integration over $\Pet_{\Phi}$ for the equality $\varpi_1^{c_1}\varpi_2^{c_2}\cdots\varpi_n^{c_n}=q \cdot \varpi_1\varpi_2\cdots\varpi_n$, we obtain $A^{\Phi}_{c_1,\ldots,c_n} = q \cdot \frac{|W|}{\det (C_{\Phi})}$ from $(1)$ and $(3)$.
As remarked before, this simple computation for the mixed $\Phi$-Eulerian numbers was given in \cite{BST} in the context of matroids when $\Phi$ is of type $A$.

Schubert calculus on Peterson variety was developed by Harada and Tymoczko (\cite{HaTy}).
We call it Peterson Schubert calculus (see Section~\ref{sect: Mixed Eulerian numbers and Peterson Schubert calculus} for details). 
There are several results on Peterson Schubert calculus (\cite{AHKZ, BaHa, Dre2, GoGo, GMS, HaTy}).
One can see that Peterson Schubert calculus is related to the mixed $\Phi$-Eulerian numbers by Theorem~\ref{theorem:intro}. 
In particular, we can provide an efficient computation for Peterson Schubert calculus by Theorem~\ref{theorem:intro}~(2) (Remark~\ref{rem:Peterson Schubert calculus}).

The paper is organized as follows.
After reviewing the definition of permutohedra and their volume in Section~\ref{sect:Permutohedron}, we see a topological interpretation for the volume of  permutohedra by using the equivariant cohomology of permutohedral varieties in Section~\ref{sect:equivariant_cohomology}. 
In Section~\ref{sect:Mixed Eulerian numbers_TypeA}, after recalling the definition of mixed Eulerian numbers (in type $A$), we interpret these numbers as an integration over the permutohedral variety. 
We then provide a combinatorial formula for the mixed Eulerian numbers in Section~\ref{sect:simple_computation_TypeA}.
Some results are extended to arbitrary Lie types in Section~\ref{sect:Any_Lie_types}. 
We then derive a simple computation for the mixed $\Phi$-Eulerian numbers (Theorem~\ref{theorem:intro}) from Peterson Schubert calculus in Section~\ref{sect: Mixed Eulerian numbers and Peterson Schubert calculus}.
Here, the coefficients $m_{i,K}^J$ in Theorem~\ref{theorem:intro}~(2) can be computed by a similar argument in \cite{AHKZ}, so we explain their computations in Appendix~\ref{sect:computation for miKJ}.
In the proof of Theorem~\ref{theorem:intro}, we will use the result that the Poincar\'e duals of the permutohedral variety and the Peterson variety are the same in the flag variety in all Lie types which was proved by Abe, Fujita, and Zeng (\cite[Corollary 3.9]{AFZ}). (In type $A$, the proof was given by Abe, DeDieu, Galetto, and Harada in \cite[Corollary~4.3]{ADGH}.)
Theorem~\ref{theorem:intro}~(3) follows from Theorem~\ref{theorem:intro}~(1) and the result in PhD Thesis of Croitoru (\cite[Proposition~2.7.5]{Cro}), while it also follows from the result of Klyachko (\cite[Theorem~3]{Kly}) together with the result of Abe, Fujita, and Zeng explained above.
Note that the result of Klyachko was proved in \cite{Kly2} in Russian, and it is probably a direct proof without going through the mixed $\Phi$-Eulerian numbers.
In other words, Theorem~\ref{theorem:intro}~(3) may be proved without going through the mixed $\Phi$-Eulerian numbers.
We prove the result of Klyachko by using equivariant cohomology of the permutohedral varieties (without going through the mixed $\Phi$-Eulerian numbers) in Appendix~\ref{sect: Integration}.

\bigskip

%%%%%%%%%%%%%%%%%%%%%%%%%%%%%%%%%%
\section{Permutohedron} \label{sect:Permutohedron}
%%%%%%%%%%%%%%%%%%%%%%%%%%%%%%%%%%

In this section we review some results on permutohedra in \cite{Pos}. 

Let $n$ be a positive integer. 
The permutation group $S_n$ on $n$ letters acts on $\R^n$ by permuting coordinates.
For $a_1, \ldots, a_n \in \R$, the \emph{permutohedron} $P_n(a_1,\ldots,a_n)$ is defined to be the convex hull of points of the $S_n$-orbit of $(a_1, \ldots, a_n)$:
\begin{align*}
P_n(a_1, \ldots, a_n):=\Conv\{ (a_{w(1)}, \ldots, a_{w(n)}) \in \R^n \mid w \in S_n \},
\end{align*}
which is at most $(n-1)$-dimensional, sitting inside the affine hyperplane $H_c = \{(t_1,\ldots, t_n) \in \R^n \mid t_1+ \cdots +t_n = c \}$ where $c = a_1 + \cdots + a_n$.
Without loss of generality, we may assume that $a_1 \geq a_2 \geq \dots \geq a_n$.

For a polytope $P \subset H_c$, its volume $\Vol P$ is defined to be the usual $(n-1)$-dimensional volume of the polytope $\pi(P) \subset \R^{n-1}$, where $\pi$ is the projection $\pi: (t_1,\ldots,t_n) \mapsto (t_1,\ldots,t_{n-1})$. 
If $c \in \Z$, then the volume of any parallelepiped formed by generators of the integer lattice $\Z^n \cap H_c$ is $1$. 
Postnikov gave a formula of the volume of the permutohedron $P_n(a_1,\ldots,a_n)$.

\begin{theorem}$($\cite[Theorem~3.1]{Pos}$)$
Let $t_1, \ldots, t_n \in \R$ be distinct real numbers.
The volume of the permutohedron $P_n=P_n(a_1,\ldots,a_n)$ is equal to 
\begin{align*}
\Vol P_n = \frac{1}{(n-1)!} \sum_{w \in S_n} \frac{(a_1 t_{w(1)}+\cdots+a_n t_{w(n)})^{n-1}}{(t_{w(1)}-t_{w(2)})(t_{w(2)}-t_{w(3)})\cdots(t_{w(n-1)}-t_{w(n)})}.
\end{align*}
\end{theorem}
Note that all $t_i$'s on the right hand side cancel each other after the symmetrization.
Motivated by this formula, the divided symmetrization is introduced in \cite[Section~3]{Pos}.
For a polynomial $f(t_1,\ldots,t_n)$, its \emph{divided symmetrization} is defined by
\begin{align} \label{eq:Div_Sym}
\langle f \rangle =\langle f(t_1,\ldots,t_n) \rangle:= \sum_{w \in S_n} w \left(\frac{f(t_1,\ldots,t_n)}{(t_{1}-t_{2})(t_{2}-t_{3})\cdots(t_{n-1}-t_{n})} \right),
\end{align}
where the symmetric group $S_n$ acts on the polynomial ring $\R[t_1,\ldots,t_n]$ by permuting the variables $t_1,\ldots,t_n$.
Then we can rewrite the volume of the permutohedron $P_n=P_n(a_1,\ldots,a_n)$ as  
\begin{align} \label{eq:Vol_Div_Sym}
\Vol P_n = \frac{1}{(n-1)!} \langle (a_1 t_{1}+\cdots+a_n t_{n})^{n-1} \rangle.
\end{align}
The divided symmetrization has been studied in \cite{Amd, NT1, Petr}.
We will see a topological interpretation for the divided symmetrization in the next section.

\bigskip

%%%%%%%%%%%%%%%%%%%%%%%%%%%%%%%%%%
\section{Equivariant cohomology of a permutohedral variety} \label{sect:equivariant_cohomology}
%%%%%%%%%%%%%%%%%%%%%%%%%%%%%%%%%%

The aim of this section is to establish a topological interpretation for the divided symmetrization. 
We refer to \cite{And} for equivariant cohomology.

Recall that the flag variety $\Flag(\C^n)$ in type $A_{n-1}$ is the set of nested subspaces $V_{\bullet}:=(V_1 \subset V_2 \subset \cdots \subset V_n=\C^n)$ of $\C^n$ where $\dim_{\C} V_i =i$ for all $i=1,\ldots,n$.
Let $S$ be a diagonal matrix with distinct eigenvalues.
Then it is known that a subvariety defined by
\begin{align} \label{eq:Perm_in_Flag}
\Perm_n \coloneqq \{V_{\bullet} \in \Flag(\C^n) \mid SV_i \subset V_{i+1} \ \textrm{for all } i=1,2,\ldots,n-1 \}
\end{align}
is a toric variety (\cite[Theorem~11]{dMPS}).
This toric variety is called a permutohedral variety.
The permutohedral variety $\Perm_n$ is smooth and its complex dimension is $n-1$ (\cite[Theorem~6 and Corollary~9]{dMPS}).

We discuss the equivariant cohomology rings of $\Flag(\C^n)$ and $\Perm_n$. 
Let $B$ be the set of upper triangular matrices in the general linear group $\GL_n(\C)$ and $T$  the set of diagonal matrices in $B$.
As is well-known, the flag variety $\Flag(\C^n)$ can be identified with $\GL_n(\C)/B$ and the torus $T$ naturally acts on the flag variety $\GL_n(\C)/B$ by left multiplication.
The $T$-action on $\Flag(\C^n)$ preserves $\Perm_n$ (\cite[Proposition~2]{dMPS}).

Let $E_i$ be a subbundle of the trivial vector bundle $\Flag(\C^n) \times \C^n$ over $\Flag(\C^n)$ whose fiber at a flag $V_{\bullet}$ is just $V_i$, which is called the $i$-th tautological vector bundle.
Consider the quotient line bundle $L_i\coloneqq E_i/E_{i-1}$ and its dual $L_i^*$.
We also denote the restriction $L_i^*|_{\Perm_n}$ by the same symbol $L_i^*$ when there are no confusion.
For $1 \leq i \leq n$, we denote by 
\begin{align} 
&\x_i^T \coloneqq c_1^T(L_i^*)=-c_1^T(L_i), \label{eq:tauiT} \\
&\x_i \coloneqq c_1(L_i^*)=-c_1(L_i), \label{eq:taui}
\end{align}
the $T$-equivariant (or ordinary) first Chern class of the line bundle $L_i^*$. 
Note that if we regard $L_i^*$ as the line bundle over $\Perm_n$, then 
$\x_i^T$ is an element of $H^2_T(\Perm_n)$.
Let $\C_i$ be the one dimensional representation of $T$ via the $i$-th projection $g=\diag(g_1, \ldots ,g_n) \mapsto g_i$, namely $g \cdot z=g_i z$ for $g \in T$ and $z \in \C_i$.
We denote by $(\C_i)^*$ the dual representation of $\C_i$ and set
\begin{align} \label{eq:ti}
t_i=c_1^T(\C_i^*)=-c_1^T(\C_i) \in H^2_T(\pt), \ 1 \leq i \leq n.
\end{align}
In this setting, as is well-known, we have $H^*_T(\pt) = \R[t_1, \ldots , t_n]$.

A useful technique in torus equivariant cohomology is the restriction to the fixed point set of the torus action. 
It is well-known that the $T$-fixed point set $\Flag(\C^n)^T$ is given by the set of permutation flags specified by $V_i := \textrm{span}_{\C}\{e_w(1), \ldots , e_w(i)\}$ for $w \in S_n$, where $\{e_1, \ldots , e_n\}$ is the standard basis of $\C^n$ (e.g. \cite[Lemma~2 in \S10.1]{Fult97}).
We may identify the $T$-fixed point set $\Flag(\C^n)^T$ with the permutation group $S_n$.
One can see that the $T$-fixed point set $\Perm_n^T$ consists of $\Flag(\C^n)^T$, namely $\Perm_n^T$ can be identified with $S_n$ again (\cite[Proposition~3]{dMPS}).
Based on the above discussions, we may consider the following commutative diagram:
\begin{equation} \label{eq:CD}
\begin{CD}
H^{\ast}_{T}(\Flag(\C^n))@>{\iota_1}>> \displaystyle H^*_T(\Flag(\C^n)^T) \cong \bigoplus_{w\in S_n} \R[t_1,\dots,t_n]\\
@V{}VV @VV{\textrm{identity map}}V\\
H^{\ast}_{T}(\Perm_n)@>{\iota_2}>> \displaystyle H^*_T(\Perm_n^T) \cong \bigoplus_{w\in S_n} \R[t_1,\dots,t_n]
\end{CD}
\end{equation}
where all the maps are induced from the inclusion maps on underlying spaces.
Since the odd degree cohomology groups of $\Flag(\C^n)$ and $\Perm_n$ vanish, both $\iota_1$ and $\iota_2$ are injective.
By these injectivity results, we may identify each of $H^*_T(\Flag(\C^n))$ and $H^*_T(\Perm_n)$ as a subring of $\bigoplus_{w \in S_n} \R[t_1,\ldots,t_n]$.
For an element $\alpha \in H^*_T(\Flag(\C^n))$ (or $\beta \in H^*_T(\Perm_n)$), we denote by $\alpha|_w$ (or $\beta|_w$) its $w$-th component of $\bigoplus_{w \in S_n} \R[t_1,\ldots,t_n]$.
One can easily see that the $w$-th component of $\x_i^T \in H^2_T(\Flag(\C^n))$ in \eqref{eq:tauiT} is given by $t_{w(i)}$ for $w \in S_n$.
It then follows from the commutative diagram \eqref{eq:CD} that the $w$-th component of $\x_i^T \in H^2_T(\Perm_n)$ in \eqref{eq:tauiT} is equal to 
\begin{align} \label{eq:Restriction_tauiTinPerm}
\x_i^T|_w=t_{w(i)} \ \textrm{for} \ w \in S_n.
\end{align}

We now explain the divided symmetrization given in \eqref{eq:Div_Sym} in terms of the equivariant cohomology of $\Perm_n$.
The collapsing map $\pr : \Perm_n \to \{\pt \}$ induces the equivariant Gysin map
$\pr_{!}^T: H^*_T(\Perm_n) \to H^{*-2(n-1)}_T(\pt)$.
By the Atiyah–Bott–Berline–Vergne formula (\cite{AtBo, BeVe}), we can compute the equivariant Gysin map by fixed point data as follows:
\begin{align} \label{eq:ABBVformula}
\pr_{!}^T(\alpha)=\sum_{w \in S_n} \frac{\alpha|_w}{e_w}
\end{align}
where $e_w$ denotes the $T$-equivariant Euler class of the normal bundle to the fixed point $w \in S_n \cong \Perm_n^T$.
It follows from \cite[Lemma~7]{dMPS} that the tangent space of $\Perm_n$ at a fixed point $w $ can be decomposed into 
\begin{align*}
\bigoplus_{i=1}^{n-1} \left( \C_{w(i+1)} \otimes \left(\C_{w(i)}\right)^* \right)
\end{align*}
as $T$-representations.
Hence, the $T$-equivariant Euler class $e_w$ of $\Perm_n$ at $w \in S_n$ is given by 
\begin{align} \label{eq:Euler_class_Perm}
e_w=(t_{w(1)}-t_{w(2)})(t_{w(2)}-t_{w(3)})\cdots(t_{w(n-1)}-t_{w(n)}).
\end{align}
We here recall that the equivariant Gysin map $\pr_{!}^T$ and the ordinary Gysin map $\pr_{!}$ commute with the forgetful maps:
\begin{equation} \label{eq:CD_1}
\begin{CD}
H^{\ast}_{T}(\Perm_n)@>{\pr_{!}^T}>> H^{*-2(n-1)}_T(\pt) \\
@V{}VV @VV{}V\\
H^{\ast}(\Perm_n)@>{\pr_{!}}>> H^{*-2(n-1)}(\pt)
\end{CD}
\end{equation}
Note that the ordinary Gysin map $\pr_{!}$ derives the usual Poincar\'e dual. 
That is, for $\alpha \in H^{2k}(\Perm_n)$ and $\beta \in H^{2(n-1)-2k}(\Perm_n)$, we have
\begin{align*}
\pr_{!}(\alpha\beta) = \int_{\Perm_n} \alpha \beta.
\end{align*}
The following lemma gives a topological interpretation for the divided symmetrization.

\begin{lemma} \label{lemma:Topology_divided_symmetrization_typeA}
Let $f(t_1,\ldots,t_n)$ be a homogeneous polynomial of degree $2(n-1)$ where $\deg t_i=2$ for $1 \leq i \leq n$. 
Then, its divided symmetrization in \eqref{eq:Div_Sym} is given by
\begin{align*}
\langle f(t_1,\ldots,t_n) \rangle = \int_{\Perm_n} f(\x_1,\ldots,\x_n)
\end{align*}
where $\x_i \in H^2(\Perm_n)$ is defined in \eqref{eq:taui}.
In particular, by \eqref{eq:Vol_Div_Sym} the volume of the permutohedron $P_n=P_n(a_1,\ldots,a_n)$ is equal to 
\begin{align} \label{eq:VolumePermTopology}
\Vol P_n = \frac{1}{(n-1)!} \int_{\Perm_n} (a_1 \x_{1}+\cdots+a_n \x_{n})^{n-1}.
\end{align}
\end{lemma}

\begin{proof}
By using \eqref{eq:Restriction_tauiTinPerm}, \eqref{eq:ABBVformula} and \eqref{eq:Euler_class_Perm}, the image of $f(\x^T_1,\ldots,\x^T_n)$ under the equivariant Gysin map $\pr_{!}^T$ can be computed as
\begin{align*}
\pr_{!}^T(f(\x^T_1,\ldots,\x^T_n))=&\sum_{w \in S_n} \frac{f(\x^T_1|_w,\ldots,\x^T_n|_w)}{(t_{w(1)}-t_{w(2)})(t_{w(2)}-t_{w(3)})\cdots(t_{w(n-1)}-t_{w(n)})} \\
=&\sum_{w \in S_n} \frac{f(t_{w(1)},\ldots,t_{w(n)})}{(t_{w(1)}-t_{w(2)})(t_{w(2)}-t_{w(3)})\cdots(t_{w(n-1)}-t_{w(n)})}  \\
=&\langle f(t_1,\ldots,t_n) \rangle.
\end{align*}
On the other hand, it follows from the commutative diagram \eqref{eq:CD_1} that
\begin{align*}
\pr_{!}^T(f(\x^T_1,\ldots,\x^T_n))=\pr_{!}(f(\x_1,\ldots,\x_n))=\int_{\Perm_n} f(\x_1,\ldots,\x_n),
\end{align*}
as desired.
\end{proof}

\begin{remark}
The argument above can be extended to a general framework for regular semisimple Hessenberg varieties $\Hess(S,h)$ by modifying the definition of the divided symmetrization.
It means the degree of $\Hess(S,h)$ used in the computation of \cite[(6.9)]{ADGH}.
\end{remark}

\bigskip

%%%%%%%%%%%%%%%%%%%%%%%%%%%%%%%%%%
\section{Mixed Eulerian numbers} \label{sect:Mixed Eulerian numbers_TypeA}
%%%%%%%%%%%%%%%%%%%%%%%%%%%%%%%%%%

In this section we describe mixed Eulerian numbers in terms of the integration over $\Perm_n$. 
We begin with the definition of mixed Eulerian numbers which is introduced in \cite{Pos}.

We return to the permutohedron $P_n= P_n(a_1,\ldots,a_n)$. 
Let us use the coordinates $u_1, \ldots ,u_{n-1}$ related to $a_1,\ldots,a_n$ by
\begin{align*}
u_1=a_1-a_2, \ u_2=a_2-a_3, \ \ldots, \ u_{n-1}=a_{n-1}-a_n.
\end{align*}
Then, the volume of $P_n$ can be written as
\begin{align} \label{eq:Mix.Eul.num}
\Vol P_n = \sum_{c_1,\ldots,c_{n-1}} A_{c_1,\ldots,c_{n-1}} \frac{u_1^{c_1}}{c_1!} \cdots \frac{u_{n-1}^{c_{n-1}}}{c_{n-1}!},
\end{align}
where the sum is over non-negative integers $c_1, \ldots , c_{n-1}$ with $c_1 +\cdots+c_{n-1} = n-1$ (\cite[Section~16]{Pos}). 
The coefficients $A_{c_1,\ldots,c_{n-1}}$ are called the \emph{mixed Eulerian numbers}.
The mixed Eulerian number means the mixed volume of hypersimplices multiplied by $(n-1)!$. 

We now expand the right hand side of \eqref{eq:VolumePermTopology} in terms of the variables $u_1,\ldots,u_{n-1}$. 
For this purpose, set 
\begin{align} \label{eq:varpi_i}
\varpi_i=\x_1+\cdots+\x_i
\end{align}
for $1 \leq i \leq n-1$.
Here, $\x_k$ denotes an element of $H^2(\Perm_n)$ given in \eqref{eq:taui}.
It is well-known that $\x_1+\cdots+\x_n=0$ and hence we have
\begin{align*}
a_1 \x_{1}+\cdots+a_n \x_{n}=&a_1\varpi_1+a_2(\varpi_2-\varpi_1)+\cdots+a_{n-1}(\varpi_{n-1}-\varpi_{n-2})+a_n(-\varpi_{n-1}) \\
=&u_1\varpi_1+\cdots+u_{n-1}\varpi_{n-1}.
\end{align*}
Therefore, we conclude that
\begin{align*}
\frac{1}{(n-1)!} \int_{\Perm_n} (a_1 \x_{1}+\cdots+a_n \x_{n})^{n-1} =& \frac{1}{(n-1)!} \int_{\Perm_n} (u_1\varpi_1+\cdots+u_{n-1}\varpi_{n-1})^{n-1} \\
=&\sum_{c_1,\ldots,c_{n-1}} \left( \int_{\Perm_n} \varpi_1^{c_1}\varpi_2^{c_2}\cdots\varpi_{n-1}^{c_{n-1}} \right) \frac{u_1^{c_1}}{c_1!} \cdots \frac{u_{n-1}^{c_{n-1}}}{c_{n-1}!},
\end{align*}
where the sum is over non-negative integers $c_1, \ldots , c_{n-1}$ with $c_1 +\cdots+c_{n-1} = n-1$.
From this together with \eqref{eq:VolumePermTopology} we obtain the following proposition.

\begin{proposition} \label{proposition:mixed Eulerian numberTypeA}
Let $c_1,\ldots, c_{n-1}$ be non-negative integers with $c_1 +\cdots+c_{n-1} = n-1$.
Then, the mixed Eulerian number $A_{c_1,\ldots,c_{n-1}}$ is given by
\begin{align*} 
A_{c_1,\ldots,c_{n-1}} = \int_{\Perm_n} \varpi_1^{c_1}\varpi_2^{c_2}\cdots\varpi_{n-1}^{c_{n-1}}.
\end{align*}
\end{proposition}

\begin{remark}
There are several proofs of Proposition~\ref{proposition:mixed Eulerian numberTypeA}. 
Berget, Spink, and Tseng gave a proof from a point of view of the connection between degrees of mixed intersections and normalized mixed volumes (\cite[Lemma~7.3]{BST}).
Also, a formula \cite[(5.2)]{NT2} given by Nadeau and Tewari together with Klyachko's result (\cite[Theorem~3]{Kly}, \cite{Kly2}) yields Proposition~\ref{proposition:mixed Eulerian numberTypeA}. 
For the result of Klyachko, see Proposition~\ref{prop:Top_Class} and Remark~\ref{remark:Klyachko} in Section~\ref{sect:Any_Lie_types}.
\end{remark}

\bigskip

%%%%%%%%%%%%%%%%%%%%%%%%%%%%%%%%%%
\section{Computation for mixed Eulerian numbers and left-right diagrams} \label{sect:simple_computation_TypeA}
%%%%%%%%%%%%%%%%%%%%%%%%%%%%%%%%%%

Postnikov provided in \cite{Pos} a combinatorial formula for the mixed $\Phi$-Eulerian numbers in terms of certain binary trees.
Recently, a simple computation for the mixed Eulerian numbers was given in \cite{BST} in the context of matroids.
In this section we explain the result of \cite{BST} reformulated in language of Peterson varieties.
Then we provide a combinatorial formula for the mixed Eulerian numbers in terms of left-right diagrams which are introduced in \cite{AHKZ} for the purpose of Peterson Schubert calculus.
We remark that Nadeau and Tewari introduce a $q$-deformation of mixed Eulerian numbers and give a combinatorial formula for them in \cite{NT3}.

We begin with the definition of Peterson variety in type $A$.
Let $N$ be a regular nilpotent matrix of size $n$, i.e. a matrix whose Jordan form consists of exactly one Jordan block with corresponding eigenvalue equal to 0. 
The \emph{Peterson variety} $\Pet_n$ (in type $A_{n-1}$) is defined to be the following subvariety of the flag variety $\Flag(\C^n)$:
\begin{align*}
\Pet_n=\{V_{\bullet} \in \Flag(\C^n) \mid NV_i \subset V_{i+1} \ {\rm for \ any} \ i=1,2,\ldots,n-1 \},
\end{align*}
which is irreducible and $\dim_{\C} \Pet_n=n-1$ (\cite[Theorem~6]{Kos}).

By abuse of notation, we denote by the same symbol $\x_i$ the image of $\x_i \in H^2(\Flag(\C^n))$ in \eqref{eq:taui} under the restriction map $H^*(\Flag(\C^n)) \to H^*(\Pet_n)$. 
By \cite[Corollary~4.3]{ADGH} the Poincar\'e dual of the permutohedral variety $\Perm_n$ coincides with the Poincar\'e dual of the Peterson variety $\Pet_n$ in $H^*(\Flag(\C^n))$.
This implies that 
\begin{align*}
\int_{\Perm_n} \varpi_1^{c_1}\varpi_2^{c_2}\cdots\varpi_{n-1}^{c_{n-1}}=\int_{\Pet_n} \varpi_1^{c_1}\varpi_2^{c_2}\cdots\varpi_{n-1}^{c_{n-1}}
\end{align*}
for non-negative integers $c_1,\ldots, c_{n-1}$ with $c_1 +\cdots+c_{n-1} = n-1$.
Hence, we can rewrite Proposition~\ref{proposition:mixed Eulerian numberTypeA} as follows.

\begin{proposition} \label{prop:mixed Eulerian numberTypeA Pet}
Let $c_1,\ldots, c_{n-1}$ be non-negative integers with $c_1 +\cdots+c_{n-1} = n-1$.
Then, the mixed Eulerian number $A_{c_1,\ldots,c_{n-1}}$ is equal to
\begin{align*} 
A_{c_1,\ldots,c_{n-1}} = \int_{\Pet_n} \varpi_1^{c_1}\varpi_2^{c_2}\cdots\varpi_{n-1}^{c_{n-1}}.
\end{align*}
\end{proposition}

\begin{proposition} $($\cite[Proposition~4.10]{AHKZ}$)$ \label{prop:Top_Class_TypeA}
We have
\begin{align*} 
\int_{\Pet_n} \varpi_1\varpi_2 \cdots \varpi_{n-1} =(n-1)!.
\end{align*}
\end{proposition}

\begin{lemma} $($\cite[Lemma~5.1]{AHKZ}$)$ \label{lemm:TypeA}
For $1 \leq \a \leq i \leq \b \leq n-1$, we have
\begin{align*}
\underbrace{\varpi_{\a}\varpi_{\a+1} \cdots \varpi_{i}}_{i-\a+1} \underbrace{\varpi_{i} \cdots \varpi_{\b-1} \varpi_{\b}}_{\b-i+1} = \frac{\b-i+1}{\b-\a+2} \varpi_{\a-1}\varpi_{\a} \cdots \varpi_{\b} + \frac{i-\a+1}{\b-\a+2} \varpi_{\a} \cdots \varpi_{\b}\varpi_{\b+1}
\end{align*}
in $H^*(\Pet_n)$.
Here, we take the convention $\varpi_0=\varpi_n=0$.
\end{lemma}

We can compute the mixed Eulerian numbers $A_{c_1,\ldots,c_{n-1}}$ by using Propositions~\ref{prop:mixed Eulerian numberTypeA Pet} and \ref{prop:Top_Class_TypeA} and Lemma~\ref{lemm:TypeA}.

\begin{example} \label{ex:calculus_TypeA}
Let $n=9$ and take $(c_1,\ldots,c_8)=(1,0,2,3,0,0,1,1)$. 
Then, the mixed Eulerian number $A_{c_1,\ldots,c_8}$ is equal to 
\begin{align*}
A_{c_1,\ldots,c_8}=\int_{\Pet_9} \varpi_1 \varpi_3^2 \varpi_4^3 \varpi_7\varpi_8 
\end{align*}
from Proposition~\ref{prop:mixed Eulerian numberTypeA Pet}.
The product $\varpi_1 \varpi_3^2 \varpi_4^3 \varpi_7\varpi_8$ in the right hand side above can be computed by using Lemma~$\ref{lemm:TypeA}$ repeatedly as follows.  
\hspace{-15pt}
\begin{align*}
&\varpi_4 \cdot \varpi_4 \cdot \varpi_3 \cdot (\varpi_1\varpi_3\varpi_4\varpi_7\varpi_8) \\
=&\varpi_4 \cdot \varpi_4 \cdot (\varpi_1\varpi_3^2\varpi_4\varpi_7\varpi_8) \\
=&\varpi_4 \cdot \varpi_4 \cdot \left( \frac{2}{3}(\varpi_1\varpi_2\varpi_3\varpi_4\varpi_7\varpi_8)+ \frac{1}{3}(\varpi_1\varpi_3\varpi_4\varpi_5\varpi_7\varpi_8) \right) \\
=&\varpi_4 \cdot \left( \frac{2}{3}(\varpi_1\varpi_2\varpi_3\varpi_4^2\varpi_7\varpi_8)+ \frac{1}{3}(\varpi_1\varpi_3\varpi_4^2\varpi_5\varpi_7\varpi_8) \right) \\
=&\varpi_4 \cdot \Big( \frac{2}{3}\cdot\frac{4}{5}(\varpi_1\varpi_2\varpi_3\varpi_4\varpi_5\varpi_7\varpi_8) + \frac{1}{3}\cdot\frac{2}{4}(\varpi_1\varpi_2\varpi_3\varpi_4\varpi_5\varpi_7\varpi_8) \\
& \hspace{30pt} + \frac{1}{3}\cdot\frac{2}{4}(\varpi_1\varpi_3\varpi_4\varpi_5\varpi_6\varpi_7\varpi_8) \Big) \\
=&\frac{2}{3}\cdot\frac{4}{5}(\varpi_1\varpi_2\varpi_3\varpi_4^2\varpi_5\varpi_7\varpi_8) + \frac{1}{3}\cdot\frac{2}{4}(\varpi_1\varpi_2\varpi_3\varpi_4^2\varpi_5\varpi_7\varpi_8) + \frac{1}{3}\cdot\frac{2}{4}(\varpi_1\varpi_3\varpi_4^2\varpi_5\varpi_6\varpi_7\varpi_8) \\
=& \left( \frac{2}{3}\cdot\frac{4}{5}\cdot\frac{4}{6}+ \frac{1}{3}\cdot\frac{2}{4}\cdot\frac{4}{6}+ \frac{1}{3}\cdot\frac{2}{4}\cdot\frac{5}{7} \right)\varpi_1\varpi_2\varpi_3\varpi_4\varpi_5\varpi_6\varpi_7\varpi_8. 
\end{align*}
Taking the integral over $\Pet_9$ for the equality above, we obtain
\begin{align*}
A_{c_1,\ldots,c_8}
=8! \left( \frac{2}{3}\cdot\frac{4}{5}\cdot\frac{4}{6}+\frac{1}{3}\cdot\frac{2}{4}\cdot\frac{4}{6}+\frac{1}{3}\cdot\frac{2}{4}\cdot\frac{5}{7}\right)  
=14336+4480+4800=23616
\end{align*}
from Proposition~\ref{prop:Top_Class_TypeA}.
\end{example}

The computation for the mixed Eulerian numbers like Example~\ref{ex:calculus_TypeA} was given in \cite{BST} in the context of matroids.
On the other hand, this computation is deeply related with Peterson Schubert calculus. 
This connection will be discussed in Section~\ref{sect: Mixed Eulerian numbers and Peterson Schubert calculus}.
The aim of this section is to describe a combinatorial formula for the mixed Eulerian numbers given by \cite{BST} in terms of \emph{left-right diagrams} which are introduced in \cite{AHKZ} for the purpose of Peterson Schubert calculus.

For non-negative integers $c_1,\ldots, c_{n-1}$ with $c_1 +\cdots+c_{n-1} = n-1$, define $M=\{ i_1, \ldots, i_{n-1} \}$ to be the multiset such that $1 \leq i_1 \leq i_2 \leq \cdots \leq i_{n-1} \leq n-1$ and such that 
\begin{align*}
\varpi_1^{c_1}\varpi_2^{c_2}\cdots\varpi_{n-1}^{c_{n-1}}=\varpi_{i_1}\varpi_{i_2}\cdots\varpi_{i_{n-1}}.
\end{align*}
Next, for the same set of $c_1,\ldots, c_{n-1}$ define the set $J$ as the set of indices corresponding to $\varpi_j$ which appear in the monomial $\varpi_1^{c_1}\varpi_2^{c_2}\cdots\varpi_{n-1}^{c_{n-1}}$. 
In other words
\begin{align}
J \coloneqq \{j \in [n-1] \mid c_j \geq1 \}. \label{eq:diagram_J} 
\end{align}
Notice that $J$ is a set while $M$ is a multiset. 
Finally define 
\begin{align}
I \coloneqq M \setminus J. \label{eq:diagram_I} 
\end{align}

We first prepare the following two steps.
\begin{enumerate}
\item Draw a rectangular grid of size $(1+|I|) \times (n-1)$. 
On the left side of the grid, write the elements of $I$ in weakly increasing order from top to bottom, but starting at the second row. 
Along the bottom of the grid, label the columns of the grid by the elements in $[n-1]$ in increasing order, starting from the left. See Example~\ref{ex:setup_TypeA} for an example.
\end{enumerate}
For each box in the grid, we define the \textit{row number} of the box as the number which is written beside the row containing the box, and define the \textit{column number} of the box as the number which is written below the column containing the box.
\begin{enumerate}
\setcounter{enumi}{1}
\item Shade the boxes in the first row whose column numbers belong to $J (\subseteq [n-1])$. Mark each box with a cross $\times$ whose row number is the same as the column number.
\end{enumerate}

\begin{example} \label{ex:setup_TypeA}
Let $n=9$ and take $(c_1,\ldots,c_8)=(1,0,2,3,0,0,1,1)$ as in the
previous example. 
Then, we have $M=\{1,3,3,4,4,4,7,8 \}$, $J=\{1,3,4,7,8\}$, and $I=\{3,4,4 \}$.
The resulting grid is described below.
\begin{figure}[h]
\begin{center}
\begin{picture}(160,95)
\put(0,83){\colorbox{gray}}
\put(0,87){\colorbox{gray}}
\put(0,93){\colorbox{gray}}
\put(0,97){\colorbox{gray}}
\put(5,83){\colorbox{gray}}
\put(5,87){\colorbox{gray}}
\put(5,93){\colorbox{gray}}
\put(5,97){\colorbox{gray}}
\put(10,83){\colorbox{gray}}
\put(10,87){\colorbox{gray}}
\put(10,93){\colorbox{gray}}
\put(10,97){\colorbox{gray}}
\put(14,83){\colorbox{gray}}
\put(14,87){\colorbox{gray}}
\put(14,93){\colorbox{gray}}
\put(14,97){\colorbox{gray}}

\put(40,83){\colorbox{gray}}
\put(40,87){\colorbox{gray}}
\put(40,93){\colorbox{gray}}
\put(40,97){\colorbox{gray}}
\put(45,83){\colorbox{gray}}
\put(45,87){\colorbox{gray}}
\put(45,93){\colorbox{gray}}
\put(45,97){\colorbox{gray}}
\put(50,83){\colorbox{gray}}
\put(50,87){\colorbox{gray}}
\put(50,93){\colorbox{gray}}
\put(50,97){\colorbox{gray}}
\put(54,83){\colorbox{gray}}
\put(54,87){\colorbox{gray}}
\put(54,93){\colorbox{gray}}
\put(54,97){\colorbox{gray}}

\put(60,83){\colorbox{gray}}
\put(60,87){\colorbox{gray}}
\put(60,93){\colorbox{gray}}
\put(60,97){\colorbox{gray}}
\put(65,83){\colorbox{gray}}
\put(65,87){\colorbox{gray}}
\put(65,93){\colorbox{gray}}
\put(65,97){\colorbox{gray}}
\put(70,83){\colorbox{gray}}
\put(70,87){\colorbox{gray}}
\put(70,93){\colorbox{gray}}
\put(70,97){\colorbox{gray}}
\put(74,83){\colorbox{gray}}
\put(74,87){\colorbox{gray}}
\put(74,93){\colorbox{gray}}
\put(74,97){\colorbox{gray}}

\put(120,83){\colorbox{gray}}
\put(120,87){\colorbox{gray}}
\put(120,93){\colorbox{gray}}
\put(120,97){\colorbox{gray}}
\put(125,83){\colorbox{gray}}
\put(125,87){\colorbox{gray}}
\put(125,93){\colorbox{gray}}
\put(125,97){\colorbox{gray}}
\put(130,83){\colorbox{gray}}
\put(130,87){\colorbox{gray}}
\put(130,93){\colorbox{gray}}
\put(130,97){\colorbox{gray}}
\put(134,83){\colorbox{gray}}
\put(134,87){\colorbox{gray}}
\put(134,93){\colorbox{gray}}
\put(134,97){\colorbox{gray}}

\put(140,83){\colorbox{gray}}
\put(140,87){\colorbox{gray}}
\put(140,93){\colorbox{gray}}
\put(140,97){\colorbox{gray}}
\put(145,83){\colorbox{gray}}
\put(145,87){\colorbox{gray}}
\put(145,93){\colorbox{gray}}
\put(145,97){\colorbox{gray}}
\put(150,83){\colorbox{gray}}
\put(150,87){\colorbox{gray}}
\put(150,93){\colorbox{gray}}
\put(150,97){\colorbox{gray}}
\put(154,83){\colorbox{gray}}
\put(154,87){\colorbox{gray}}
\put(154,93){\colorbox{gray}}
\put(154,97){\colorbox{gray}}

\put(0,20){\framebox(20,20)}
\put(20,20){\framebox(20,20)}
\put(40,20){\framebox(20,20)}
\put(60,20){\framebox(20,20)}
\put(80,20){\framebox(20,20)}
\put(100,20){\framebox(20,20)}
\put(120,20){\framebox(20,20)}
\put(140,20){\framebox(20,20)}
\put(0,40){\framebox(20,20)}
\put(20,40){\framebox(20,20)}
\put(40,40){\framebox(20,20)}
\put(60,40){\framebox(20,20)}
\put(80,40){\framebox(20,20)}
\put(100,40){\framebox(20,20)}
\put(120,40){\framebox(20,20)}
\put(140,40){\framebox(20,20)}
\put(0,60){\framebox(20,20)}
\put(20,60){\framebox(20,20)}
\put(40,60){\framebox(20,20)}
\put(60,60){\framebox(20,20)}
\put(80,60){\framebox(20,20)}
\put(100,60){\framebox(20,20)}
\put(120,60){\framebox(20,20)}
\put(140,60){\framebox(20,20)}
\put(0,80){\framebox(20,20)}
\put(20,80){\framebox(20,20)}
\put(40,80){\framebox(20,20)}
\put(60,80){\framebox(20,20)}
\put(80,80){\framebox(20,20)}
\put(100,80){\framebox(20,20)}
\put(120,80){\framebox(20,20)}
\put(140,80){\framebox(20,20)}

\put(-10,65){$3$}
\put(-10,45){$4$}
\put(-10,25){$4$}

\put(7,5){$1$}
\put(27,5){$2$}
\put(47,5){$3$}
\put(67,5){$4$}
\put(87,5){$5$}
\put(107,5){$6$}
\put(127,5){$7$}
\put(147,5){$8$}

\put(45,67){$\times$}
\put(65,47){$\times$}
\put(65,27){$\times$}
\end{picture}
\end{center}
\end{figure}
\end{example}

We now play a combinatorial game on the prepared grid above. 
The rule of the game is inductively determined as follows.

\noindent 
\begin{enumerate}
\item[\textbf{(Game):}] Assume that some boxes in the $i$-th row are shaded ($1\le i< |I|+1$). Then shade the boxes in the $(i+1)$-th row whose column numbers are the same as those of the shaded boxes in the $i$-th row. 
In the $(i+1)$-th row, consider the (maximal) consecutive string of shaded boxes which contains the (unique) marked box in the $(i+1)$-th row. 
By construction, there exists an unshaded box to the immediate right or immediate left of this consecutive string.
Pick one such adjacent unshaded box and shade it.
We call this the \emph{additional shaded box}. 
\end{enumerate}

We define a \emph{left-right diagram} associated with $(c_1,\ldots,c_{n-1})$ as a configuration of boxes on a square grid of size $(1+|I|) \times (n-1)$ whose shaded boxes can be obtained from a game. 
The set of left-right diagrams associated with $(c_1,\ldots, c_{n-1})$ is denote by $\Delta_{c_1,\ldots, c_{n-1}}$. 

\begin{example} \label{ex:diagram_TypeA}
We take $(c_1,\ldots,c_8)=(1,0,2,3,0,0,1,1)$ given in Example~$\ref{ex:setup_TypeA}$.
Then, the resulting left-right diagrams are shown in Figure~$\ref{picture:left-right_diagrams_TypeA}$ where the boxes marked with a circle denote the additional shaded boxes.
\begin{figure}[h]
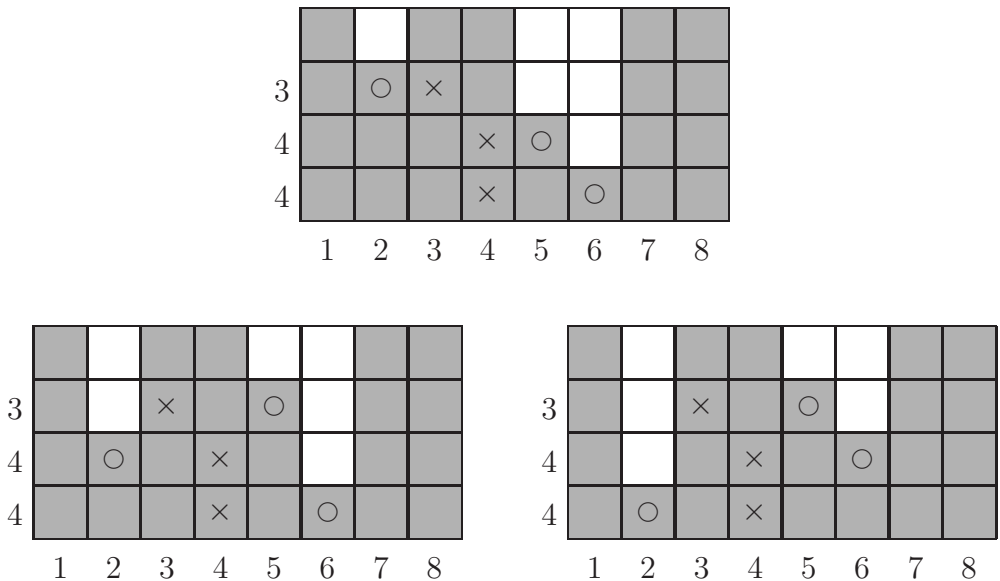

\begin{center}

\end{center}
\caption{The left-right diagrams associated with $(c_1,\ldots,c_8)=(1,0,2,3,0,0,1,1)$.}
\label{picture:left-right_diagrams_TypeA}
\end{figure}
\end{example}

For a left-right diagram $P \in \Delta_{c_1,\ldots, c_{n-1}}$, we define its weight as follows.
We first assign a rational number to each row of $P$ (except for the top row) by the following rule.
Consider the consecutive string of the shaded boxes which contains the marked box in the row under consideration. Then the set of the column numbers for these boxes must be of the form $\{a,a+1,\ldots,b\}$ for some $a,b\in [n-1]$ with $a \leq b$, and the column number $i$ of the marked box satisfies $a\le i\le b$.
To this row, we assign a rational number defined by 
\begin{align*}
&\frac{\b-i+1}{\b-\a+2} \ \ \  \textrm{if the additional shaded box is to the left of the marked box}, \\
&\frac{i-\a+1}{\b-\a+2} \ \ \  \textrm{if the additional shaded box is to the right of the marked box}.
\end{align*}
We can pictorially interpret this rational number as follows.
\begin{enumerate}
\item[$\bullet$] 
The denominator is computed as follows. 
Let $k$ be the number of boxes in the consecutive string of shaded boxes containing the marked box, counted before adding the additional shaded box. Then the denominator is $k+1$.
\item[$\bullet$] 
The numerator is computed as follows. 
Suppose the additional shaded box is to the right of the marked box. 
Then count the number of boxes in the maximal consecutive string contained the marked box which are to the left of, or equal to, the marked box.
If the additional shaded box is to the left of the marked box, then do the same except replace ``left'' with ``right''.
\end{enumerate}
We define the \emph{weight} $\wt(P)$ of a left-right diagram $P$ as the product of these rational numbers assigned to rows (except for the first row). 

\begin{example} \label{ex:diagram_weight_TypeA}
Continuing with Example~$\ref{ex:diagram_TypeA}$, the weights of left-right diagrams associated with $(c_1,\ldots,c_8)=(1,0,2,3,0,0,1,1)$ are computed as follows:
\begin{align*}
\wt(P_1)=\frac{2}{3}\cdot\frac{4}{5}\cdot\frac{4}{6}, \ \ \wt(P_2)=\frac{1}{3}\cdot\frac{2}{4}\cdot\frac{4}{6}, \ \ \wt(P_3)=\frac{1}{3}\cdot\frac{2}{4}\cdot\frac{5}{7}, 
\end{align*}
where $P_1$, $P_2$, and $P_3$ are shown in Figure~$\ref{picture:left-right_diagrams_weight_TypeA}$.

\begin{figure}[h]
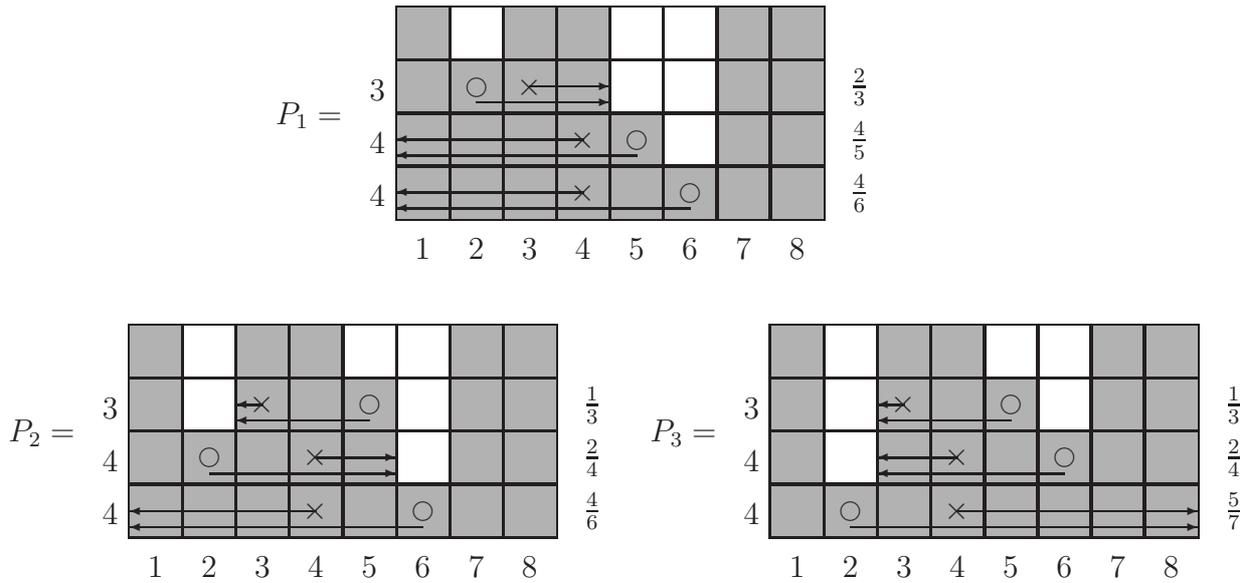

\begin{center}

\end{center}
\caption{The weights of left-right diagrams associated with $(c_1,\ldots,c_8)=(1,0,2,3,0,0,1,1)$.}
\label{picture:left-right_diagrams_weight_TypeA}
\end{figure}
\end{example}

We now summarize the computation for mixed Eulerian numbers in terms of left-right diagrams. 
The proof is similar to that of \cite[Theorem~5.6]{AHKZ}, but we give a proof for the reader.

\begin{theorem} \label{theorem:mainTypeA}
Let $c_1,\ldots, c_{n-1}$ be non-negative integers with $c_1 +\cdots+c_{n-1} = n-1$.
Let $\Delta_{c_1,\ldots, c_{n-1}}$ be the set of left-right diagrams associated with $(c_1,\ldots, c_{n-1})$.
Then, the mixed Eulerian number $A_{c_1,\ldots,c_{n-1}}$ is equal to
\begin{align*} 
A_{c_1,\ldots,c_{n-1}}=(n-1)! \sum_{P \in \Delta_{c_1,\ldots,c_{n-1}}} \wt(P).
\end{align*}
\end{theorem}

\begin{proof}
By Proposition~\ref{prop:mixed Eulerian numberTypeA Pet} the mixed Eulerian number $A_{c_1,\ldots,c_{n-1}}$ is equal to
\begin{align*} 
\int_{\Pet_n} \varpi_1^{c_1}\varpi_2^{c_2}\cdots\varpi_{n-1}^{c_{n-1}}  
= \int_{\Pet_n} \varpi_{i_1}\varpi_{i_2}\cdots\varpi_{i_{n-1}}  
= \int_{\Pet_n} \left(\prod_{i \in I} \varpi_i \right) \cdot \left( \prod_{j \in J} \varpi_j \right), 
\end{align*}
where $I$ and $J$ are defined in \eqref{eq:diagram_J} and \eqref{eq:diagram_I}.
We compute the product in the right hand side of this equality. For this purpose, take the decomposition $J= J_1 \sqcup \cdots \sqcup J_s$ into the maximal consecutive strings. 
Let $i$ be a smallest element of $I$. 
Then we have $i\in J_r$ for some $r$ ($1\le r\le s$). 
Since $J_r$ is consecutive, one can express it as $J_r=\{a,a+1,\ldots,b\}$ for some $a,b\in J$ with $a\leq i\leq b$. 
It then follows from Lemma~\ref{lemm:TypeA} that 
\begin{align*}
\varpi_i\cdot \left( \prod_{j\in J} \varpi_j \right)=
\frac{\b-i+1}{\b-\a+2} \prod_{j\in J \cup \{a-1\} } \varpi_j + \frac{i-\a+1}{\b-\a+2} \prod_{j\in J \cup \{b+1\} } \varpi_j,
\end{align*}
where we have no squares of $\varpi_j$'s in the right hand side since $J_r$ is a maximal consecutive string of $J$. 
Next, take a smallest element $i'$ of $I \setminus\{i\}$. Multiplying $\varpi_{i'}$ to the right hand side of the equality above, we can expand it by square-free monomials in $\varpi_1,\ldots,\varpi_{n-1}$ by Lemma~\ref{lemm:TypeA} again.
Repeating this procedure for each element of $I$ in weakly increasing order, we obtain
\begin{align*} 
\left(\prod_{i \in I} \varpi_i \right) \cdot \left( \prod_{j \in J} \varpi_j \right) = \left(\sum_{P \in \Delta_{c_1,\ldots,c_{n-1}}} \wt(P) \right) \varpi_1 \cdots \varpi_{n-1}
\end{align*}
by the definition of the left-right diagrams and their weights. 
Taking the integral over $\Pet_n$ for this equality, we obtain from Proposition~\ref{prop:Top_Class_TypeA} that 
\begin{align*} 
A_{c_1,\ldots,c_{n-1}}=(n-1)! \sum_{P \in \Delta_{c_1,\ldots,c_{n-1}}} \wt(P),
\end{align*}
as desired.
\end{proof}

\begin{example} \label{ex:Mixed_Eulerian_numbers_diagrams_TypeA}
Let $(c_1,\ldots,c_8)=(1,0,2,3,0,0,1,1)$ as in Example~$\ref{ex:calculus_TypeA}$. 
The weights of left-right diagrams associated with $(c_1,\ldots,c_8)$ are computed in Example~$\ref{ex:diagram_weight_TypeA}$.
Hence, the mixed Eulerian number $A_{c_1,\ldots,c_{8}}$ is equal to
\begin{align*}
8! \left( \frac{2}{3}\cdot\frac{4}{5}\cdot\frac{4}{6}+\frac{1}{3}\cdot\frac{2}{4}\cdot\frac{4}{6}+\frac{1}{3}\cdot\frac{2}{4}\cdot\frac{5}{7}\right) =14336+4480+4800=23616
\end{align*}
by Theorem~$\ref{theorem:mainTypeA}$.
Note that we have already seen the result in Example~$\ref{ex:calculus_TypeA}$.
\end{example}

\bigskip

%%%%%%%%%%%%%%%%%%%%%%%%%%%%%%%%%%
\section{Mixed $\Phi$-Eulerian numbers} \label{sect:Any_Lie_types}
%%%%%%%%%%%%%%%%%%%%%%%%%%%%%%%%%%

Postnikov introduced in \cite{Pos} the mixed Eulerian numbers for arbitrary Lie types which are called the mixed $\Phi$-Eulerian numbers.
In this section we extend results in previous sections to arbitrary Lie types. 

Recall from the introduction that $\Phi$ is a crystallographic root system of rank $n$ and $\Lambda$ is the associated integer weight lattice.
The Weyl group $W$ acts on the weight space $\Lambda_{\R} \coloneqq \Lambda \otimes \R$.
We denote by $( \ , \ )$ a $W$-invariant inner product on $\Lambda_{\R}$. 
For a point $\chi \in \Lambda_{\R}$, the \emph{weight polytope} $P_\Phi(\chi)$ is defined as
\begin{align*}
P_\Phi(\chi):=\Conv\{ w(\chi) \in \Lambda_{\R} \mid w \in W \}.
\end{align*}
Let $\Sigma \coloneqq \{\alpha_1, \ldots , \alpha_n\} \subset \Phi$ be a set of simple roots.
A notation $\Vol$ denotes the volume form on $\Lambda_{\R}$ normalized so that the volume of the parallelepiped generated by the simple roots $\alpha_1, \ldots , \alpha_n$ is equal to $1$.
Recall that a weight $\chi \in \Lambda_{\R}$ is called regular if $(\chi, \alpha) \neq 0$ for all $\alpha \in \Phi$. 
A weight $\chi$ is called dominant if $(\chi,\alpha_i) \geq 0$ for $i=1,\ldots, n$. 
Without loss of generality, we may assume that $\chi$ is dominant for the weight polytope $P_{\Phi}(\chi)$.

\begin{theorem}$($\cite[Theorem~4.2]{Pos}$)$
Let $t \in \Lambda_{\R}$ be a regular weight. 
The volume of the weight polytope $P_\Phi(\chi)$ is equal to 
\begin{align} \label{eq:Vol_weight_polytope}
\Vol P_\Phi(\chi) = \frac{1}{n!} \sum_{w \in W} \frac{\big(t,w(\chi)\big)^n}{\big(t,w(\alpha_1)\big) \cdots \big(t,w(\alpha_n)\big)}.
\end{align}
\end{theorem}
Take an orthonormal basis $e_1,\ldots,e_n$ for $\Lambda_{\R}$. 
When we write $t \in \Lambda_{\R}$ as $t=t_1e_1+\cdots+t_ne_n$ for some $t_1,\ldots,t_n \in \R$, the right hand side in \eqref{eq:Vol_weight_polytope} can be expressed as
\begin{align*}
\frac{1}{n!} \sum_{w \in W} \frac{\Big( \big(e_1,w(\chi)\big)t_1+\cdots+\big(e_n,w(\chi)\big)t_n \Big)^n}{\Big( \big(e_1,w(\alpha_1)\big)t_1+\cdots+\big(e_n,w(\alpha_1)\big)t_n \Big) \cdots \Big( \big(e_1,w(\alpha_n)\big)t_1+\cdots+\big(e_n,w(\alpha_n)\big)t_n \Big)}.
\end{align*}
As in the case of type $A$, we may regard $t_1,\ldots,t_n$ as variables (see \eqref{eq:Vol_Div_Sym}).
With this in mind, we see a topological interpretation for the volume of the weight polytope.

Let $G$ be a simply connected semisimple algebraic group over $\C$ and $B$ a fixed Borel subgroup of $G$.
Take a maximal torus $T$ in $B$.
We denote the Lie algebras of $G$, $B$, and $T$ by $\mathfrak{g}$, $\mathfrak{b}$, and $\mathfrak{t}$, respectively.
The root space associated with a root $\alpha$ is denoted by $\mathfrak{g}_{\alpha}$.

Let $S \in \mathfrak{g}$ be a regular semisimple element sitting in $\mathfrak{t}$. 
Then, the subvariety 
\begin{align} \label{eq:Weight_in_Flag}
X_{\Phi} \coloneqq \left\{gB \in G/B \mid \mbox{Ad}(g^{-1})(S) \in \mathfrak{b} \oplus \bigoplus_{i=1}^{n} \mathfrak{g}_{-\alpha_i} \right\}
\end{align}
admits an $T$-action which is induced from the natural $T$-action on the flag variety $G/B$ (\cite[Proposition~2]{dMPS}).
It is known that $\X_{\Phi}$ is a toric variety which corresponds to the normal fan of the weight polytope $P_\Phi(\chi)$ associated to a regular weight $\chi$ (\cite[Theorem~11]{dMPS}).
The toric variety $\X_{\Phi}$ is smooth and its complex dimension is $n=\rank(\Phi)$ (\cite[Theorem~6 and Corollary~9]{dMPS}).

To each weight $\chi$ we can assign a line bundle\footnote{We take the $B$-action on the product $G \times \C_\chi$ such that $[g, z]=[gb,b^{-1}z]$ for $g \in G, z \in \C_\chi, b \in B$ in the quotient.} $L_{\chi}=G \times_B \C_{\chi}$ over the flag variety $G/B$ where $\C_{\chi}$ is the one-dimensional $B$-representation via $\chi:T \to \C^{\times}$ by composing with the canonical projection $B \twoheadrightarrow T$.
We also denote the restriction $L_{\chi}^*|_{X_{\Phi}}$ by the same symbol $L_{\chi}^*$ when there are no confusion.
By abuse of notation, for the fundamental weight $\varpi_i \in \Lambda \ (1 \leq i \leq n)$, we denote by 
\begin{align} 
&\varpi_i^T \coloneqq c_1^T(L_{\varpi_i}^*)=-c_1^T(L_{\varpi_i}), \label{eq:varpiTanytype} \\
&\varpi_i \coloneqq c_1(L_{\varpi_i}^*)=-c_1(L_{\varpi_i}), \label{eq:varpianytype}
\end{align}
the $T$-equivariant (or ordinary) first Chern class of the line bundle $L_{\varpi_i}^*$. 

\begin{remark} \label{rem:notation_varpi}
As is well-known, assigning the first Chern class $c_1(L_\chi^*)$ to $\chi$ induces an isomorphism $\Lambda_{\R} \cong H^2(G/B)$. 
For this reason, we use the same notation $\varpi_i$ for $c_1(L_{\varpi_i}^*) \in H^2(G/B)$ and $c_1(L_{\varpi_i}^*|_{X_\Phi}) \in H^2(X_{\Phi})$ in \eqref{eq:varpianytype}.
\end{remark}

It is known that the correspondence $\chi \mapsto c_1^T(\C_{\chi})$ gives an isomorphism from $\Lambda$ to $H^2_T(\pt;\Z)$.
Since $H^*_T(\pt;\Z)$ is a polynomial ring over $\Z$ in $n$ variables of degree $2$, this isomorphism extends to the ring isomorphism $\Sym \Lambda_\R \cong H^*_T(\pt)$ where $\Sym \Lambda_\R$ denotes the symmetric algebra of $\Lambda_\R$. 
Fix an orthonormal basis $e_1,\ldots,e_n$ for $\Lambda_\R$.
Let $t_i \in H^2_T(\pt) \ (1 \leq i \leq n)$ be the image of $-e_i$ under the isomorphism $\Sym \Lambda_\R \cong H^*_T(\pt)$ above.
In this setting we have $H^*_T(\pt) = \R[t_1, \ldots , t_n]$.
It is known that the $T$-fixed point set $X_{\Phi}^T$ consists of $(G/B)^T$ which is identified with the Weyl group $W$ (\cite[Proposition~3]{dMPS}).
Since the restriction map $\iota: H^*_T(X_\Phi) \to H^*_T((X_\Phi)^T)=\bigoplus_{w \in W} \R[t_1,\ldots,t_n]$ is injective, the equivariant cohomology $H^*_T(X_\Phi)$ can be regarded as a subset of $\bigoplus_{w \in W} \R[t_1,\ldots,t_n]$.
Then, one can easily see that the $w$-th component of the equivariant first Chern class $c_1^T(L_\chi^*) \in H^2_T(X_\Phi)$ is given by 
\begin{align*}
c_1^T(L_\chi^*)|_w=\big(e_1,w(\chi)\big)t_1+\cdots+\big(e_n,w(\chi)\big)t_n 
\end{align*}
for $w \in W$ (e.g. \cite[Lemma~5.2]{AHMMS}).
The following lemma can be proved by using the Atiyah--Bott--Berline--Vergne formula  (\cite{AtBo, BeVe}), as in the proof of Lemma~\ref{lemma:Topology_divided_symmetrization_typeA}.

\begin{lemma} \label{lemma:Topology_divided_symmetrization}
The volume of the weight polytope $P_\Phi(\chi)$ (for a dominant weight $\chi$) is 
\begin{align*} 
\Vol P_\Phi(\chi) = \frac{1}{n!} \int_{X_\Phi} c_1(L_\chi^*)^n.
\end{align*}
\end{lemma}

We now explain the mixed $\Phi$-Eulerian numbers.
Let $\varpi_1,\ldots,\varpi_n$ be the fundamental weights.
We take $\chi=u_1\varpi_1+\cdots+u_n\varpi_n$ and consider the associated weight polytope $P_\Phi(\chi)$.
Its volume is a homogeneous polynomial $V_\Phi$ of degree $n$ in the variables $u_1,\ldots,u_n$:
\begin{align*}
V_\Phi(u_1,\ldots,u_n) \coloneqq \Vol P_\Phi(u_1\varpi_1+\cdots+u_n\varpi_n).
\end{align*}

The \emph{mixed $\Phi$-Eulerian numbers} $A^{\Phi}_{c_1,\ldots,c_n}$, for $c_1, \ldots , c_n \geq 0$ with $c_1 +\cdots+c_n =n$, are defined in \cite[Definition~18.4]{Pos} as the coefficients of the polynomial 
\begin{align*} 
V_\Phi(u_1,\ldots,u_n) = \sum_{c_1,\ldots,c_n} A^{\Phi}_{c_1,\ldots,c_n} \frac{u_1^{c_1}}{c_1!} \cdots \frac{u_{n}^{c_{n}}}{c_{n}!}.
\end{align*}
The mixed $\Phi$-Eulerian number $A^{\Phi}_{c_1,\ldots,c_n}$ means the mixed volume of $c_1$ copies of $P_\Phi(\varpi_1)$, $c_2$ copies of $P_\Phi(\varpi_2)$, $\cdots$, and $c_n$ copies of $P_\Phi(\varpi_n)$, multiplied by $n!$.
Here, the weight polytopes $P_\Phi(\varpi_1), P_\Phi(\varpi_2), \ldots, P_\Phi(\varpi_n)$ are called the \emph{$\Phi$-hypersimplices}.

It follows from Lemma~\ref{lemma:Topology_divided_symmetrization} that
\begin{align*} 
\Vol P_\Phi(u_1\varpi_1+\cdots+u_n\varpi_n) =& \frac{1}{n!} \int_{X_{\Phi}} (u_1\varpi_1+\cdots+u_n\varpi_n)^n \\
=& \sum_{c_1,\ldots,c_n} \left( \int_{X_{\Phi}} \varpi_1^{c_1}\varpi_2^{c_2}\cdots\varpi_{n}^{c_n} \right) \frac{u_1^{c_1}}{c_1!} \cdots \frac{u_{n}^{c_{n}}}{c_{n}!},
\end{align*}
where the sum is over non-negative integers $c_1, \ldots , c_n$ with $c_1 +\cdots+c_n = n$.
As a consequence, we obtain the following result.

\begin{proposition} \label{proposition:mixed Eulerian number}
Let $c_1,\ldots, c_n$ be non-negative integers with $c_1 +\cdots+c_n = n$.
Then, the mixed $\Phi$-Eulerian number $A^{\Phi}_{c_1,\ldots,c_n}$ is equal to
\begin{align*} 
A^{\Phi}_{c_1,\ldots,c_n} = \int_{X_{\Phi}} \varpi_1^{c_1}\varpi_2^{c_2}\cdots\varpi_{n}^{c_n},
\end{align*}
where $\varpi_i$ is defined in \eqref{eq:varpianytype}.
\end{proposition}

By Proposition~\ref{proposition:mixed Eulerian number} we have 
\begin{align*} 
A^{\Phi}_{1,\ldots,1} = \int_{X_{\Phi}} \varpi_1\varpi_2\cdots\varpi_{n}.
\end{align*}
However, it is known that the left hand side is equal to $\frac{|W|}{\det (C_{\Phi})}$ by \cite[Proposition~2.7.5]{Cro} where $C_{\Phi}$ denotes the associated Cartan matrix.
Note that this number can be explicitly described below.
\begin{table}[htb]
  \begin{tabular}{|c|c|} \hline
    $\Phi$ & $\frac{|W|}{\det (C_{\Phi})}$ \\ \hline \hline
    $A_n$ & $n!$ \\ \hline
    $B_n,C_n$ & $2^{n-1} \cdot n!$ \\ \hline
    $D_n$ & $2^{n-3} \cdot n!$ \\ \hline
    $E_6$ & $2^{7} \cdot 3^{3} \cdot 5$ \\ \hline
    $E_7$ & $2^{9} \cdot 3^{4} \cdot 5 \cdot 7$ \\ \hline
    $E_8$ & $2^{14} \cdot 3^{5} \cdot 5^{2} \cdot 7$ \\ \hline
    $F_4$ & $2^{7} \cdot 3^{2}$ \\ \hline
    $G_2$ &  $2^{2} \cdot 3$ \\ \hline
  \end{tabular}
  \caption{A list of values of $\frac{|W|}{\det (C_{\Phi})}$.}
\label{tab:A list of values mPhi}
\end{table}

Hence, we have the following proposition. 

\begin{proposition} \label{prop:Top_Class}
We have
\begin{align*} 
\int_{X_{\Phi}} \varpi_1\varpi_2 \cdots \varpi_n =\frac{|W|}{\det (C_{\Phi})}.
\end{align*}
\end{proposition}

\begin{remark} \label{remark:Klyachko}
Proposition~\ref{prop:Top_Class} is originally the result of Klyachko (\cite[Theorem~3]{Kly}). 
Klyachko's proof was written in \cite{Kly2} in Russian and it is probably a different argument from ours.
That is, one may prove Proposition~\ref{prop:Top_Class} without going through the mixed $\Phi$-Eulerian numbers. 
We will give an alternative proof (without going through the mixed Eulerian numbers) by using equivariant cohomology of $X_{\Phi}$ in Appendix~\ref{sect: Integration}.
\end{remark}

In the next section we will derive a simple computaion for the mixed Eulerian numbers $A^{\Phi}_{c_1,\ldots,c_n}$ from Peterson Schubert calculus.

\bigskip

%%%%%%%%%%%%%%%%%%%%%%%%%%%%%
\section{Computation for mixed $\Phi$-Eulerian numbers and Peterson Schubert calculus} \label{sect: Mixed Eulerian numbers and Peterson Schubert calculus}
%%%%%%%%%%%%%%%%%%%%%%%%%%%%%

In this section we explain a connection between mixed $\Phi$-Eulerian numbers and Peterson Schubert calculus in all Lie types.
As a consequence, we derive a simple computation for the mixed $\Phi$-Eulerian numbers by using results of Drellich in \cite{Dre2}.
In what follows, we may assume that $\Phi$ is an irreducible root system.

Fix a simple system $\Sigma=\{\alpha_1,\ldots,\alpha_n \}$ of $\Phi$.
Recall that the set of vertices of the Dynkin diagram of $\Phi$ is in one-to-one correspondence with $\Sigma$. 
Here and below, we assume to be fixed an ordering of the simple roots as given in \cite[p.58]{Hum1}. (See Figure~\ref{pic: Dynkin diagrams}.)

\begin{figure}[h]
\setlength{\unitlength}{1mm}
\begin{center} 
  \begin{picture}(50, 15)(0,0)
  \put(-25,9){Type $A_n$:}
  \put(0,10){\circle{2}}
  \put(10,10){\circle{2}}
  \put(20,10){\circle{2}}
  \put(50,10){\circle{2}}
  \put(60,10){\circle{2}}
  \put(1,10){\line(1,0){8}}
  \put(11,10){\line(1,0){8}}
  \put(21,10){\line(1,0){8}}
  \put(33,8.8){$\cdots$}
  \put(41,10){\line(1,0){8}}
  \put(51,10){\line(1,0){8}}
  \put(-2,3){$\alpha_1$}
  \put(8,3){$\alpha_2$}
  \put(18,3){$\alpha_3$}
  \put(48,3){$\alpha_{n-1}$}
  \put(58,3){$\alpha_n$}
  \end{picture}
\end{center}  

\setlength{\unitlength}{1mm}
\begin{center}
  \begin{picture}(50, 15)(0,0)
  \put(-25,9){Type $B_n$:}
  \put(0,10){\circle{2}}
  \put(10,10){\circle{2}}
  \put(20,10){\circle{2}}
  \put(50,10){\circle{2}}
  \put(60,10){\circle{2}}
  \put(1,10){\line(1,0){8}}
  \put(11,10){\line(1,0){8}}
  \put(21,10){\line(1,0){8}}
  \put(33,8.8){$\cdots$}
  \put(41,10){\line(1,0){8}}
  \put(51,10.5){\line(1,0){8}}
  \put(51,9.5){\line(1,0){8}}
  \put(-2,3){$\alpha_1$}
  \put(8,3){$\alpha_2$}
  \put(18,3){$\alpha_3$}
  \put(48,3){$\alpha_{n-1}$}
  \put(58,3){$\alpha_n$}
  \qbezier(54.5,11.5)(57.5,10)(54.5,8.5)
  \end{picture}
\end{center}  

\setlength{\unitlength}{1mm}
\begin{center}
  \begin{picture}(50, 15)(0,0)
  \put(-25,9){Type $C_n$:}
  \put(0,10){\circle{2}}
  \put(10,10){\circle{2}}
  \put(20,10){\circle{2}}
  \put(50,10){\circle{2}}
  \put(60,10){\circle{2}}
  \put(1,10){\line(1,0){8}}
  \put(11,10){\line(1,0){8}}
  \put(21,10){\line(1,0){8}}
  \put(33,8.8){$\cdots$}
  \put(41,10){\line(1,0){8}}
  \put(51,10.5){\line(1,0){8}}
  \put(51,9.5){\line(1,0){8}}
  \put(-2,3){$\alpha_1$}
  \put(8,3){$\alpha_2$}
  \put(18,3){$\alpha_3$}
  \put(48,3){$\alpha_{n-1}$}
  \put(58,3){$\alpha_n$}
  \qbezier(55.5,11.5)(52.5,10)(55.5,8.5)
  \end{picture}
\end{center}  

\setlength{\unitlength}{1mm}
\begin{center}
  \begin{picture}(50, 20)(0,0)
  \put(-25,9){Type $D_n$:}
  \put(0,10){\circle{2}}
  \put(10,10){\circle{2}}
  \put(20,10){\circle{2}}
  \put(50,10){\circle{2}}
  \put(60,4){\circle{2}}
  \put(60,16){\circle{2}}
  \put(1,10){\line(1,0){8}}
  \put(11,10){\line(1,0){8}}
  \put(21,10){\line(1,0){8}}
  \put(33,8.8){$\cdots$}
  \put(41,10){\line(1,0){8}}
  \put(51,10.5){\line(3,2){8}}
  \put(51,9.5){\line(3,-2){8}}
  \put(-2,3){$\alpha_1$}
  \put(8,3){$\alpha_2$}
  \put(18,3){$\alpha_3$}
  \put(48,3){$\alpha_{n-2}$}
  \put(63,15){$\alpha_{n-1}$}
  \put(63,3){$\alpha_n$}
  \end{picture}
\end{center}  

\setlength{\unitlength}{1mm}
\begin{center}
  \begin{picture}(50, 25)(0,0)
  \put(-25,9){Type $E_6$:}
  \put(0,10){\circle{2}}
  \put(10,10){\circle{2}}
  \put(20,10){\circle{2}}
  \put(30,10){\circle{2}}
  \put(40,10){\circle{2}}
  \put(20,20){\circle{2}}
  \put(1,10){\line(1,0){8}}
  \put(11,10){\line(1,0){8}}
  \put(21,10){\line(1,0){8}}
  \put(31,10){\line(1,0){8}}
  \put(20,11){\line(0,1){8}}
  \put(-2,3){$\alpha_1$}
  \put(8,3){$\alpha_3$}
  \put(18,3){$\alpha_4$}
  \put(28,3){$\alpha_5$}
  \put(38,3){$\alpha_6$}
  \put(22,19){$\alpha_2$}
  \end{picture}
\end{center}  

\setlength{\unitlength}{1mm}
\begin{center}
  \begin{picture}(50, 25)(0,0)
  \put(-25,9){Type $E_7$:}
  \put(0,10){\circle{2}}
  \put(10,10){\circle{2}}
  \put(20,10){\circle{2}}
  \put(30,10){\circle{2}}
  \put(40,10){\circle{2}}
  \put(50,10){\circle{2}}
  \put(20,20){\circle{2}}
  \put(1,10){\line(1,0){8}}
  \put(11,10){\line(1,0){8}}
  \put(21,10){\line(1,0){8}}
  \put(31,10){\line(1,0){8}}
  \put(41,10){\line(1,0){8}}
  \put(20,11){\line(0,1){8}}
  \put(-2,3){$\alpha_1$}
  \put(8,3){$\alpha_3$}
  \put(18,3){$\alpha_4$}
  \put(28,3){$\alpha_5$}
  \put(38,3){$\alpha_6$}
  \put(48,3){$\alpha_7$}
  \put(22,19){$\alpha_2$}
  \end{picture}
\end{center}  

\setlength{\unitlength}{1mm}
\begin{center}
  \begin{picture}(50, 25)(0,0)
  \put(-25,9){Type $E_8$:}
  \put(0,10){\circle{2}}
  \put(10,10){\circle{2}}
  \put(20,10){\circle{2}}
  \put(30,10){\circle{2}}
  \put(40,10){\circle{2}}
  \put(50,10){\circle{2}}
  \put(60,10){\circle{2}}
  \put(20,20){\circle{2}}
  \put(1,10){\line(1,0){8}}
  \put(11,10){\line(1,0){8}}
  \put(21,10){\line(1,0){8}}
  \put(31,10){\line(1,0){8}}
  \put(41,10){\line(1,0){8}}
  \put(51,10){\line(1,0){8}}
  \put(20,11){\line(0,1){8}}
  \put(-2,3){$\alpha_1$}
  \put(8,3){$\alpha_3$}
  \put(18,3){$\alpha_4$}
  \put(28,3){$\alpha_5$}
  \put(38,3){$\alpha_6$}
  \put(48,3){$\alpha_7$}
  \put(58,3){$\alpha_8$}
  \put(22,19){$\alpha_2$}
  \end{picture}
\end{center}  

\setlength{\unitlength}{1mm}
\begin{center}
  \begin{picture}(30, 15)(0,0)
  \put(-35,9){Type $F_4$:}
  \put(0,10){\circle{2}}
  \put(10,10){\circle{2}}
  \put(20,10){\circle{2}}
  \put(30,10){\circle{2}}
  \put(1,10){\line(1,0){8}}
  \put(11,10.5){\line(1,0){8}}
  \put(11,9.5){\line(1,0){8}}
  \put(21,10){\line(1,0){8}}
  \put(-2,3){$\alpha_1$}
  \put(8,3){$\alpha_2$}
  \put(18,3){$\alpha_3$}
  \put(28,3){$\alpha_4$}
  \qbezier(14.5,11.5)(17.5,10)(14.5,8.5)
  \end{picture}
\end{center}  

\setlength{\unitlength}{1mm}
\begin{center}
  \begin{picture}(10, 15)(0,0)
  \put(-45,9){Type $G_2$:}
  \put(0,10){\circle{2}}
  \put(10,10){\circle{2}}
  \put(1,10){\line(1,0){8}}
  \put(1,10.7){\line(1,0){8}}
  \put(1,9.3){\line(1,0){8}}
  \put(-2,3){$\alpha_1$}
  \put(8,3){$\alpha_2$}
  \qbezier(5.5,11.5)(2.5,10)(5.5,8.5)
  \end{picture}
\end{center}  
\caption{The Dynkin diagrams.}
\label{pic: Dynkin diagrams}
\end{figure}
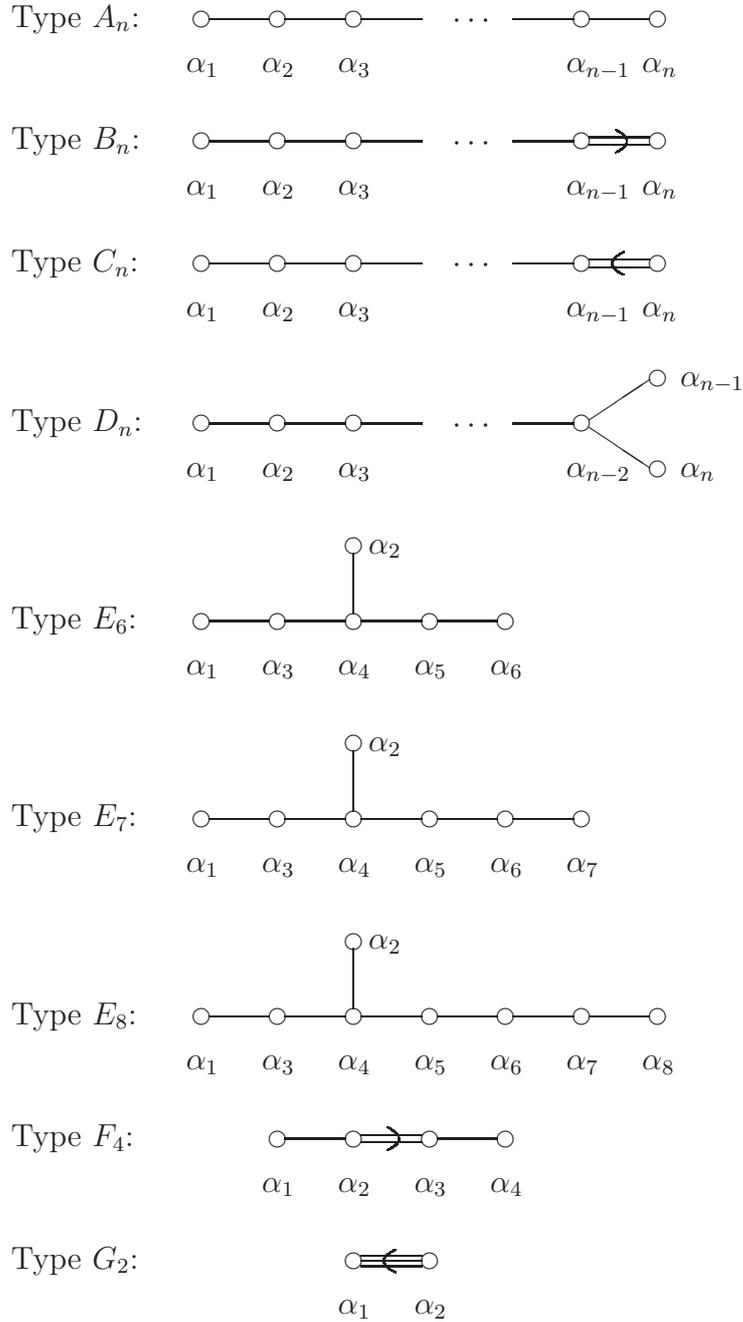

\begin{definition} $($\cite[Definition~3.2]{Dre2}$)$
A subset of simple roots $K \subset \Sigma$ is called \emph{connected} if the induced
Dynkin diagram with the vertices $K$ is a connected subgraph of the Dynkin diagram of $\Phi$.
\end{definition}

Note that a connected subset $K \subset \Sigma$ has its own Lie type.
We denote by $\Phi_K$ the root subsystem associated with $K$. 
In other words, $K$ is the simple system of $\Phi_K$.

\begin{example}
When $\Phi$ is of type $E_6$, an ordering of the simple roots $\alpha_1,\ldots,\alpha_6$ is described in Figure~\ref{pic: Dynkin diagrams}.
For instance, if we take $K=\{\alpha_1,\alpha_2,\alpha_3,\alpha_4,\alpha_5 \}$, then $\Phi_K$ is of type $D_5$.
If we take $K'=\{\alpha_1,\alpha_3,\alpha_4,\alpha_5,\alpha_6 \}$, then $\Phi_{K'}$ is of type $A_5$.
\end{example}

Every subset $K \subset \Sigma$ can be written as $K = K_1 \sqcup \cdots \sqcup K_m$ where each $K_i$ is a maximal connected subset. 
Note that each connected subset corresponds to its own Lie type.

\begin{definition} $($\cite[Definition~3.3]{Dre2}$)$ \label{def:vK}
For a subset $K\subset \Sigma$, we define $v_K \in W$ as follows.
\begin{enumerate}
\item For a connected subset $K \subset \Sigma$, we fix an ordering of the simple roots in $\Phi_K$ as in Figure~\ref{pic: Dynkin diagrams}.
Then we define $v_K \in W$ as 
\begin{align*}
v_K \coloneqq \prod_{\Root_K(i)=1}^{|K|} s_i 
\end{align*}
where $\Root_K(i)$ is the fixed index of the corresponding root in the root system $\Phi_K$. 
\item If a subset $K\subset \Sigma$ can be decomposed into connected components as $K = K_1 \sqcup \cdots \sqcup K_m$, then we define $v_K \coloneqq v_{K_1} v_{K_2} \cdots v_{K_m}$.
\end{enumerate}
\end{definition}

\begin{example}
Let $K=\{\alpha_1,\alpha_2,\alpha_3,\alpha_4,\alpha_5 \}$ be a subset of the simple roots of $\Phi_{E_6}$ as in the previous example. 
Noting that $\Phi_K$ is of type $D_5$, fix an ordering of elements of $K$ by $\Root_K(1)=1, \Root_K(3)=2, \Root_K(4)=3, \Root_K(5)=4, \Root_K(2)=5$. 
Then we have $v_K=s_1s_3s_4s_5s_2$.
If we take $K'=\{\alpha_1,\alpha_3,\alpha_4,\alpha_5,\alpha_6 \}$ in $\Phi_{E_6}$, then $\Phi_{K'}$ is of type $A_5$. When we fix an ordering of elements of $K'$ by $\Root_{K'}(1)=1, \Root_{K'}(3)=2, \Root_{K'}(4)=3, \Root_{K'}(5)=4, \Root_{K'}(6)=5$, one has $v_{K'}=s_1s_3s_4s_5s_6$.
\end{example}

These elements $v_K$ play an important role in Peterson Schubert calculus as described below.

\begin{remark}
In the previous example, if we change an ordering of elements of $K'$ as $\Root_{K'}(6)=1, \Root_{K'}(5)=2, \Root_{K'}(4)=3, \Root_{K'}(3)=4, \Root_{K'}(1)=5$, then we have $v_{K'}=s_6s_5s_4s_3s_1$ which is different than the element $s_1s_3s_4s_5s_6$. 
However, we will see that the Peterson Schubert classes associated to these elements are the same (Remark~\ref{remark:GiambelliPetGeneralization}). 
\end{remark}

Take a regular nilpotent element $N \in \mathfrak{g}$. 
We may assume $N$ is in the standard form, i.e. $N=\sum_{i=1}^n E_{\alpha_i}$ where $E_{\alpha_i}$ is a basis of $\mathfrak{g}_{\alpha_i}$. 
We have already observed in the introduction that the \emph{Peterson variety} $\Pet_{\Phi}$ is defined by
\begin{align*}
\Pet_{\Phi} \coloneqq \left\{gB \in G/B \mid \mbox{Ad}(g^{-1})(N) \in \mathfrak{b} \oplus \bigoplus_{i=1}^{n} \mathfrak{g}_{-\alpha_i} \right\},
\end{align*}
which is irreducible and its complex dimension has been explicitly expressed as $n=\rank\Phi$ (\cite[Corollary~14]{Pre18}).
Let $B^{-}=w_0 B w_0$ be the opposite Borel subgroup where $w_0$ is the longest element of $W$. 
The opposite Schubert variety associated with $w \in W$ is defined to be the closure $\overline{B^{-}wB/B}$ of $B^{-}$-orbit of $w$ in $G/B$.
We denote by $\sigma_w \in H^*(G/B)$ the Schubert class, namely the Poincar\'e dual of the opposite Schubert variety $\overline{B^{-}wB/B}$. 
Let $p_w$ be the image of $\sigma_w$ under the restriction map $H^*(G/B) \to H^*(\Pet_{\Phi})$ which is called the \emph{Peterson Schubert class}. 

\begin{theorem} $($\cite[Theorem~3.5]{Dre2}$)$ \label{Theorem_Peterson_basis}
The set $\{p_{v_K} \mid K \subset \Sigma \}$ forms a basis of $H^*(\Pet_{\Phi})$.
\end{theorem}

\begin{remark}
Theorem~\ref{Theorem_Peterson_basis} for type $A$ was proved in \cite{HaTy}. 
\end{remark}

By Theorem~\ref{Theorem_Peterson_basis} the product of two Peterson Schubert classes $p_{v_I}$ and $p_{v_K}$ can be written as a linear combination of Peterson Schubert classes $\{p_{v_J} \mid J \subset \Sigma \}$: 
\begin{align} \label{eq:Peterson_Schubert_calculus}
p_{v_I} \cdot p_{v_K} = \sum_{J \subset \Sigma} c_{IK}^J p_{v_J}. 
\end{align}
To compute the coefficients $c_{IK}^J$ is called \emph{Peterson Schubert calculus}.
There are several results on Peterson Schubert calculus (\cite{AHKZ, BaHa, Dre2, GoGo, GMS, HaTy}). 
We use results in \cite{Dre2} for the purpose of the calculation of the mixed $\Phi$-Eulerian numbers.

Fix a reduced word decomposition for $w = s_{b_1} s_{b_2} \cdots s_{b_{\ell}}$ where $\ell := \ell(w)$ is the length of $w \in W$.
For $j \leq \ell$, set $r(i,w) = s_{b_1} s_{b_2} \cdots s_{b_{i-1}}(\alpha_{b_i})$. 
Then we define
\begin{align} \label{eq:Billey}
\sigma_v(w) \coloneqq \sum_{{\rm reduced \ words} \atop v=s_{b_{j_1}}s_{b_{j_2}} \cdots s_{b_{j_{\ell(v)}}}} \prod_{i=1}^{\ell(v)} r(j_i,w).
\end{align}
Note that $\sigma_v(w)$ is a homogeneous polynomial of degree $\ell(v)$ in the variables $\alpha_1,\ldots, \alpha_{n}$.
Consider the map $\R[\alpha_1,\ldots,\alpha_n] \to \R[t]$ which sends $\alpha_i$ to $t$ for all $1 \leq i \leq n$, and define $p_v(w)$ by the image of $\sigma_v(w)$ under the map $\R[\alpha_1,\ldots,\alpha_n] \to \R[t]$. 
Note that $p_v(w)$ is a homogeneous polynomial of degree $\ell(v)$ in the variable $t$.

\begin{remark}
The polynomial $\sigma_v(w)$ is, in fact, the $w$-component of the $T$-equivariant Schubert class $\sigma_v \in H_T^{2\ell(v)}(G/B)$ and the formula in \eqref{eq:Billey} is known to be Billey's formula (\cite{Bil}).
The $T$-action on $G/B$ does not preserve $\Pet_{\Phi}$, but there is one-dimensional subtorus $S^1 \subset T$ such that $S^1$ preserves $\Pet_{\Phi}$. 
The polynomial $p_v(w)$ means the $w$-component of the $S^1$-equivariant Peterson Schubert class $p_v \in H_{S^1}^{2\ell(v)}(\Pet_{\Phi})$. 
More precisely, the definitions of $p_v(w)$ geometrically make sense for the $S^1$-fixed points $w$.
\end{remark}

For a subset $K \subset \Sigma$, define the parabolic subgroup $W_K \subset W$, namely $W_K$ is generated by $s_k$'s with $\alpha_k \in K$. 
Let $w_K$ be the longest element of $W_K$.
We use the following two results which are the special case of the results in \cite{Dre2}.

\begin{theorem} $($\cite[Theorem~4.2]{Dre2}$)$ \label{theorem:MonkPet}
Let $K \subset \Sigma$ be a connected subset.
Let $i$ be a positive integer such that $\alpha_i \in K$.
Then, the Peterson Schubert classes satisfy 
\begin{align*}
p_{s_i} \cdot p_{v_K} = \sum_{J \supset K \atop |J|=|K|+1} c_{i,K}^J p_{v_J}
\end{align*}
where the coefficients $c_{i,K}^J$ are non-negative rational numbers given by 
\begin{align} \label{eq:ciKL}
c_{i,K}^J = \big(p_{s_i}(w_J) - p_{s_i}(w_K) \big) \cdot \frac{p_{v_K}(w_J)}{p_{v_J}(w_J)}.
\end{align}
\end{theorem}

\begin{remark} \label{remark:GiambelliPet}
Under the assumption of Theorem~\ref{theorem:MonkPet}, one can see that $p_{s_i}(w_J) = p_{s_i}(w_K)$ unless $J$ is connected. 
By \eqref{eq:ciKL} we may assume that $J$ is connected in the summation.
\end{remark}

\begin{theorem} $($\cite[Theorem~5.5]{Dre2}$)$ \label{theorem:GiambelliPet}
For a connected subset $K \subset \Sigma$,
we have
\begin{align*}
\frac{|K|!}{|\Reduce(v_K)|} \, p_{v_K}=\prod_{\alpha_k \in K} p_{s_k},
\end{align*}
where $\Reduce(w)$ denotes the set of reduced words of $w \in W$.
Note that $|\Reduce(v_K)|$ is explicitly described as follows: 
\begin{equation*} 
|\Reduce(v_K)| = \begin{cases}
1 &\textrm{when } \Phi_K \textrm{ is of type } A_n, B_n, C_n, F_4 \textrm{ or } G_2; \\
2 &\textrm{when } \Phi_K \textrm{ is of type } D_n; \\
3 &\textrm{when } \Phi_K \textrm{ is of type } E_n \ (n=6,7,8). 
\end{cases}
\end{equation*}
\end{theorem}

By abuse of notation, for the fundamental weight $\varpi_i \in \Lambda \ (1 \leq i \leq n)$, we denote by the same symbol
\begin{align} \label{eq:varpianytypePet}
&\varpi_i \coloneqq c_1(L_{\varpi_i}^*|_{\Pet_{\Phi}})=-c_1(L_{\varpi_i}|_{\Pet_{\Phi}}) \in H^2(\Pet_\Phi), 
\end{align}
the first Chern class of the line bundle $L_{\varpi_i}^*|_{\Pet_{\Phi}}$ over $\Pet_\Phi$, as discussed in Remark~\ref{rem:notation_varpi}.
It is well-known that the Schubert class $\sigma_{s_i}$ is equal to $c_1(L_{\varpi_i}^*)$ in $H^2(G/B)$ (e.g. \cite{BGG}), so we have 
\begin{align} \label{eq:PetersonSchubert-varpi}
\varpi_i = p_{s_i} 
\end{align}
in $H^2(\Pet_{\Phi})$.

\begin{remark} \label{remark:GiambelliPetGeneralization}
By \cite[Theorem~4]{Kly} (see also \cite{Kly2,NT2}) the image of the Schubert class $\sigma_w \in H^{2\ell(w)}(G/B)$ under the restriction map $\iota^*: H^*(G/B) \to H^*(X_{\Phi})$ is given by 
\begin{align*}
\iota^*(\sigma_w)= \frac{1}{\ell(w)!} \sum_{i_1i_2\cdots i_{\ell(w)} \in \Reduce(w)} \varpi_{i_1}\varpi_{i_2}\cdots\varpi_{i_{\ell(w)}}.
\end{align*}
Combining this formula with \cite[Theorems~1.1 and 1.4]{AHMMS}, we obtain
\begin{align*}
p_{w}= \frac{1}{\ell(w)!} \sum_{i_1i_2\cdots i_{\ell(w)} \in \Reduce(w)} \varpi_{i_1}\varpi_{i_2}\cdots\varpi_{i_{\ell(w)}}.
\end{align*}
Noting that \eqref{eq:PetersonSchubert-varpi}, this formula gives a generalization of Theorem~\ref{theorem:GiambelliPet}.
\end{remark}

We now explain an analogue of Lemma~\ref{lemm:TypeA}.
Let $K, J \subset \Sigma$ be connected subsets with $J \supset K$ and $|J|=|K|+1$. 
Take a positive integer $i$ such that $\alpha_i \in K$.
Then, we define
\begin{align} \label{eq:miKJ}
m_{i,K}^J \coloneqq \frac{|\Reduce(v_J)|}{|\Reduce(v_K)|} \frac{1}{|J|} \, c_{i,K}^J,
\end{align}
where $c_{i,K}^J$ is given in \eqref{eq:ciKL}.

\begin{lemma} \label{lemm:computationanytype1}
For a connected subset $K \subset \Sigma$ and a positive integer $i$ with $\alpha_i \in K$, we have
\begin{align*}
\varpi_i \cdot \prod_{\alpha_k \in K} \varpi_{k} = \sum_{J \subset \Sigma: \, {\rm connected} \atop J \supset K \, {\rm and } \, |J|=|K|+1} m_{i,K}^J \prod_{\alpha_j \in J} \varpi_{j}, 
\end{align*}
where $\varpi_i$ is defined in \eqref{eq:varpianytypePet}.
\end{lemma}

\begin{proof}
We first note that \eqref{eq:PetersonSchubert-varpi}.
It then follows from Theorems~\ref{theorem:MonkPet} and \ref{theorem:GiambelliPet} (see also Remark~\ref{remark:GiambelliPet}) that
\begin{align*}
\varpi_i \cdot \prod_{\alpha_k \in K} \varpi_{k} &= p_{s_i} \cdot \prod_{\alpha_k \in K} p_{s_k} = \frac{|K|!}{|\Reduce(v_K)|} \, p_{s_i} \cdot p_{v_K} = \sum_{J \subset \Sigma: \, {\rm connected} \atop J \supset K \, {\rm and } \, |J|=|K|+1} \frac{|K|!}{|\Reduce(v_K)|} c_{i,K}^J p_{v_J} \notag \\
&= \sum_{J \subset \Sigma: \, {\rm connected} \atop J \supset K \, {\rm and } \, |J|=|K|+1} \frac{|K|!}{|\Reduce(v_K)|}\frac{|\Reduce(v_J)|}{|J|!} c_{i,K}^J \prod_{\alpha_j \in J} \varpi_{j} = \sum_{J \subset \Sigma: \, {\rm connected} \atop J \supset K \, {\rm and } \, |J|=|K|+1} m_{i,K}^J \prod_{\alpha_j \in J} \varpi_{j}, 
\end{align*}
as desired.
\end{proof}

The following lemma says that $m_{i,K}^J$ depends only on the Dynkin diagrams of $(\Phi_K,\Phi_J)$ and the vertex $\alpha_i$ of the Dynkin diagram of $\Phi_K$.

\begin{lemma} \label{lemm:computationanytype1.5}
Let $\Phi$ and $\Phi'$ be root systems with the simple systems $\Sigma=\{\alpha_1,\ldots,\alpha_n\}$ and $\Sigma'=\{\tilde\alpha_1,\ldots,\tilde\alpha_{m}\}$, respectively.
Let $K, J \subset \Sigma$ be connected subsets with $J \supset K$ and $|J|=|K|+1$. 
Also, let $K', J' \subset \Sigma'$ be connected subsets with $J' \supset K'$ and $|J'|=|K'|+1$. 
Suppose that Lie types of $\Phi_K$ and $\Phi_J$ are the same as those of $\Phi_{K'}$ and $\Phi_{J'}$, respectively.
Let $\varphi$ be an isomorpshism from the Dynkin diagram of $\Phi_J$ to that of $\Phi_{J'}$ such that $\varphi$ induces an isomorphism between the Dynkin diagrams of $\Phi_K$ and $\Phi_{K'}$ by restriction. 
Let $\tilde\alpha_{i'} \in K'$ be the image of $\alpha_i \in K$ under the isomorphism $\varphi$.
Then we have $m_{i,K}^J=m_{i',K'}^{J'}$. 
Namely, $m_{i,K}^J$ depends only on the Dynkin diagrams of $(\Phi_K,\Phi_J)$ and the vertex $\alpha_i$ of the Dynkin diagram of $\Phi_K$.
\end{lemma}

\begin{proof}
By the definition \eqref{eq:miKJ}, we have
\begin{align*}
m_{i,K}^J = \frac{|\Reduce(v_J)|}{|\Reduce(v_K)|} \frac{1}{|J|} \, c_{i,K}^J \ \ \ {\rm and} \ \ \ 
m_{i',K'}^{J'} = \frac{|\Reduce(v_{J'})|}{|\Reduce(v_{K'})|} \frac{1}{|{J'}|} \, c_{i',K'}^{J'}. 
\end{align*}
Notice that $|J|=\rank(\Phi_J)=\rank(\Phi_{J'})=|J'|$.
Recalling that we fix an ordering of the simple roots in $\Phi_K$ as in Figure~\ref{pic: Dynkin diagrams} in Definition~\ref{def:vK}, we see that $|\Reduce(v_K)|=|\Reduce(v_{K'})|$.
Similarly, one has $|\Reduce(v_J)|=|\Reduce(v_{J'})|$.
Hence, it is enough to prove that $c_{i,K}^J=c_{i',K'}^{J'}$.
By the definition of $p_v(w)$ we have $p_{s_i}(w_J)=p_{s_{i'}}(w_{J'}), p_{s_i}(w_K)=p_{s_{i'}}(w_{K'}), p_{v_K}(w_J)=p_{v_{K'}}(w_{J'})$, and $p_{v_J}(w_J)=p_{v_{J'}}(w_{J'})$.
It follows from \eqref{eq:ciKL} that $c_{i,K}^J=c_{i',K'}^{J'}$, as desired.
\end{proof}

We now make a list of values of $m_{i,K}^J$ in \eqref{eq:miKJ}.
For this purpose we first classify the Dynkin diagrams of $(\Phi_K,\Phi_J)$.

\smallskip
\smallskip

\begin{description}
\item[Type $(A_{\j-1},A_\j)$] Let $\Phi_K$ and $\Phi_J$ be root systems of type $A_{\j-1}$ and $A_\j$, respectively. 
Set $K=\{\beta_1, \ldots, \beta_{\j-1}\}$ and $J=\{\beta_1, \ldots, \beta_\j\}$ and we fix an ordering of the simple roots in $\Phi_K$ and $\Phi_J$ as follows. 
\setlength{\unitlength}{1mm}
\begin{center}
  \begin{picture}(80, 15)(0,0)
  \put(0,10){\circle{2}}
  \put(10,10){\circle{2}}
  \put(20,10){\circle{2}}
  \put(50,10){\circle{2}}
  \put(60,10){\circle{2}}
  \put(1,10){\line(1,0){8}}
  \put(11,10){\line(1,0){8}}
  \put(21,10){\line(1,0){8}}
  \put(33,8.8){$\cdots$}
  \put(41,10){\line(1,0){8}}
  \put(51,10){\line(1,0){8}}
  \put(-2,3){$\beta_1$}
  \put(8,3){$\beta_2$}
  \put(18,3){$\beta_3$}
  \put(48,3){$\beta_{\j-1}$}
  \put(58,3){$\beta_\j$}
  \end{picture}
\end{center}  
%\bigskip
\item[Type $(A_{\j-1},B_\j)$] Let $\Phi_K$ and $\Phi_J$ be root systems of type $A_{\j-1}$ and $B_\j$, respectively. 
Set $K=\{\beta_1, \ldots, \beta_{\j-1}\}$ and $J=\{\beta_1, \ldots, \beta_\j\}$ and we fix an ordering of the simple roots in $\Phi_K$ and $\Phi_J$ as follows. 
\setlength{\unitlength}{1mm}
\begin{center}
  \begin{picture}(80, 15)(0,0)
  \put(0,10){\circle{2}}
  \put(10,10){\circle{2}}
  \put(20,10){\circle{2}}
  \put(50,10){\circle{2}}
  \put(60,10){\circle{2}}
  \put(1,10){\line(1,0){8}}
  \put(11,10){\line(1,0){8}}
  \put(21,10){\line(1,0){8}}
  \put(33,8.8){$\cdots$}
  \put(41,10){\line(1,0){8}}
  \put(51,10.5){\line(1,0){8}}
  \put(51,9.5){\line(1,0){8}}
  \put(-2,3){$\beta_1$}
  \put(8,3){$\beta_2$}
  \put(18,3){$\beta_3$}
  \put(48,3){$\beta_{\j-1}$}
  \put(58,3){$\beta_\j$}
  \qbezier(54.5,11.5)(57.5,10)(54.5,8.5)
  \end{picture}
\end{center}  
%\bigskip
\item[Type $(B_{\j-1},B_\j)$] Let $\Phi_K$ and $\Phi_J$ be root systems of type $B_{\j-1}$ and $B_\j$, respectively. 
Set $K=\{\beta_2, \ldots, \beta_\j\}$ and $J=\{\beta_1, \ldots, \beta_\j\}$ and we fix an ordering of the simple roots in $\Phi_K$ and $\Phi_J$ as follows. 
\setlength{\unitlength}{1mm}
\begin{center}
  \begin{picture}(80, 15)(0,0)
  \put(0,10){\circle{2}}
  \put(10,10){\circle{2}}
  \put(20,10){\circle{2}}
  \put(50,10){\circle{2}}
  \put(60,10){\circle{2}}
  \put(1,10){\line(1,0){8}}
  \put(11,10){\line(1,0){8}}
  \put(21,10){\line(1,0){8}}
  \put(33,8.8){$\cdots$}
  \put(41,10){\line(1,0){8}}
  \put(51,10.5){\line(1,0){8}}
  \put(51,9.5){\line(1,0){8}}
  \put(-2,3){$\beta_1$}
  \put(8,3){$\beta_2$}
  \put(18,3){$\beta_3$}
  \put(48,3){$\beta_{\j-1}$}
  \put(58,3){$\beta_\j$}
  \qbezier(54.5,11.5)(57.5,10)(54.5,8.5)
  \end{picture}
\end{center}  
%\bigskip
\item[Type $(A_{\j-1},C_\j)$] Let $\Phi_K$ and $\Phi_J$ be root systems of type $A_{\j-1}$ and $C_\j$, respectively. 
Set $K=\{\beta_1, \ldots, \beta_{\j-1}\}$ and $J=\{\beta_1, \ldots, \beta_\j\}$ and we fix an ordering of the simple roots in $\Phi_K$ and $\Phi_J$ as follows. 
\setlength{\unitlength}{1mm}
\begin{center}
  \begin{picture}(80, 15)(0,0)
  \put(0,10){\circle{2}}
  \put(10,10){\circle{2}}
  \put(20,10){\circle{2}}
  \put(50,10){\circle{2}}
  \put(60,10){\circle{2}}
  \put(1,10){\line(1,0){8}}
  \put(11,10){\line(1,0){8}}
  \put(21,10){\line(1,0){8}}
  \put(33,8.8){$\cdots$}
  \put(41,10){\line(1,0){8}}
  \put(51,10.5){\line(1,0){8}}
  \put(51,9.5){\line(1,0){8}}
  \put(-2,3){$\beta_1$}
  \put(8,3){$\beta_2$}
  \put(18,3){$\beta_3$}
  \put(48,3){$\beta_{\j-1}$}
  \put(58,3){$\beta_\j$}
  \qbezier(55.5,11.5)(52.5,10)(55.5,8.5)
  \end{picture}
\end{center}  
%\bigskip
\item[Type $(C_{\j-1},C_\j)$] Let $\Phi_K$ and $\Phi_J$ be root systems of type $C_{\j-1}$ and $C_\j$, respectively. 
Set $K=\{\beta_2, \ldots, \beta_\j\}$ and $J=\{\beta_1, \ldots, \beta_\j\}$ and we fix an ordering of the simple roots in $\Phi_K$ and $\Phi_J$ as follows. 
\setlength{\unitlength}{1mm}
\begin{center}
  \begin{picture}(80, 15)(0,0)
  \put(0,10){\circle{2}}
  \put(10,10){\circle{2}}
  \put(20,10){\circle{2}}
  \put(50,10){\circle{2}}
  \put(60,10){\circle{2}}
  \put(1,10){\line(1,0){8}}
  \put(11,10){\line(1,0){8}}
  \put(21,10){\line(1,0){8}}
  \put(33,8.8){$\cdots$}
  \put(41,10){\line(1,0){8}}
  \put(51,10.5){\line(1,0){8}}
  \put(51,9.5){\line(1,0){8}}
  \put(-2,3){$\beta_1$}
  \put(8,3){$\beta_2$}
  \put(18,3){$\beta_3$}
  \put(48,3){$\beta_{\j-1}$}
  \put(58,3){$\beta_\j$}
  \qbezier(55.5,11.5)(52.5,10)(55.5,8.5)
  \end{picture}
\end{center}  
%\bigskip
\item[Type $(A_{\j-1},D_\j)$] Let $\Phi_K$ and $\Phi_J$ be root systems of type $A_{\j-1}$ and $D_\j$, respectively. 
Set $K=\{\beta_1, \ldots, \beta_{\j-1}\}$ and $J=\{\beta_1, \ldots, \beta_\j\}$ and we fix an ordering of the simple roots in $\Phi_K$ and $\Phi_J$ as follows. 
\setlength{\unitlength}{1mm}
\begin{center}
  \begin{picture}(80, 17)(0,0)
  \put(0,10){\circle{2}}
  \put(10,10){\circle{2}}
  \put(20,10){\circle{2}}
  \put(50,10){\circle{2}}
  \put(60,4){\circle{2}}
  \put(60,16){\circle{2}}
  \put(1,10){\line(1,0){8}}
  \put(11,10){\line(1,0){8}}
  \put(21,10){\line(1,0){8}}
  \put(33,8.8){$\cdots$}
  \put(41,10){\line(1,0){8}}
  \put(51,10.5){\line(3,2){8}}
  \put(51,9.5){\line(3,-2){8}}
  \put(-2,3){$\beta_1$}
  \put(8,3){$\beta_2$}
  \put(18,3){$\beta_3$}
  \put(48,3){$\beta_{\j-2}$}
  \put(63,15){$\beta_{\j-1}$}
  \put(63,3){$\beta_\j$}
  \end{picture}
\end{center}  
%\smallskip
\item[Type $(D_{\j-1},D_\j)$] Let $\Phi_K$ and $\Phi_J$ be root systems of type $D_{\j-1}$ and $D_\j$, respectively. 
Set $K=\{\beta_2, \ldots, \beta_\j\}$ and $J=\{\beta_1, \ldots, \beta_\j\}$ and we fix an ordering of the simple roots in $\Phi_K$ and $\Phi_J$ as follows. 
\setlength{\unitlength}{1mm}
\begin{center}
  \begin{picture}(80, 17)(0,0)
  \put(0,10){\circle{2}}
  \put(10,10){\circle{2}}
  \put(20,10){\circle{2}}
  \put(50,10){\circle{2}}
  \put(60,4){\circle{2}}
  \put(60,16){\circle{2}}
  \put(1,10){\line(1,0){8}}
  \put(11,10){\line(1,0){8}}
  \put(21,10){\line(1,0){8}}
  \put(33,8.8){$\cdots$}
  \put(41,10){\line(1,0){8}}
  \put(51,10.5){\line(3,2){8}}
  \put(51,9.5){\line(3,-2){8}}
  \put(-2,3){$\beta_1$}
  \put(8,3){$\beta_2$}
  \put(18,3){$\beta_3$}
  \put(48,3){$\beta_{\j-2}$}
  \put(63,15){$\beta_{\j-1}$}
  \put(63,3){$\beta_\j$}
  \end{picture}
\end{center}  
%\smallskip
\item[Type $(A_5,E_6)$] Let $\Phi_K$ and $\Phi_J$ be root systems of type $A_5$ and $E_6$, respectively. 
Set $K=\{\beta_1, \ldots, \beta_5\}$ and $J=\{\beta_1, \ldots, \beta_6\}$ and we fix an ordering of the simple roots in $\Phi_K$ and $\Phi_J$ as follows. 
\setlength{\unitlength}{1mm}
\begin{center}
  \begin{picture}(80, 23)(0,0)
  \put(0,10){\circle{2}}
  \put(10,10){\circle{2}}
  \put(20,10){\circle{2}}
  \put(30,10){\circle{2}}
  \put(40,10){\circle{2}}
  \put(20,20){\circle{2}}
  \put(1,10){\line(1,0){8}}
  \put(11,10){\line(1,0){8}}
  \put(21,10){\line(1,0){8}}
  \put(31,10){\line(1,0){8}}
  \put(20,11){\line(0,1){8}}
  \put(-2,3){$\beta_1$}
  \put(8,3){$\beta_2$}
  \put(18,3){$\beta_3$}
  \put(28,3){$\beta_4$}
  \put(38,3){$\beta_5$}
  \put(22,19){$\beta_6$}
  \end{picture}
\end{center}  
%\bigskip
\item[Type $(D_5,E_6)$] Let $\Phi_K$ and $\Phi_J$ be root systems of type $D_5$ and $E_6$, respectively. 
Set $K=\{\beta_1, \ldots, \beta_5\}$ and $J=\{\beta_1, \ldots, \beta_6\}$ and we fix an ordering of the simple roots in $\Phi_K$ and $\Phi_J$ as follows. 
\setlength{\unitlength}{1mm}
\begin{center}
  \begin{picture}(80, 23)(0,0)
  \put(0,10){\circle{2}}
  \put(10,10){\circle{2}}
  \put(20,10){\circle{2}}
  \put(30,10){\circle{2}}
  \put(40,10){\circle{2}}
  \put(20,20){\circle{2}}
  \put(1,10){\line(1,0){8}}
  \put(11,10){\line(1,0){8}}
  \put(21,10){\line(1,0){8}}
  \put(31,10){\line(1,0){8}}
  \put(20,11){\line(0,1){8}}
  \put(-2,3){$\beta_1$}
  \put(8,3){$\beta_3$}
  \put(18,3){$\beta_4$}
  \put(28,3){$\beta_5$}
  \put(38,3){$\beta_6$}
  \put(22,19){$\beta_2$}
  \end{picture}
\end{center}  
%\bigskip
\item[Type $(A_6,E_7)$] Let $\Phi_K$ and $\Phi_J$ be root systems of type $A_6$ and $E_7$, respectively. 
Set $K=\{\beta_1, \ldots, \beta_6\}$ and $J=\{\beta_1, \ldots, \beta_7\}$ and we fix an ordering of the simple roots in $\Phi_K$ and $\Phi_J$ as follows. 
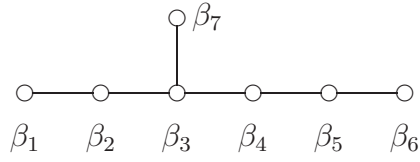
\setlength{\unitlength}{1mm}
\begin{center}
  \begin{picture}(80, 23)(0,0)
  \put(0,10){\circle{2}}
  \put(10,10){\circle{2}}
  \put(20,10){\circle{2}}
  \put(30,10){\circle{2}}
  \put(40,10){\circle{2}}
  \put(50,10){\circle{2}}
  \put(20,20){\circle{2}}
  \put(1,10){\line(1,0){8}}
  \put(11,10){\line(1,0){8}}
  \put(21,10){\line(1,0){8}}
  \put(31,10){\line(1,0){8}}
  \put(41,10){\line(1,0){8}}
  \put(20,11){\line(0,1){8}}
  \put(-2,3){$\beta_1$}
  \put(8,3){$\beta_2$}
  \put(18,3){$\beta_3$}
  \put(28,3){$\beta_4$}
  \put(38,3){$\beta_5$}
  \put(48,3){$\beta_6$}
  \put(22,19){$\beta_7$}
  \end{picture}
\end{center}  
%\bigskip
\item[Type $(D_6,E_7)$] Let $\Phi_K$ and $\Phi_J$ be root systems of type $D_6$ and $E_7$, respectively. 
Set $K=\{\beta_2, \ldots, \beta_7\}$ and $J=\{\beta_1, \ldots, \beta_7\}$ and we fix an ordering of the simple roots in $\Phi_K$ and $\Phi_J$ as follows. 
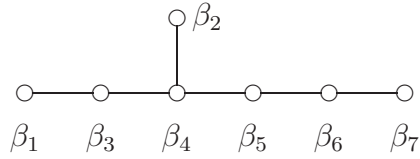
\setlength{\unitlength}{1mm}
\begin{center}
  \begin{picture}(80, 23)(0,0)
  \put(0,10){\circle{2}}
  \put(10,10){\circle{2}}
  \put(20,10){\circle{2}}
  \put(30,10){\circle{2}}
  \put(40,10){\circle{2}}
  \put(50,10){\circle{2}}
  \put(20,20){\circle{2}}
  \put(1,10){\line(1,0){8}}
  \put(11,10){\line(1,0){8}}
  \put(21,10){\line(1,0){8}}
  \put(31,10){\line(1,0){8}}
  \put(41,10){\line(1,0){8}}
  \put(20,11){\line(0,1){8}}
  \put(-2,3){$\beta_1$}
  \put(8,3){$\beta_3$}
  \put(18,3){$\beta_4$}
  \put(28,3){$\beta_5$}
  \put(38,3){$\beta_6$}
  \put(48,3){$\beta_7$}
  \put(22,19){$\beta_2$}
  \end{picture}
\end{center}  
%\bigskip
\item[Type $(E_6,E_7)$] Let $\Phi_K$ and $\Phi_J$ be root systems of type $E_6$ and $E_7$, respectively. 
Set $K=\{\beta_1, \ldots, \beta_6\}$ and $J=\{\beta_1, \ldots, \beta_7\}$ and we fix an ordering of the simple roots in $\Phi_K$ and $\Phi_J$ as follows. 
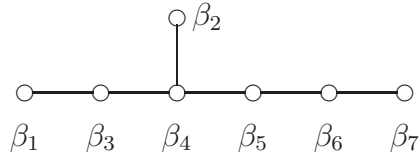
\setlength{\unitlength}{1mm}
\begin{center}
  \begin{picture}(80, 23)(0,0)
  \put(0,10){\circle{2}}
  \put(10,10){\circle{2}}
  \put(20,10){\circle{2}}
  \put(30,10){\circle{2}}
  \put(40,10){\circle{2}}
  \put(50,10){\circle{2}}
  \put(20,20){\circle{2}}
  \put(1,10){\line(1,0){8}}
  \put(11,10){\line(1,0){8}}
  \put(21,10){\line(1,0){8}}
  \put(31,10){\line(1,0){8}}
  \put(41,10){\line(1,0){8}}
  \put(20,11){\line(0,1){8}}
  \put(-2,3){$\beta_1$}
  \put(8,3){$\beta_3$}
  \put(18,3){$\beta_4$}
  \put(28,3){$\beta_5$}
  \put(38,3){$\beta_6$}
  \put(48,3){$\beta_7$}
  \put(22,19){$\beta_2$}
  \end{picture}
\end{center}  
%\bigskip
\item[Type $(A_7,E_8)$] Let $\Phi_K$ and $\Phi_J$ be root systems of type $A_7$ and $E_8$, respectively. 
Set $K=\{\beta_1, \ldots, \beta_7\}$ and $J=\{\beta_1, \ldots, \beta_8\}$ and we fix an ordering of the simple roots in $\Phi_K$ and $\Phi_J$ as follows. 
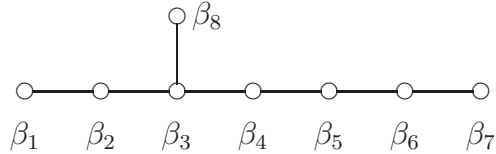
\setlength{\unitlength}{1mm}
\begin{center}
  \begin{picture}(80, 23)(0,0)
  \put(0,10){\circle{2}}
  \put(10,10){\circle{2}}
  \put(20,10){\circle{2}}
  \put(30,10){\circle{2}}
  \put(40,10){\circle{2}}
  \put(50,10){\circle{2}}
  \put(60,10){\circle{2}}
  \put(20,20){\circle{2}}
  \put(1,10){\line(1,0){8}}
  \put(11,10){\line(1,0){8}}
  \put(21,10){\line(1,0){8}}
  \put(31,10){\line(1,0){8}}
  \put(41,10){\line(1,0){8}}
  \put(51,10){\line(1,0){8}}
  \put(20,11){\line(0,1){8}}
  \put(-2,3){$\beta_1$}
  \put(8,3){$\beta_2$}
  \put(18,3){$\beta_3$}
  \put(28,3){$\beta_4$}
  \put(38,3){$\beta_5$}
  \put(48,3){$\beta_6$}
  \put(58,3){$\beta_7$}
  \put(22,19){$\beta_8$}
  \end{picture}
\end{center}  
%\bigskip
\item[Type $(D_7,E_8)$] Let $\Phi_K$ and $\Phi_J$ be root systems of type $D_7$ and $E_8$, respectively. 
Set $K=\{\beta_2, \ldots, \beta_8\}$ and $J=\{\beta_1, \ldots, \beta_8\}$ and we fix an ordering of the simple roots in $\Phi_K$ and $\Phi_J$ as follows. 
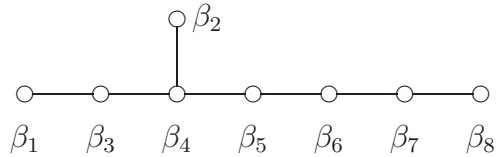
\setlength{\unitlength}{1mm}
\begin{center}
  \begin{picture}(80, 23)(0,0)
  \put(0,10){\circle{2}}
  \put(10,10){\circle{2}}
  \put(20,10){\circle{2}}
  \put(30,10){\circle{2}}
  \put(40,10){\circle{2}}
  \put(50,10){\circle{2}}
  \put(60,10){\circle{2}}
  \put(20,20){\circle{2}}
  \put(1,10){\line(1,0){8}}
  \put(11,10){\line(1,0){8}}
  \put(21,10){\line(1,0){8}}
  \put(31,10){\line(1,0){8}}
  \put(41,10){\line(1,0){8}}
  \put(51,10){\line(1,0){8}}
  \put(20,11){\line(0,1){8}}
  \put(-2,3){$\beta_1$}
  \put(8,3){$\beta_3$}
  \put(18,3){$\beta_4$}
  \put(28,3){$\beta_5$}
  \put(38,3){$\beta_6$}
  \put(48,3){$\beta_7$}
  \put(58,3){$\beta_8$}
  \put(22,19){$\beta_2$}
  \end{picture}
\end{center}  
%\bigskip
\item[Type $(E_7,E_8)$] Let $\Phi_K$ and $\Phi_J$ be root systems of type $E_7$ and $E_8$, respectively. 
Set $K=\{\beta_1, \ldots, \beta_7\}$ and $J=\{\beta_1, \ldots, \beta_8\}$ and we fix an ordering of the simple roots in $\Phi_K$ and $\Phi_J$ as follows. 
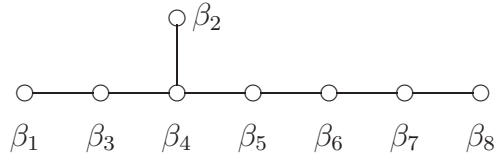
\setlength{\unitlength}{1mm}
\begin{center}
  \begin{picture}(80, 23)(0,0)
  \put(0,10){\circle{2}}
  \put(10,10){\circle{2}}
  \put(20,10){\circle{2}}
  \put(30,10){\circle{2}}
  \put(40,10){\circle{2}}
  \put(50,10){\circle{2}}
  \put(60,10){\circle{2}}
  \put(20,20){\circle{2}}
  \put(1,10){\line(1,0){8}}
  \put(11,10){\line(1,0){8}}
  \put(21,10){\line(1,0){8}}
  \put(31,10){\line(1,0){8}}
  \put(41,10){\line(1,0){8}}
  \put(51,10){\line(1,0){8}}
  \put(20,11){\line(0,1){8}}
  \put(-2,3){$\beta_1$}
  \put(8,3){$\beta_3$}
  \put(18,3){$\beta_4$}
  \put(28,3){$\beta_5$}
  \put(38,3){$\beta_6$}
  \put(48,3){$\beta_7$}
  \put(58,3){$\beta_8$}
  \put(22,19){$\beta_2$}
  \end{picture}
\end{center}  
%\bigskip
\item[Type $(B_3,F_4)$] Let $\Phi_K$ and $\Phi_J$ be root systems of type $B_3$ and $F_4$, respectively. 
Set $K=\{\beta_1, \beta_2, \beta_3\}$ and $J=\{\beta_1, \beta_2, \beta_3, \beta_4\}$ and we fix an ordering of the simple roots in $\Phi_K$ and $\Phi_J$ as follows. 
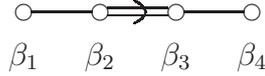
\setlength{\unitlength}{1mm}
\begin{center}
  \begin{picture}(60, 15)(0,0)
  \put(0,10){\circle{2}}
  \put(10,10){\circle{2}}
  \put(20,10){\circle{2}}
  \put(30,10){\circle{2}}
  \put(1,10){\line(1,0){8}}
  \put(11,10.5){\line(1,0){8}}
  \put(11,9.5){\line(1,0){8}}
  \put(21,10){\line(1,0){8}}
  \put(-2,3){$\beta_1$}
  \put(8,3){$\beta_2$}
  \put(18,3){$\beta_3$}
  \put(28,3){$\beta_4$}
  \qbezier(14.5,11.5)(17.5,10)(14.5,8.5)
  \end{picture}
\end{center}  
%\bigskip
\item[Type $(C_3,F_4)$] Let $\Phi_K$ and $\Phi_J$ be root systems of type $C_3$ and $F_4$, respectively. 
Set $K=\{\beta_2, \beta_3, \beta_4\}$ and $J=\{\beta_1, \beta_2, \beta_3, \beta_4\}$ and we fix an ordering of the simple roots in $\Phi_K$ and $\Phi_J$ as follows. 
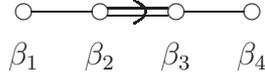
\setlength{\unitlength}{1mm}
\begin{center}
  \begin{picture}(60, 15)(0,0)
  \put(0,10){\circle{2}}
  \put(10,10){\circle{2}}
  \put(20,10){\circle{2}}
  \put(30,10){\circle{2}}
  \put(1,10){\line(1,0){8}}
  \put(11,10.5){\line(1,0){8}}
  \put(11,9.5){\line(1,0){8}}
  \put(21,10){\line(1,0){8}}
  \put(-2,3){$\beta_1$}
  \put(8,3){$\beta_2$}
  \put(18,3){$\beta_3$}
  \put(28,3){$\beta_4$}
  \qbezier(14.5,11.5)(17.5,10)(14.5,8.5)
  \end{picture}
\end{center}  
%\bigskip
\item[Type $(A_1,G_2){\rm \mathchar`-(i)}$] Let $\Phi_K$ and $\Phi_J$ be root systems of type $A_1$ and $G_2$, respectively. 
Set $K=\{\beta_1\}$ and $J=\{\beta_1, \beta_2\}$ and we fix an ordering of the simple roots in $\Phi_K$ and $\Phi_J$ as follows. 
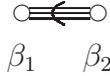
\setlength{\unitlength}{1mm}
\begin{center}
  \begin{picture}(40, 12)(0,0)
  \put(0,10){\circle{2}}
  \put(10,10){\circle{2}}
  \put(1,10){\line(1,0){8}}
  \put(1,10.7){\line(1,0){8}}
  \put(1,9.3){\line(1,0){8}}
  \put(-2,3){$\beta_1$}
  \put(8,3){$\beta_2$}
  \qbezier(5.5,11.5)(2.5,10)(5.5,8.5)
  \end{picture}
\end{center}  
%\bigskip
\item[Type $(A_1,G_2){\rm \mathchar`-(ii)}$] Let $\Phi_K$ and $\Phi_J$ be root systems of type $A_1$ and $G_2$, respectively. 
Set $K=\{\beta_2\}$ and $J=\{\beta_1, \beta_2\}$ and we fix an ordering of the simple roots in $\Phi_K$ and $\Phi_J$ as follows. 
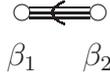
\setlength{\unitlength}{1mm}
\begin{center}
  \begin{picture}(40, 12)(0,0)
  \put(0,10){\circle{2}}
  \put(10,10){\circle{2}}
  \put(1,10){\line(1,0){8}}
  \put(1,10.7){\line(1,0){8}}
  \put(1,9.3){\line(1,0){8}}
  \put(-2,3){$\beta_1$}
  \put(8,3){$\beta_2$}
  \qbezier(5.5,11.5)(2.5,10)(5.5,8.5)
  \end{picture}
\end{center}  
\smallskip
\end{description}

In this setting we describe a list of values of $m_{i,K}^J$.
Recall that 
\begin{align*}
\j=|J|.
\end{align*}
For a positive integer $i$ with $\alpha_i \in K$, set
\begin{align*} 
\beta_{i'} \coloneqq \alpha_i.
\end{align*}

We will give a proof of the following lemma in Appendix~\ref{sect:computation for miKJ}.

\begin{lemma} \label{lemm:computationanytype2}
The rational number $m_{i,K}^J$ in \eqref{eq:miKJ} is given in Table~$\ref{tab:A list of values}$.
\end{lemma}

\begin{table}[htb]
  \begin{tabular}{|c|c|c|} \hline
    Type $(\Phi_K,\Phi_J)$ & Ambient root system $\Phi$ & Value $m_{i,K}^J$ \\ \hline \hline
    $(A_{\j-1},A_\j)$ & $A,B,C,D,E,F$ & $\frac{i'}{\j}$ \\ \hline
    $(A_{\j-1},B_\j)$ & $B,F$ & $\frac{2i'}{\j}$ \\ \hline
    $(B_{\j-1},B_\j)$ & $B,F$ & $1$ \ if $\beta_{i'} \neq \beta_\j$ \\ 
                       &       & $\frac{1}{2}$ \ if $\beta_{i'} = \beta_\j$ \\ \hline
    $(A_{\j-1},C_\j)$ & $C,F$ & $\frac{i'}{\j}$ \\ \hline
    $(C_{\j-1},C_\j)$ & $C,F$ & $1$ \\ \hline
    $(A_{\j-1},D_\j)$ & $D,E$ & $\frac{2i'}{\j}$ \ if $\beta_{i'} \neq \beta_\j$ \ \ \ \\ 
                       &          & $\frac{\j-2}{\j}$ \ if $\beta_{i'} = \beta_{\j-1}$ \\ \hline
    $(D_{\j-1},D_\j)$ & $D,E$ & $1$ \ if $\beta_{i'} \neq \beta_{\j-1}, \beta_\j$ \\ 
                       &          & $\frac{1}{2}$ \ if $\beta_{i'} = \beta_{\j-1}, \beta_\j$ \\ \hline
    $(A_5,E_6)$ & $E$ &  $\frac{i'}{2}$ if $\beta_{i'} = \beta_1, \beta_2, \beta_3$ \ \ \ \\ 
                       &          & $\frac{6-i'}{2}$ \ if $\beta_{i'} = \beta_4, \beta_5$ \\ \hline
    $(D_5,E_6)$ & $E$ & $\frac{i'+1}{4}$ \ if $\beta_{i'} = \beta_1, \beta_2, \beta_3$ \\ 
                       &          & $\frac{10-i'}{4}$ \ if $\beta_{i'} = \beta_4, \beta_5$ \\ \hline
    $(A_6,E_7)$ & $E$ & $\frac{4i'}{7}$ \ if $\beta_{i'} = \beta_1, \beta_2, \beta_3$ \\ 
                    &       & $\frac{3(7-i')}{7}$ \ if $\beta_{i'} = \beta_4, \beta_5, \beta_6$ \\ \hline
    $(D_6,E_7)$ & $E$ & $\frac{i'}{2}$ \ if $\beta_{i'} = \beta_2, \beta_3$ \\ 
                       &    & $\frac{8-i'}{2}$ \ if $\beta_{i'} = \beta_4, \beta_5, \beta_6, \beta_7$ \\ \hline
    $(E_6,E_7)$ & $E$ & $\frac{i'+1}{3}$ \ if $\beta_{i'} = \beta_1, \beta_2, \beta_3$ \\ 
                       &    & $\frac{10-i'}{3}$ \ if $\beta_{i'} = \beta_4, \beta_5, \beta_6$ \\ \hline
    $(A_7,E_8)$ & $E$ & $\frac{5i'}{8}$ \ if $\beta_{i'} = \beta_1, \beta_2, \beta_3$ \\ 
                    &       & $\frac{3(8-i')}{8}$ \ if $\beta_{i'} = \beta_4, \beta_5, \beta_6, \beta_7$ \\ \hline
    $(D_7,E_8)$ & $E$ & $\frac{2i'+1}{4}$ \ if $\beta_{i'} = \beta_2, \beta_3$ \\ 
                       &    & $\frac{9-i'}{2}$ \ if $\beta_{i'} = \beta_4, \beta_5, \beta_6, \beta_7, \beta_8$ \\ \hline 
    $(E_7,E_8)$ & $E$ & $\frac{i'+1}{2}$ \ if $\beta_{i'} = \beta_1, \beta_2, \beta_3$ \\ 
                       &    & $\frac{10-i'}{2}$ \ if $\beta_{i'} = \beta_4, \beta_5, \beta_6,\beta_7$ \\ \hline
    $(B_3,F_4)$ & $F$ & $i'$ if $\beta_{i'} = \beta_1, \beta_2$ \\ 
                       &    & $\frac{3}{2}$ \ if $\beta_{i'} = \beta_3$ \\ \hline
    $(C_3,F_4)$ & $F$ & $\frac{5-i'}{2}$ \\ \hline
    $(A_1,G_2){\rm \mathchar`-(i)}$ & $G$ & $\frac{1}{2}$ \ \ \ \\ \hline
    $(A_1,G_2){\rm \mathchar`-(ii)}$ & $G$ & $\frac{3}{2}$ \\ \hline
  \end{tabular}
  \caption{A list of values of $m_{i,K}^J$.}
\label{tab:A list of values}
\end{table}

\begin{example} \label{ex:m_iKJ type A}
Consider the case of type $A_{n-1}$.
Let $K=\{\alpha_k \in \Sigma \mid \a \leq k \leq \b \}$ and fix $i$ with $\a \leq i \leq \b$.
Then, we have 
\begin{align*}
\varpi_i \cdot \prod_{\alpha_k \in K} \varpi_{k} = m_{i,K}^{J_1} \prod_{\alpha_j \in J_1} \varpi_{j} + m_{i,K}^{J_2} \prod_{\alpha_j \in J_2} \varpi_{j}. 
\end{align*}
by Lemma~\ref{lemm:computationanytype1} where $J_1=\{\alpha_k \in \Sigma \mid \a-1 \leq k \leq \b \}$ and $J_2=\{\alpha_k \in \Sigma \mid \a \leq k \leq \b+1 \}$.
Note that we take the convention $\prod_{\alpha_j \in J_1} \varpi_{j}=0$ whenever $\a=1$ and $\prod_{\alpha_j \in J_2} \varpi_{j}=0$ whenever $\b=n-1$.
We compute the coefficients $m_{i,K}^{J_1}$ and $m_{i,K}^{J_2}$ by using Table~\ref{tab:A list of values}.
Notice that $(\Phi_K, \Phi_{J_1})$ is type $(A_{\j-1},A_{\j})$ with $\beta_j=\alpha_{\b-j+1}$ for $1 \leq j \leq \j$ and $\j=\b-\a+2$.
Then we have $\alpha_i=\beta_{i'}=\alpha_{\b-i'+1}$ and hence $i'=\b-i+1$.
It follows from Table~\ref{tab:A list of values} that $m_{i,K}^{J_1}=\frac{\b-i+1}{\b-\a+2}$.
On the other hand, $(\Phi_K, \Phi_{J_2})$ is type $(A_{\j-1},A_{\j})$ with $\beta_j=\alpha_{j+\a-1}$ for $1 \leq j \leq \j$ and $\j=\b-\a+2$.
Note that $\alpha_i=\beta_{i'}=\alpha_{i'+\a-1}$, namely $i'=i-\a+1$.
By using Table~\ref{tab:A list of values}, we have $m_{i,K}^{J_2}=\frac{i-\a+1}{\b-\a+2}$.
Therefore, we conclude that 
\begin{align*}
\varpi_i \cdot \big( \varpi_\a \varpi_{\a+1} \cdots \varpi_{\b} \big) =  \frac{\b-i+1}{\b-\a+2}  \varpi_{\a-1} \varpi_\a \cdots \varpi_{\b} + \frac{i-\a+1}{\b-\a+2} \varpi_\a \varpi_{\a+1} \cdots \varpi_{\b+1}. 
\end{align*}
This is nothing but Lemma~\ref{lemm:TypeA}.
\end{example}

\begin{example} \label{ex:m_iKJ type B}
Consider the case of type $B_n$.
Let $K=\{\alpha_k \in \Sigma \mid \a \leq k \leq \b \}$ and fix $i$ with $\a \leq i \leq \b$.
Then, we have 
\begin{align} \label{eq:m_iKJ type B}
\varpi_i \cdot \prod_{\alpha_k \in K} \varpi_{k} = m_{i,K}^{J_1} \prod_{\alpha_j \in J_1} \varpi_{j} + m_{i,K}^{J_2} \prod_{\alpha_j \in J_2} \varpi_{j}. 
\end{align}
by Lemma~\ref{lemm:computationanytype1} where $J_1=\{\alpha_k \in \Sigma \mid \a-1 \leq k \leq \b \}$ and $J_2=\{\alpha_k \in \Sigma \mid \a \leq k \leq \b+1 \}$.
We take the convention again $\prod_{\alpha_j \in J_1} \varpi_{j}=0$ whenever $\a=1$ and $\prod_{\alpha_j \in J_2} \varpi_{j}=0$ whenever $\b=n$.
We compute the coefficients $m_{i,K}^{J_1}$ and $m_{i,K}^{J_2}$ by using Table~\ref{tab:A list of values}.
For this purpose we take cases.\\
\textbf{Case(i):} Suppose that $\b \leq n-2$. 
Then both of types for $(\Phi_K, \Phi_{J_1})$ and $(\Phi_K, \Phi_{J_2})$ are type $(A_{\j-1},A_{\j})$. 
Since $m_{i,K}^J$ depends only on the Dynkin diagrams of $(\Phi_K,\Phi_J)$ and the vertex $\alpha_i$ by Lemma~\ref{lemm:computationanytype1.5}, we have $m_{i,K}^{J_1}=\frac{\b-i+1}{\b-\a+2}$ and $m_{i,K}^{J_2}=\frac{i-\a+1}{\b-\a+2}$ from Example~\ref{ex:m_iKJ type A}.
In other words, one obtains
\begin{align*}
\varpi_i \cdot \big( \varpi_\a \varpi_{\a+1} \cdots \varpi_{\b} \big) =  \frac{\b-i+1}{\b-\a+2}  \varpi_{\a-1} \varpi_\a \cdots \varpi_{\b} + \frac{i-\a+1}{\b-\a+2} \varpi_\a \varpi_{\a+1} \cdots \varpi_{\b+1}. 
\end{align*}
\textbf{Case(ii):} Suppose that $\b=n-1$. Then $(\Phi_K, \Phi_{J_1})$ is type $(A_{\j-1},A_{\j})$, so one has $m_{i,K}^{J_1}=\frac{n-i}{n-\a+1}$.
On the other hand, $(\Phi_K, \Phi_{J_2})$ is type $(A_{\j-1},B_{\j})$ with $\beta_j=\alpha_{j+\a-1}$ for $1 \leq j \leq \j$ and $\j=n-\a+1$.
Since $\alpha_i=\beta_{i'}=\alpha_{i'+\a-1}$, namely $i'=i-\a+1$, we have $m_{i,K}^{J_2}=\frac{2(i-\a+1)}{n-\a+1}$ by Table~\ref{tab:A list of values}.
Hence, we obtain 
\begin{align*}
\varpi_i \cdot \big( \varpi_\a \varpi_{\a+1} \cdots \varpi_{n-1} \big) =  \frac{n-i}{n-\a+1}  \varpi_{\a-1} \varpi_\a \cdots \varpi_{n-1} + \frac{2(i-\a+1)}{n-\a+1} \varpi_\a \varpi_{\a+1} \cdots \varpi_{n}. 
\end{align*}
This equality will be proved in Appendix~\ref{sect:computation for miKJ} (Lemma~\ref{lemm:TypeB-1}). \\
\textbf{Case(iii):} Suppose that $\b=n$. Then, we first note that the second term of the right hand side in \eqref{eq:m_iKJ type B} does not appear, so we compute only $m_{i,K}^{J_1}$.
Noticing that $(\Phi_K, \Phi_{J_1})$ is type $(B_{\j-1},B_{\j})$ with $\beta_j=\alpha_{j+\a-2}$ for $1 \leq j \leq \j$ and $\j=n-\a+2$, we have $\alpha_i=\beta_{i'}=\alpha_{i'+\a-2}$ and $i'=i-\a+2$.
This implies that $m_{i,K}^{J_1}=1$ if $i < n$ and $m_{n,K}^{J_1}=\frac{1}{2}$ from Table~\ref{tab:A list of values}.
Thus, we have  
\begin{align*}
\varpi_i \cdot \big( \varpi_\a \varpi_{\a+1} \cdots \varpi_{n} \big) &=  \varpi_{\a-1} \varpi_\a \cdots \varpi_{n} \ \ \ \ {\rm if} \ a \leq i<n, \\
\varpi_n \cdot \big( \varpi_\a \varpi_{\a+1} \cdots \varpi_{n} \big) &=  \frac{1}{2} \varpi_{\a-1} \varpi_\a \cdots \varpi_{n}, 
\end{align*}
in this case (see Lemma~\ref{lemm:TypeB-2}).
Analogous formulas for other types will be described in Appendix~\ref{sect:computation for miKJ}.
\end{example}

\begin{remark} \label{rem:Peterson Schubert calculus}
Recall that Peterson Schubert calculus is to calculate the coefficients $c_{IK}^J$ in \eqref{eq:Peterson_Schubert_calculus}. 
By using Lemmas~\ref{lemm:computationanytype1} and \ref{lemm:computationanytype2}, we can derive an efficient computation for Peterson Schubert calculus.
In fact, if we decompose each subset $I\subset \Sigma$ and $K\subset \Sigma$ into connected components $I = I_1 \sqcup \cdots \sqcup I_c$ and $K = K_1 \sqcup \cdots \sqcup K_d$ respectively, then we have $p_{v_I}=\prod_{j=1}^c p_{v_{I_j}}$ and $p_{v_K}=\prod_{m=1}^d p_{v_{K_m}}$ by \cite[Theorem~5.3]{Dre2}.
Hence, it then follows from Theorem~\ref{theorem:GiambelliPet} and \eqref{eq:PetersonSchubert-varpi} that 
\begin{align*} 
p_{v_I} \cdot p_{v_K} &= \prod_{j=1}^c p_{v_{I_j}} \cdot \prod_{m=1}^d p_{v_{K_m}} = \prod_{j=1}^c \frac{|\Reduce(v_{I_j})|}{|{I_j}|!} \left( \prod_{\alpha_i \in I_j} \varpi_i \right) \cdot \prod_{m=1}^d \frac{|\Reduce(v_{K_j})|}{|{K_j}|!} \left( \prod_{\alpha_k \in K_m} \varpi_k \right) \\
&=\frac{|\Reduce(v_{I_1})| \cdots |\Reduce(v_{I_c})|}{|{I_1}|! \cdots |{I_c}|!} \frac{|\Reduce(v_{K_1})| \cdots |\Reduce(v_{K_d})|}{|{K_1}|! \cdots |{K_d}|!} \left( \prod_{\alpha_i \in I} \varpi_i \right) \cdot \left( \prod_{\alpha_k \in K} \varpi_k \right) \\
&=\frac{|\Reduce(v_{I_1})| \cdots |\Reduce(v_{I_c})|}{|{I_1}|! \cdots |{I_c}|!} \frac{|\Reduce(v_{K_1})| \cdots |\Reduce(v_{K_d})|}{|{K_1}|! \cdots |{K_d}|!} \left( \prod_{\alpha_p \in I \cap K} \varpi_p \right) \cdot \left( \prod_{\alpha_q \in I \cup K} \varpi_q \right). 
\end{align*}
Here, by using Lemma~\ref{lemm:computationanytype1} repeatedly, the product $\left( \prod_{\alpha_p \in I \cap K} \varpi_p \right) \cdot \left( \prod_{\alpha_q \in I \cup K} \varpi_q \right)$ can be expanded by square-free monomials in $\varpi_1, \ldots, \varpi_n$, as in Example~\ref{ex:calculus_TypeA}.
The expansion is explicit by Lemma~\ref{lemm:computationanytype2}.
Namely, the product $p_{v_I} \cdot p_{v_K}$ can be explicitly expanded by square-free monomials in $\varpi_1, \ldots, \varpi_n$.
By \cite[Theorem~5.3]{Dre2}, Theorem~\ref{theorem:GiambelliPet}, and \eqref{eq:PetersonSchubert-varpi} again, we derive an explicit formula for \eqref{eq:Peterson_Schubert_calculus}.
In particular, one can see that all coefficients $c_{IK}^J$ in \eqref{eq:Peterson_Schubert_calculus} are non-negative. 
\end{remark}

We now summarize our computation of the mixed $\Phi$-Eulerian numbers.
Recall that $\varpi_i$ is defined in \eqref{eq:varpianytypePet}.

\begin{theorem} \label{theorem:main_any_type}
Let $\Phi$ be an irreducible root system.
Let $c_1,\ldots, c_n$ be non-negative integers with $c_1 +\cdots+c_n = n$.
\begin{enumerate}
\item The mixed $\Phi$-Eulerian number $A^{\Phi}_{c_1,\ldots,c_n}$ is equal to
\begin{align*} 
A^{\Phi}_{c_1,\ldots,c_n} = \int_{\Pet_{\Phi}} \varpi_1^{c_1}\varpi_2^{c_2}\cdots\varpi_n^{c_n}. 
\end{align*}
\item If $K \subset \Sigma$ is connected and $K \ni \alpha_i$, then we have 
\begin{align*}
\varpi_i \cdot \prod_{\alpha_k \in K} \varpi_{k} = \sum_{J \subset \Sigma: \, {\rm connected} \atop J \supset K \, {\rm and } \, |J|=|K|+1} m_{i,K}^J \prod_{\alpha_j \in J} \varpi_{j}. 
\end{align*}
Here, $m_{i,K}^J$ depends only on the Dynkin diagrams of $(\Phi_K,\Phi_J)$ and the vertex $\alpha_i$ of the Dynkin diagram of $\Phi_K$.
The coefficients $m_{i,K}^J$ are given in Table~$\ref{tab:A list of values}$.
\item We have
\begin{align*} 
\int_{\Pet_{\Phi}} \varpi_1\varpi_2\cdots\varpi_n = \frac{|W|}{\det (C_{\Phi})}, 
\end{align*}
where $W$ is the Weyl group and $C_{\Phi}$ is the associated Cartan matrix.
$($See Table~$\ref{tab:A list of values mPhi}$ in Section~$\ref{sect:Any_Lie_types}$ for the explicit values.$)$
\end{enumerate}
\end{theorem}

\begin{proof}
By \cite[Corollary 3.9]{AFZ} we obtain
\begin{align*}
\int_{X_{\Phi}} \varpi_1^{c_1}\varpi_2^{c_2}\cdots\varpi_n^{c_n}=\int_{\Pet_{\Phi}} \varpi_1^{c_1}\varpi_2^{c_2}\cdots\varpi_n^{c_n}
\end{align*}
for non-negative integers $c_1,\ldots, c_n$ with $c_1 +\cdots+c_n = n$.
Then, the claims~(1) and (3) follow from Propositions~\ref{proposition:mixed Eulerian number} and \ref{prop:Top_Class}. 

The claim~(2) is nothing but Lemmas~\ref{lemm:computationanytype1}, \ref{lemm:computationanytype1.5}, and \ref{lemm:computationanytype2}.
\end{proof}

Theorem~\ref{theorem:main_any_type} yields a simple computation for the mixed $\Phi$-Eulerian numbers, as discussed in Section~\ref{sect:simple_computation_TypeA}.

\begin{example}
When $\Phi$ is of type $B_n$, we compute the mixed $\Phi$-Eulerian number $A^{\Phi_{B_n}}_{k,0,\ldots,0,n-k}$. 
By Theorem~\ref{theorem:main_any_type}, we have 
\begin{align*} 
A^{\Phi_{B_n}}_{k,0,\ldots,0,n-k} = \int_{\Pet_{\Phi_{B_n}}} \varpi_1^{k} \cdot \varpi_n^{n-k}. 
\end{align*}
By the formula in Case(i) of Example~\ref{ex:m_iKJ type B}, $\varpi_1^{k}$ can be inductively computed as 
\begin{align*} 
\varpi_1^{k} &= \varpi_1^{2} \cdot \varpi_1^{k-2} = \left(\frac{1}{2}\varpi_1\varpi_2 \right) \cdot \varpi_1^{k-2} \\
&= \left(\frac{1}{2}\varpi_1^2\varpi_2 \right) \cdot \varpi_1^{k-3}= \left(\frac{1}{2}\cdot\frac{1}{3}\varpi_1\varpi_2\varpi_3 \right) \cdot \varpi_1^{k-3} \\
&= \cdots = \frac{1}{k!} \varpi_1\varpi_2 \cdots \varpi_{k}. 
\end{align*}
On the other hand, using the formula in Case(iii) of Example~\ref{ex:m_iKJ type B},
we can also calculate $\varpi_n^{n-k}$ by
\begin{align*} 
\varpi_n^{n-k} &= \varpi_n^{2} \cdot \varpi_n^{n-k-2} = \left(\frac{1}{2}\varpi_{n-1}\varpi_n \right) \cdot \varpi_n^{n-k-2} \\
&= \left(\frac{1}{2}\varpi_{n-1}\varpi_n^2 \right) \cdot \varpi_n^{n-k-3}= \left(\frac{1}{2}\cdot\frac{1}{2}\varpi_{n-2}\varpi_{n-1}\varpi_n \right) \cdot \varpi_n^{n-k-3} \\ 
&=\cdots=\frac{1}{2^{n-k-1}} \varpi_{k+1}\varpi_{k+2} \cdots \varpi_{n}. 
\end{align*}
Therefore, we obtain
\begin{align*} 
A^{\Phi_{B_n}}_{k,0,\ldots,0,n-k} &= \frac{1}{k!} \cdot \frac{1}{2^{n-k-1}} \int_{\Pet_{\Phi_{B_n}}} \varpi_1\varpi_2 \cdots \varpi_n \\
&= \frac{1}{k!} \cdot \frac{1}{2^{n-k-1}} \cdot 2^{n-1} \cdot n! \\
&= \dbinom{n}{k} \cdot (n-k)! \cdot 2^{k}.
\end{align*}
This is an analogue of \cite[Theorem~16.3 (8)]{Pos}.
\end{example}

\begin{example}
When $\Phi$ is of type $E_6$, we compute the mixed $\Phi$-Eulerian number $A^{\Phi_{E_6}}_{0,1,0,2,3,0}$. We obtain  
\begin{align*} 
A^{\Phi_{E_6}}_{0,1,0,2,3,0} = \int_{\Pet_{\Phi_{E_6}}} \varpi_2 \cdot \varpi_4^2 \cdot \varpi_5^3
\end{align*}
from Theorem~\ref{theorem:main_any_type}~(1). 
Noting that $\{\alpha_2, \alpha_4, \alpha_5 \}$ is connected whose own Lie type is $A_3$ in $\Phi_{E_6}$ as shown in Figure~\ref{pic: Dynkin diagrams}, we have 
\begin{align*}
\varpi_4 \cdot (\varpi_2 \varpi_4\varpi_5) = \varpi_2\varpi_3\varpi_4\varpi_5 + \frac{1}{2}\varpi_2 \varpi_4\varpi_5\varpi_6
\end{align*}
by Theorem~\ref{theorem:main_any_type}~(2) since $\{\alpha_2, \alpha_3, \alpha_4, \alpha_5 \}$ is of type $D_4$ and $\{\alpha_2, \alpha_4, \alpha_5, \alpha_6 \}$ is of type $A_4$. 
By using Theorem~\ref{theorem:main_any_type}~(2) repeatedly, we obtain
\begin{align*} 
\varpi_5 \cdot \varpi_5 \cdot \varpi_4 \cdot (\varpi_2 \varpi_4\varpi_5) =& \varpi_5 \cdot \varpi_5 \cdot \left( \varpi_2\varpi_3\varpi_4\varpi_5 + \frac{1}{2}\varpi_2 \varpi_4\varpi_5\varpi_6 \right)\\
=& \varpi_5 \cdot \left( \frac{1}{2}\varpi_1\varpi_2\varpi_3\varpi_4\varpi_5 + \varpi_2\varpi_3\varpi_4\varpi_5\varpi_6 + \frac{4}{5} \cdot \frac{1}{2}\varpi_2\varpi_3\varpi_4\varpi_5\varpi_6 \right) \\
=& \left( \frac{5}{4} \cdot \frac{1}{2} + 1 + \frac{4}{5} \cdot \frac{1}{2} \right) \varpi_1 \varpi_2 \varpi_3 \varpi_4 \varpi_5 \varpi_6.
\end{align*}
Therefore, we conclude from Theorem~\ref{theorem:main_any_type}~(3) that 
\begin{align*} 
A^{\Phi_{E_6}}_{0,1,0,2,3,0} &= \left( \frac{5}{4} \cdot \frac{1}{2} + 1 + \frac{4}{5} \cdot \frac{1}{2} \right) \int_{\Pet_{\Phi_{E_6}}} \varpi_1 \varpi_2 \varpi_3 \varpi_4 \varpi_5 \varpi_6 \\
&= \left( \frac{5}{4} \cdot \frac{1}{2} + 1 + \frac{4}{5} \cdot \frac{1}{2} \right) \cdot 2^{7} \cdot 3^{3} \cdot 5 \\
&= 10800 + 17280 + 6912 =34992.
\end{align*}
\end{example}

\bigskip
\noindent \textbf{Funding.}  
This work was supported in part by Osaka City University Advanced Mathematical Institute (MEXT Joint Usage/Research Center on Mathematics and Theoretical Physics); and in part by JSPS Grant-in-Aid for Young Scientists: 19K14508.

\bigskip
\noindent \textbf{Acknowledgements.}  
I am grateful to Mikiya Masuda for his support and encouragement.
I would like to thank Hiraku Abe, Hideya Kuwata, and Haozhi Zeng for a valuable discussion.
This work is an outcome of a project in \cite{AHKZ} with them.
I also appreciate a useful discussion with Rebecca Goldin, Brent Gorbutt, Megumi Harada, Leonardo Mihalcea, and Rahul Singh.
Finally, I thank the referee for valuable and concrete suggestions to improve the paper.

\smallskip
\smallskip

\appendix

%%%%%%%%%%%%%%%%%%%%%%%%%%%%%%%%%%
\section{A computation for $m_{i,K}^J$} \label{sect:computation for miKJ}
%%%%%%%%%%%%%%%%%%%%%%%%%%%%%%%%%%

We give a proof of Lemma~\ref{lemm:computationanytype2} by case-by-case arguments.
By using Theorem~\ref{theorem:MonkPet} one can compute $m_{i,K}^J$, but we give an alternative proof which is a similar argument in \cite[Lemma~5.1]{AHKZ}. 

Let $\mathcal{R} \coloneqq \Sym \Lambda_\R$ be the symmetric algebra of the weight space $\Lambda_\R$. 
As discussed in Section~\ref{sect:Any_Lie_types}, we can construct a ring homomorphism $\mathcal{R} \to H^*(G/B)$ which sends a weight $\chi$ to $c_1(L_{\chi}^*)$. 
It is well-known that this map is surjective (\cite{Bor}).
Composing this map with the restriction map $H^*(G/B) \to H^*(\Pet_{\Phi})$, we obtain the surjective ring homomorphism
\begin{align*}
\varphi: \mathcal{R} \twoheadrightarrow H^*(\Pet_\Phi).
\end{align*}
Here, we note that the restriction map $H^*(G/B) \to H^*(\Pet_{\Phi})$ is surjective by Theorem~\ref{Theorem_Peterson_basis}.

\begin{theorem} $($\cite[Theorem~4.1]{HHM}$)$ \label{theorem:presentation_Pet}
Let $\Phi$ be a crystallographic root system of rank $n$.
The kernel of $\varphi$ is generated by quadratic forms $\alpha_i \varpi_i \ (1 \leq i \leq n)$. 
Namely, the map $\varphi$ induces the isomorphism 
\begin{align*}
H^*(\Pet_\Phi) \cong \mathcal{R}/(\alpha_i \varpi_i \mid 1 \leq i \leq n).
\end{align*}
\end{theorem}

We now compute $m_{i,K}^J$ by case-by-case arguments.
In what follows, we take the convention $\varpi_0=\varpi_{n+1}=0$.
Recall that we assume to be fixed an ordering of the simple roots as given in Figure~\ref{pic: Dynkin diagrams}.
Note that the Cartan matrix (e.g. \cite[p.59]{Hum1}) expresses the change of basis from the fundamental weights $\varpi_1, \ldots, \varpi_n$ to the simple roots $\alpha_1, \ldots, \alpha_n$.

\subsection{Classical types} 
The following lemmas correspond to Lemma~\ref{lemm:TypeA}.

\begin{lemma} \label{lemm:TypeB-1}
Let $\Phi$ be an irreducible root system of type $B_n$.
For $1 \leq \a \leq i \leq n-1$, we have
\begin{align*}
\varpi_i \cdot \big( \varpi_\a \varpi_{\a+1} \cdots \varpi_{n-1} \big) =  \frac{n-i}{n-\a+1}  \varpi_{\a-1} \varpi_\a \cdots \varpi_{n-1} + \frac{2(i-\a+1)}{n-\a+1} \varpi_\a \varpi_{\a+1} \cdots \varpi_{n}. 
\end{align*}
in $H^*(\Pet_{\Phi_{B_n}})$.
In particular, the second summand of the right hand side gives the value of $m_{i,K}^J$ when $(\Phi_K,\Phi_J)$ is of type $(A_{\j-1},B_{\j})$ in Table~$\ref{tab:A list of values}$. 
\end{lemma}

\begin{proof}
We prove the claim by induction on $n-\a (\geq 1)$. As the base case when $n-a=1$, we have $\a=i=n-1$. 
Then since $\alpha_{n-1} = -\varpi_{n-2}+2\varpi_{n-1}-2\varpi_{n}$, we have
\begin{align*}
\varpi_{n-1} (-\varpi_{n-2}+2\varpi_{n-1}-2\varpi_{n}) = \varpi_{n-1} \alpha_{n-1} =0
\end{align*}
by Theorem~\ref{theorem:presentation_Pet}.
This equality can be expressed as
\begin{align*}
\varpi_{n-1}^2 = \frac{1}{2}\varpi_{n-2}\varpi_{n-1}+\frac{2}{2}\varpi_{n-1}\varpi_{n},
\end{align*}
which proves the base case.

We proceed to the inductive step.
Suppose now that $1<n-a$ and that the claim holds for $n-\a'<n-\a$, with any allowable choices of $\a'\le i' \le n-1$.
When $i=a$, we have
\begin{align*}
\varpi_{a}^2(\varpi_{a+1}\cdots \varpi_{n-1}) 
&=\Big( \varpi_{a}^2(\varpi_{a+1}\cdots \varpi_{n-2}) \Big) \varpi_{n-1} \\
&=\Big(\frac{n-a-1}{n-a}\varpi_{a-1}\varpi_{a}\cdots \varpi_{n-2} + \frac{1}{n-a}\varpi_{a}\varpi_{a+1}\cdots \varpi_{n-1} \Big)\varpi_{n-1} \\
&\hspace{180pt}\text{(by Lemmas~\ref{lemm:TypeA}, \ref{lemm:computationanytype1}, and \ref{lemm:computationanytype1.5})}\\
&=\frac{n-a-1}{n-a} \varpi_{a-1}\varpi_{a}\cdots \varpi_{n-1} + \frac{1}{n-a} \varpi_{a}\Big( \varpi_{a+1}\cdots \varpi_{n-2}\varpi_{n-1}^2 \Big) \\
&=\frac{n-a-1}{n-a} \varpi_{a-1}\varpi_{a}\cdots \varpi_{n-1} \\
&\hspace{30pt}+ \frac{1}{(n-a)^2} \varpi_{a}\Big( \varpi_{a}\varpi_{a+1}\cdots \varpi_{n-1} + 2(n-a-1) \varpi_{a+1}\varpi_{a+2}\cdots \varpi_{n}\Big) \\
&\hspace{200pt} \text{(by the inductive assumption)}.
\end{align*}
Since the left hand side and the second summand of the right hand side are proportional, this equality can be written as
\begin{align*}
\frac{(n-a)^2-1}{(n-a)^2}\varpi_{a}^2(\varpi_{a+1}\cdots \varpi_{n-1}) 
=\frac{n-a-1}{n-a} \varpi_{a-1}\varpi_{a}\cdots \varpi_{n-1} + \frac{2(n-a-1)}{(n-a)^2} \varpi_{a}\varpi_{a+1}\cdots \varpi_{n}.
\end{align*}
Noting that the numerator of the coefficient of the left hand side is equal to $(n-a)^2-1=(n-a-1)(n-a+1)$, we obtain 
\begin{align}\label{eq:proof_lemma_B1}
\varpi_{a}^2(\varpi_{a+1}\cdots \varpi_{n-1}) = \ \frac{n-a}{n-a+1}\varpi_{a-1}\varpi_{a}\cdots \varpi_{n-1} + \frac{2}{n-a+1} \varpi_{a}\varpi_{a+1}\cdots \varpi_{n}, 
\end{align}
which proves the claim for the case $i=a$. 
When $a<i \ (\le n-1)$, we have 
\begin{align*}
&(\varpi_{a}\varpi_{a+1}\cdots \varpi_{i-1}) \varpi_{i}^2 (\varpi_{i+1}\varpi_{i+2}\cdots \varpi_{n-1})\\
=&\varpi_{a}\Big((\varpi_{a+1}\cdots \varpi_{i-1}) \varpi_{i}^2 (\varpi_{i+1}\varpi_{i+2}\cdots \varpi_{n-1}) \Big) \\
=&\varpi_{a}\Big( \frac{n-i}{n-a}\varpi_{a}\varpi_{a+1}\cdots \varpi_{n-1} + \frac{2(i-a)}{n-a} \varpi_{a+1}\varpi_{a+2}\cdots \varpi_{n} \Big)\\
&\hspace{220pt} \text{(by the inductive assumption)} \\
=&\frac{n-i}{n-a}\varpi_{a}^2(\varpi_{a+1}\cdots \varpi_{n-1}) + \frac{2(i-a)}{n-a}\varpi_{a}\varpi_{a+1}\cdots \varpi_{n} \\ 
=&\frac{n-i}{n-a} \Big( \frac{n-a}{n-a+1}\varpi_{a-1}\varpi_{a}\cdots \varpi_{n-1} + \frac{2}{n-a+1}\varpi_{a}\varpi_{a+1}\cdots \varpi_{n}\Big) \\
&\hspace{30pt}+ \frac{2(i-a)}{n-a}\varpi_{a}\varpi_{a+1}\cdots \varpi_{n}\qquad \text{(by \eqref{eq:proof_lemma_B1})}\\
=&\frac{n-i}{n-a+1} \varpi_{a-1}\varpi_{a}\cdots \varpi_{n-1} + \frac{2(i-a+1)}{n-a+1}\varpi_{a}\varpi_{a+1}\cdots \varpi_{n}.
\end{align*}
This completes the proof by induction.
\end{proof}

\begin{lemma} \label{lemm:TypeB-2}
Let $\Phi$ be an irreducible root system of type $B_n$.
For $1 \leq \a \leq i \leq n$, we have
\begin{align*}
\varpi_i \cdot \big( \varpi_\a \varpi_{\a+1} \cdots \varpi_{n} \big) &= b_i \, \varpi_{\a-1} \varpi_\a \cdots \varpi_{n}, 
\end{align*}
in $H^*(\Pet_{\Phi_{B_n}})$ where $b_i \ (\a \leq i \leq n)$ is defined by 
\begin{align*}
b_i=\begin{cases}
1 \ \ \ &{\rm if} \ a \leq i \leq n-1, \\
\frac{1}{2} \ \ \ &{\rm if} \ i = n.
\end{cases}
\end{align*}
In particular, this formula derives the value of $m_{i,K}^J$ when $(\Phi_K,\Phi_J)$ is of type $(B_{\j-1},B_{\j})$ in Table~$\ref{tab:A list of values}$. 
\end{lemma}

\begin{proof}
We prove the claim by induction on $n-\a +1(\geq 1)$. 
As the base case when $n-a+1=1$, we have $\a=i=n$. 
Then since $\alpha_{n} = -\varpi_{n-1}+2\varpi_{n}$, we have
\begin{align*}
\varpi_{n} (-\varpi_{n-1}+2\varpi_{n}) = \varpi_{n} \alpha_{n} =0
\end{align*}
by Theorem~\ref{theorem:presentation_Pet}.
This equality can be written as
\begin{align*}
\varpi_{n}^2 = \frac{1}{2}\varpi_{n-1}\varpi_{n},
\end{align*}
which proves the base case.

Suppose now that $1<n-a+1$ and that the claim holds for $n-\a' +1<n-\a+1$ with $\a'\le i' \le n$.
When $i=a$, we have
\begin{align*}
\varpi_{a}^2(\varpi_{a+1}\cdots \varpi_{n}) 
&=\Big( \varpi_{a}^2(\varpi_{a+1}\cdots \varpi_{n-1}) \Big) \varpi_{n} \\
&=\Big(\frac{n-a}{n-a+1}\varpi_{a-1}\varpi_{a}\cdots \varpi_{n-1} + \frac{2}{n-a+1}\varpi_{a}\varpi_{a+1}\cdots \varpi_{n} \Big)\varpi_{n} \\
&\hspace{200pt}\text{(by Lemma~\ref{lemm:TypeB-1})}\\
&=\frac{n-a}{n-a+1} \varpi_{a-1}\varpi_{a}\cdots \varpi_{n} + \frac{2}{n-a+1} \varpi_{a}\Big( \varpi_{a+1}\cdots \varpi_{n-1}\varpi_{n}^2 \Big) \\
&=\frac{n-a}{n-a+1} \varpi_{a-1}\varpi_{a}\cdots \varpi_{n} + \frac{1}{n-a+1} \varpi_{a}\Big(  \varpi_{a}\varpi_{a+1}\cdots \varpi_{n} \Big) \\
&\hspace{200pt} \text{(by the inductive hypothesis)}.
\end{align*}
Since the left hand side and the second summand of the right hand side are proportional, this equality can be written as
\begin{align}\label{eq:proof_lemma_B2}
\varpi_{a}^2(\varpi_{a+1}\cdots \varpi_{n}) = \varpi_{a-1}\varpi_{a}\cdots \varpi_{n},
\end{align}
which proves the claim for the case $i=a$. 
When $a<i \ (\le n)$, we obtain 
\begin{align*}
&(\varpi_{a}\varpi_{a+1}\cdots \varpi_{i-1}) \varpi_{i}^2 (\varpi_{i+1}\varpi_{i+2}\cdots \varpi_{n})\\
=&\varpi_{a}\Big((\varpi_{a+1}\cdots \varpi_{i-1}) \varpi_{i}^2 (\varpi_{i+1}\varpi_{i+2}\cdots \varpi_{n}) \Big) \\
=&\varpi_{a}\Big( b_i \varpi_{a}\varpi_{a+1}\cdots \varpi_{n} \Big) \ \ \ \text{(by the induction hypothesis)} \\
=&b_i\varpi_{a}^2\varpi_{a+1}\cdots \varpi_{n} \\ 
=&b_i\varpi_{a-1}\varpi_{a}\cdots \varpi_{n} \ \ \ \text{(by \eqref{eq:proof_lemma_B2})},
\end{align*}
as desired.
This completes the proof by induction.
\end{proof}

We can prove the following lemmas for type $C_n$ and $D_n$ by a similar argument of Lemmas~\ref{lemm:TypeB-1} and \ref{lemm:TypeB-2} with suitable modifications, so we omit details (note that the base case in type $D_n$ is $a=n-2$).

\begin{lemma} \label{lemm:TypeC-1}
Let $\Phi$ be an irreducible root system of type $C_n$.
For $1 \leq \a \leq i \leq n-1$, we have
\begin{align*}
\varpi_i \cdot \big( \varpi_\a \varpi_{\a+1} \cdots \varpi_{n-1} \big) =  \frac{n-i}{n-\a+1}  \varpi_{\a-1} \varpi_\a \cdots \varpi_{n-1} + \frac{i-\a+1}{n-\a+1} \varpi_\a \varpi_{\a+1} \cdots \varpi_{n} 
\end{align*}
in $H^*(\Pet_{\Phi_{C_n}})$.
In particular, the second summand of the right hand side gives the value of $m_{i,K}^J$ when $(\Phi_K,\Phi_J)$ is of type $(A_{\j-1},C_{\j})$ in Table~$\ref{tab:A list of values}$. 
\end{lemma}

\begin{lemma} \label{lemm:TypeC-2}
Let $\Phi$ be an irreducible root system of type $C_n$.
For $1 \leq \a \leq i \leq n$, we have
\begin{align*}
\varpi_i \cdot \big( \varpi_\a \varpi_{\a+1} \cdots \varpi_{n} \big) &=  \varpi_{\a-1} \varpi_\a \cdots \varpi_{n}  
\end{align*}
in $H^*(\Pet_{\Phi_{C_n}})$.
In particular, this formula derives the value of $m_{i,K}^J$ when $(\Phi_K,\Phi_J)$ is of type $(C_{\j-1},C_{\j})$ in Table~$\ref{tab:A list of values}$. 
\end{lemma}

\begin{lemma} \label{lemm:TypeD-1}
Let $\Phi$ be an irreducible root system of type $D_n$.
Let $1 \leq \a \leq i \leq n-1$. 
Assume that $a \leq n-2$. 
Then we have
\begin{align*}
\varpi_i \cdot \big( \varpi_\a \varpi_{\a+1} \cdots \varpi_{n-1} \big) =  \frac{n-i}{n-\a+1}  \varpi_{\a-1} \varpi_\a \cdots \varpi_{n-1} + d_i^{(1)} \, \varpi_\a \varpi_{\a+1} \cdots \varpi_{n}
\end{align*}
in $H^*(\Pet_{\Phi_{D_n}})$ where $d_i^{(1)} \ (\a \leq i \leq n-1)$ is defined by 
\begin{align*}
d_i^{(1)}=\begin{cases}
\displaystyle\frac{2(i-\a+1)}{n-\a+1} \ \ \ &{\rm if} \ a \leq i \leq n-2, \\
\displaystyle\frac{n-\a-1}{n-\a+1} \ \ \ &{\rm if} \ i = n-1.
\end{cases}
\end{align*}
In particular, the second summand of the right hand side gives the value of $m_{i,K}^J$ when $(\Phi_K,\Phi_J)$ is of type $(A_{\j-1},D_{\j})$ in Table~$\ref{tab:A list of values}$. 
\end{lemma}

\begin{lemma} \label{lemm:TypeD-2}
Let $\Phi$ be an irreducible root system of type $D_n$.
Let $1 \leq \a \leq i \leq n$. 
Assume that $a \leq n-2$. 
Then we have
\begin{align*}
\varpi_i \cdot \big( \varpi_\a \varpi_{\a+1} \cdots \varpi_{n} \big) &= d_i^{(2)} \varpi_{\a-1} \varpi_\a \cdots \varpi_{n}  
\end{align*}
in $H^*(\Pet_{\Phi_{D_n}})$ where $d_i^{(2)} \ (\a \leq i \leq n-1)$ is defined by 
\begin{align*}
d_i^{(2)}=\begin{cases}
1 \ \ \ &{\rm if} \ a \leq i \leq n-2, \\
\displaystyle\frac{1}{2} \ \ \ &{\rm if} \ i = n-1, n.
\end{cases}
\end{align*}
In particular, this formula derives the value of $m_{i,K}^J$ when $(\Phi_K,\Phi_J)$ is of type $(D_{\j-1},D_{\j})$ in Table~$\ref{tab:A list of values}$. 
\end{lemma}

\smallskip

\subsection{Exceptional types} 

By using the fundamental relations $\alpha_i \varpi_i=0 \ (1 \leq i \leq n)$ in $H^*(\Pet_\Phi)$,
one can compute the values of $m_{i,K}^J$ for exceptional types by some elementary algebraic manipulations. 
We sketch this for only \eqref{eq:F-1} in Lemma~\ref{lemm:TypeF} below.

\begin{lemma} \label{lemm:TypeG}
Let $\Phi$ be an irreducible root system of type $G_2$.
We have
\begin{align*}
\varpi_1 \cdot \varpi_1 =  \frac{1}{2} \varpi_1 \varpi_2 \ {\rm and} \ \varpi_2 \cdot \varpi_2 =  \frac{3}{2} \varpi_1 \varpi_2 
\end{align*}
in $H^*(\Pet_{\Phi_{G_2}})$.
In particular, these formulas give the value of $m_{i,K}^J$ when $(\Phi_K,\Phi_J)$ is of type $(A_1,G_2){\rm \mathchar`-(i)}$ and $(A_1,G_2){\rm \mathchar`-(ii)}$ in Table~$\ref{tab:A list of values}$, respectively. 
\end{lemma}

\begin{lemma} \label{lemm:TypeF}
Let $\Phi$ be an irreducible root system of type $F_4$.
We have
\begin{align}\label{eq:F-1}
 \begin{cases}
\varpi_1 \cdot (\varpi_1\varpi_2\varpi_3) =  \varpi_1 \varpi_2 \varpi_3 \varpi_4, \\
\varpi_2 \cdot (\varpi_1\varpi_2\varpi_3) =  2 \varpi_1 \varpi_2 \varpi_3 \varpi_4, \\
\varpi_3 \cdot (\varpi_1\varpi_2\varpi_3) =  \frac{3}{2} \varpi_1 \varpi_2 \varpi_3 \varpi_4,  \end{cases}
\end{align}
\begin{align}\label{eq:F-2}
 \begin{cases}
\varpi_2 \cdot (\varpi_2\varpi_3\varpi_4) &=  \frac{3}{2} \varpi_1 \varpi_2 \varpi_3 \varpi_4, \\
\varpi_3 \cdot (\varpi_2\varpi_3\varpi_4) &=  \varpi_1 \varpi_2 \varpi_3 \varpi_4, \\
\varpi_4 \cdot (\varpi_2\varpi_3\varpi_4) &=  \frac{1}{2} \varpi_1 \varpi_2 \varpi_3 \varpi_4, 
 \end{cases}
\end{align}
in $H^*(\Pet_{\Phi_{F_4}})$.
In particular, these formulas~\eqref{eq:F-1} and \eqref{eq:F-2} give the value of $m_{i,K}^J$ when $(\Phi_K,\Phi_J)$ is of type $(B_3,F_4)$ and $(C_3,F_4)$ in Table~$\ref{tab:A list of values}$, respectively. 
\end{lemma}

\begin{proof}
Since $\alpha_1 = 2\varpi_1-\varpi_2, \alpha_2 = -\varpi_1+2\varpi_2-2\varpi_3$, $\alpha_3=-\varpi_2+2\varpi_3-\varpi_4$, and $\alpha_4=-\varpi_3+2\varpi_4$, we have
\begin{align*}
\varpi_1 (2\varpi_1-\varpi_2) &= \varpi_1 \alpha_1 =0, \\
\varpi_2 (-\varpi_1+2\varpi_2-2\varpi_3) &= \varpi_2 \alpha_2 =0, \\ 
\varpi_3 (-\varpi_2+2\varpi_3-\varpi_4) &= \varpi_3 \alpha_3 =0, \\
\varpi_4 (-\varpi_3+2\varpi_4) &= \varpi_4 \alpha_4 =0, 
\end{align*}
by Theorem~\ref{theorem:presentation_Pet}.
These equalities can be written as
\begin{align}
\varpi_1^2 &= \frac{1}{2}\varpi_1\varpi_2, \label{eq:1proofTypeF} \\
\varpi_2^2 &= \frac{1}{2}\varpi_1\varpi_2 + \varpi_2\varpi_3, \label{eq:2proofTypeF} \\ 
\varpi_3^2 &= \frac{1}{2}\varpi_2\varpi_3 + \frac{1}{2}\varpi_3\varpi_4, \label{eq:3proofTypeF} \\
\varpi_4^2 &= \frac{1}{2}\varpi_3\varpi_4. \label{eq:3.5proofTypeF}
\end{align}
We then have 
\begin{align}
\varpi_1^2 \varpi_2 \varpi_3 &= \frac{1}{2}\varpi_1 \varpi_2^2 \varpi_3, \label{eq:4proofTypeF} \\
\varpi_1 \varpi_2^2 \varpi_3 &= \frac{1}{2}\varpi_1^2 \varpi_2 \varpi_3 + \varpi_1 \varpi_2 \varpi_3^2, \label{eq:5proofTypeF} \\
\varpi_1 \varpi_2 \varpi_3^2 &= \frac{1}{2}\varpi_1\varpi_2^2\varpi_3 + \frac{1}{2}\varpi_1\varpi_2\varpi_3\varpi_4, \label{eq:6proofTypeF}
\end{align}
by \eqref{eq:1proofTypeF}, \eqref{eq:2proofTypeF}, and \eqref{eq:3proofTypeF}.
Hence, we obtain
\begin{align*} 
\varpi_1 \varpi_2^2 \varpi_3 &= \frac{1}{2}\varpi_1^2 \varpi_2 \varpi_3 + \varpi_1 \varpi_2 \varpi_3^2 \ \ \ \text{(by \eqref{eq:5proofTypeF})} \notag \\
&= \frac{1}{2}\varpi_1^2 \varpi_2 \varpi_3 + \frac{1}{2}\varpi_1\varpi_2^2\varpi_3 + \frac{1}{2}\varpi_1\varpi_2\varpi_3\varpi_4 \ \ \ \text{(by \eqref{eq:6proofTypeF})} \notag \\
&=\frac{1}{4}\varpi_1 \varpi_2^2 \varpi_3 + \frac{1}{2}\varpi_1\varpi_2^2\varpi_3 + \frac{1}{2}\varpi_1\varpi_2\varpi_3\varpi_4 \ \ \ \text{(by \eqref{eq:4proofTypeF})} \notag \\
&=\frac{3}{4}\varpi_1 \varpi_2^2 \varpi_3 + \frac{1}{2}\varpi_1\varpi_2\varpi_3\varpi_4. \notag
\end{align*}
This can be written as
\begin{align*} 
\varpi_1 \varpi_2^2 \varpi_3 = 2 \varpi_1\varpi_2\varpi_3\varpi_4.
\end{align*}
It also follows from \eqref{eq:4proofTypeF} and \eqref{eq:6proofTypeF} that 
\begin{align*} 
\varpi_1^2 \varpi_2 \varpi_3 &= \varpi_1\varpi_2\varpi_3\varpi_4, \\
\varpi_1 \varpi_2 \varpi_3^2 &= \frac{3}{2} \varpi_1\varpi_2\varpi_3\varpi_4,
\end{align*}
which shows \eqref{eq:F-1}.
A similar manipulation yields \eqref{eq:F-2}. 
\end{proof}

\begin{lemma} \label{lemm:TypeE6}
Let $\Phi$ be an irreducible root system of type $E_6$. 
Then we have
\begin{align}\label{eq:E6-1}
 \begin{cases}
\varpi_1 \cdot (\varpi_1\varpi_3\varpi_4\varpi_5\varpi_6) &=  \frac{1}{2} \varpi_1 \varpi_2 \varpi_3 \varpi_4\varpi_5\varpi_6, \\
\varpi_3 \cdot (\varpi_1\varpi_3\varpi_4\varpi_5\varpi_6) &=  \varpi_1 \varpi_2 \varpi_3 \varpi_4\varpi_5\varpi_6, \\
\varpi_4 \cdot (\varpi_1\varpi_3\varpi_4\varpi_5\varpi_6) &=  \frac{3}{2} \varpi_1 \varpi_2 \varpi_3 \varpi_4\varpi_5\varpi_6, \\
\varpi_5 \cdot (\varpi_1\varpi_3\varpi_4\varpi_5\varpi_6) &=  \varpi_1 \varpi_2 \varpi_3 \varpi_4\varpi_5\varpi_6, \\
\varpi_6 \cdot (\varpi_1\varpi_3\varpi_4\varpi_5\varpi_6) &=  \frac{1}{2} \varpi_1 \varpi_2 \varpi_3 \varpi_4\varpi_5\varpi_6,   
 \end{cases}
\end{align}
\begin{align}\label{eq:E6-2}
 \begin{cases}
\varpi_1 \cdot (\varpi_1\varpi_2\varpi_3\varpi_4\varpi_5) &=  \frac{1}{2} \varpi_1 \varpi_2 \varpi_3 \varpi_4\varpi_5\varpi_6, \\
\varpi_2 \cdot (\varpi_1\varpi_2\varpi_3\varpi_4\varpi_5) &=  \frac{3}{4} \varpi_1 \varpi_2 \varpi_3 \varpi_4\varpi_5\varpi_6, \\
\varpi_3 \cdot (\varpi_1\varpi_2\varpi_3\varpi_4\varpi_5) &=  \varpi_1 \varpi_2 \varpi_3 \varpi_4\varpi_5\varpi_6, \\
\varpi_4 \cdot (\varpi_1\varpi_2\varpi_3\varpi_4\varpi_5) &=  \frac{3}{2} \varpi_1 \varpi_2 \varpi_3 \varpi_4\varpi_5\varpi_6, \\
\varpi_5 \cdot (\varpi_1\varpi_2\varpi_3\varpi_4\varpi_5) &=  \frac{5}{4} \varpi_1 \varpi_2 \varpi_3 \varpi_4\varpi_5\varpi_6,  
 \end{cases}
\end{align}
in $H^*(\Pet_{\Phi_{E_6}})$.
In particular, these formulas~\eqref{eq:E6-1} and \eqref{eq:E6-2} give the value of $m_{i,K}^J$ when $(\Phi_K,\Phi_J)$ is of type $(A_5,E_6)$ and $(D_5,E_6)$ in Table~$\ref{tab:A list of values}$, respectively. 
\end{lemma}

\begin{lemma} \label{lemm:TypeE7}
Let $\Phi$ be an irreducible root system of type $E_7$. 
Then we have
\begin{align}\label{eq:E7-1}
 \begin{cases}
\varpi_1 \cdot (\varpi_1\varpi_3\varpi_4\varpi_5\varpi_6\varpi_7) &=  \frac{4}{7} \varpi_1 \varpi_2 \varpi_3 \varpi_4\varpi_5\varpi_6\varpi_7, \\
\varpi_3 \cdot (\varpi_1\varpi_3\varpi_4\varpi_5\varpi_6\varpi_7) &=  \frac{8}{7} \varpi_1 \varpi_2 \varpi_3 \varpi_4\varpi_5\varpi_6\varpi_7, \\
\varpi_4 \cdot (\varpi_1\varpi_3\varpi_4\varpi_5\varpi_6\varpi_7) &=  \frac{12}{7} \varpi_1 \varpi_2 \varpi_3 \varpi_4\varpi_5\varpi_6\varpi_7, \\
\varpi_5 \cdot (\varpi_1\varpi_3\varpi_4\varpi_5\varpi_6\varpi_7) &=  \frac{9}{7} \varpi_1 \varpi_2 \varpi_3 \varpi_4\varpi_5\varpi_6\varpi_7, \\
\varpi_6 \cdot (\varpi_1\varpi_3\varpi_4\varpi_5\varpi_6\varpi_7) &=  \frac{6}{7} \varpi_1 \varpi_2 \varpi_3 \varpi_4\varpi_5\varpi_6\varpi_7, \\
\varpi_7 \cdot (\varpi_1\varpi_3\varpi_4\varpi_5\varpi_6\varpi_7) &=  \frac{3}{7} \varpi_1 \varpi_2 \varpi_3 \varpi_4\varpi_5\varpi_6\varpi_7,
 \end{cases}
\end{align}
\begin{align}\label{eq:E7-2}
 \begin{cases}
\varpi_2 \cdot (\varpi_2\varpi_3\varpi_4\varpi_5\varpi_6\varpi_7) &=  \varpi_1 \varpi_2 \varpi_3 \varpi_4\varpi_5\varpi_6\varpi_7, \\
\varpi_3 \cdot (\varpi_2\varpi_3\varpi_4\varpi_5\varpi_6\varpi_7) &=  \frac{3}{2} \varpi_1 \varpi_2 \varpi_3 \varpi_4\varpi_5\varpi_6\varpi_7, \\
\varpi_4 \cdot (\varpi_2\varpi_3\varpi_4\varpi_5\varpi_6\varpi_7) &=  2 \varpi_1 \varpi_2 \varpi_3 \varpi_4\varpi_5\varpi_6\varpi_7, \\
\varpi_5 \cdot (\varpi_2\varpi_3\varpi_4\varpi_5\varpi_6\varpi_7) &=  \frac{3}{2} \varpi_1 \varpi_2 \varpi_3 \varpi_4\varpi_5\varpi_6\varpi_7, \\
\varpi_6 \cdot (\varpi_2\varpi_3\varpi_4\varpi_5\varpi_6\varpi_7) &=  \varpi_1 \varpi_2 \varpi_3 \varpi_4\varpi_5\varpi_6\varpi_7, \\ 
\varpi_7 \cdot (\varpi_2\varpi_3\varpi_4\varpi_5\varpi_6\varpi_7) &=  \frac{1}{2} \varpi_1 \varpi_2 \varpi_3 \varpi_4\varpi_5\varpi_6\varpi_7,  
 \end{cases}
\end{align}
\begin{align}\label{eq:E7-3}
 \begin{cases}
\varpi_1 \cdot (\varpi_1\varpi_2\varpi_3\varpi_4\varpi_5\varpi_6) &=  \frac{2}{3} \varpi_1 \varpi_2 \varpi_3 \varpi_4\varpi_5\varpi_6\varpi_7, \\
\varpi_2 \cdot (\varpi_1\varpi_2\varpi_3\varpi_4\varpi_5\varpi_6) &=  \varpi_1 \varpi_2 \varpi_3 \varpi_4\varpi_5\varpi_6\varpi_7, \\
\varpi_3 \cdot (\varpi_1\varpi_2\varpi_3\varpi_4\varpi_5\varpi_6) &=  \frac{4}{3} \varpi_1 \varpi_2 \varpi_3 \varpi_4\varpi_5\varpi_6\varpi_7, \\
\varpi_4 \cdot (\varpi_1\varpi_2\varpi_3\varpi_4\varpi_5\varpi_6) &=  2 \varpi_1 \varpi_2 \varpi_3 \varpi_4\varpi_5\varpi_6\varpi_7, \\
\varpi_5 \cdot (\varpi_1\varpi_2\varpi_3\varpi_4\varpi_5\varpi_6) &=  \frac{5}{3} \varpi_1 \varpi_2 \varpi_3 \varpi_4\varpi_5\varpi_6\varpi_7, \\
\varpi_6 \cdot (\varpi_1\varpi_2\varpi_3\varpi_4\varpi_5\varpi_6) &=  \frac{4}{3} \varpi_1 \varpi_2 \varpi_3 \varpi_4\varpi_5\varpi_6\varpi_7,    
 \end{cases}
\end{align}
in $H^*(\Pet_{\Phi_{E_7}})$.
In particular, these formulas~\eqref{eq:E7-1}, \eqref{eq:E7-2}, and \eqref{eq:E7-3} give the value of $m_{i,K}^J$ when $(\Phi_K,\Phi_J)$ is of type $(A_6,E_7)$, $(D_6,E_7)$, and $(E_6,E_7)$ in Table~$\ref{tab:A list of values}$, respectively. 
\end{lemma}

\begin{lemma} \label{lemm:TypeE8}
Let $\Phi$ be an irreducible root system of type $E_8$. 
Then we have
\begin{align}\label{eq:E8-1}
 \begin{cases}
\varpi_1 \cdot (\varpi_1\varpi_3\varpi_4\varpi_5\varpi_6\varpi_7\varpi_8) &=  \frac{5}{8} \varpi_1 \varpi_2 \varpi_3 \varpi_4\varpi_5\varpi_6\varpi_7\varpi_8, \\
\varpi_3 \cdot (\varpi_1\varpi_3\varpi_4\varpi_5\varpi_6\varpi_7\varpi_8) &=  \frac{5}{4} \varpi_1 \varpi_2 \varpi_3 \varpi_4\varpi_5\varpi_6\varpi_7\varpi_8, \\
\varpi_4 \cdot (\varpi_1\varpi_3\varpi_4\varpi_5\varpi_6\varpi_7\varpi_8) &=  \frac{15}{8} \varpi_1 \varpi_2 \varpi_3 \varpi_4\varpi_5\varpi_6\varpi_7\varpi_8, \\
\varpi_5 \cdot (\varpi_1\varpi_3\varpi_4\varpi_5\varpi_6\varpi_7\varpi_8) &=  \frac{3}{2} \varpi_1 \varpi_2 \varpi_3 \varpi_4\varpi_5\varpi_6\varpi_7\varpi_8, \\
\varpi_6 \cdot (\varpi_1\varpi_3\varpi_4\varpi_5\varpi_6\varpi_7\varpi_8) &=  \frac{9}{8} \varpi_1 \varpi_2 \varpi_3 \varpi_4\varpi_5\varpi_6\varpi_7\varpi_8, \\
\varpi_7 \cdot (\varpi_1\varpi_3\varpi_4\varpi_5\varpi_6\varpi_7\varpi_8) &=  \frac{3}{4} \varpi_1 \varpi_2 \varpi_3 \varpi_4\varpi_5\varpi_6\varpi_7\varpi_8, \\
\varpi_8 \cdot (\varpi_1\varpi_3\varpi_4\varpi_5\varpi_6\varpi_7\varpi_8) &=  \frac{3}{8} \varpi_1 \varpi_2 \varpi_3 \varpi_4\varpi_5\varpi_6\varpi_7\varpi_8,
 \end{cases}
\end{align}
\begin{align}\label{eq:E8-2}
 \begin{cases}
\varpi_2 \cdot (\varpi_2\varpi_3\varpi_4\varpi_5\varpi_6\varpi_7\varpi_8) &=  \frac{5}{4} \varpi_1 \varpi_2 \varpi_3 \varpi_4\varpi_5\varpi_6\varpi_7\varpi_8, \\
\varpi_3 \cdot (\varpi_2\varpi_3\varpi_4\varpi_5\varpi_6\varpi_7\varpi_8) &=  \frac{7}{4} \varpi_1 \varpi_2 \varpi_3 \varpi_4\varpi_5\varpi_6\varpi_7\varpi_8, \\
\varpi_4 \cdot (\varpi_2\varpi_3\varpi_4\varpi_5\varpi_6\varpi_7\varpi_8) &=  \frac{5}{2} \varpi_1 \varpi_2 \varpi_3 \varpi_4\varpi_5\varpi_6\varpi_7\varpi_8, \\
\varpi_5 \cdot (\varpi_2\varpi_3\varpi_4\varpi_5\varpi_6\varpi_7\varpi_8) &=  2 \varpi_1 \varpi_2 \varpi_3 \varpi_4\varpi_5\varpi_6\varpi_7\varpi_8, \\
\varpi_6 \cdot (\varpi_2\varpi_3\varpi_4\varpi_5\varpi_6\varpi_7\varpi_8) &=  \frac{3}{2} \varpi_1 \varpi_2 \varpi_3 \varpi_4\varpi_5\varpi_6\varpi_7\varpi_8, \\ 
\varpi_7 \cdot (\varpi_2\varpi_3\varpi_4\varpi_5\varpi_6\varpi_7\varpi_8) &=  \varpi_1 \varpi_2 \varpi_3 \varpi_4\varpi_5\varpi_6\varpi_7\varpi_8, \\ 
\varpi_8 \cdot (\varpi_2\varpi_3\varpi_4\varpi_5\varpi_6\varpi_7\varpi_8) &=  \frac{1}{2} \varpi_1 \varpi_2 \varpi_3 \varpi_4\varpi_5\varpi_6\varpi_7\varpi_8,  
 \end{cases}
\end{align}
\begin{align}\label{eq:E8-3}
 \begin{cases}
\varpi_1 \cdot (\varpi_1\varpi_2\varpi_3\varpi_4\varpi_5\varpi_6\varpi_7) &=  \varpi_1 \varpi_2 \varpi_3 \varpi_4\varpi_5\varpi_6\varpi_7\varpi_8, \\
\varpi_2 \cdot (\varpi_1\varpi_2\varpi_3\varpi_4\varpi_5\varpi_6\varpi_7) &=  \frac{3}{2} \varpi_1 \varpi_2 \varpi_3 \varpi_4\varpi_5\varpi_6\varpi_7\varpi_8, \\
\varpi_3 \cdot (\varpi_1\varpi_2\varpi_3\varpi_4\varpi_5\varpi_6\varpi_7) &=  2 \varpi_1 \varpi_2 \varpi_3 \varpi_4\varpi_5\varpi_6\varpi_7\varpi_8, \\
\varpi_4 \cdot (\varpi_1\varpi_2\varpi_3\varpi_4\varpi_5\varpi_6\varpi_7) &=  3 \varpi_1 \varpi_2 \varpi_3 \varpi_4\varpi_5\varpi_6\varpi_7\varpi_8, \\
\varpi_5 \cdot (\varpi_1\varpi_2\varpi_3\varpi_4\varpi_5\varpi_6\varpi_7) &=  \frac{5}{2} \varpi_1 \varpi_2 \varpi_3 \varpi_4\varpi_5\varpi_6\varpi_7\varpi_8, \\
\varpi_6 \cdot (\varpi_1\varpi_2\varpi_3\varpi_4\varpi_5\varpi_6\varpi_7) &=  2 \varpi_1 \varpi_2 \varpi_3 \varpi_4\varpi_5\varpi_6\varpi_7\varpi_8, \\    
\varpi_7 \cdot (\varpi_1\varpi_2\varpi_3\varpi_4\varpi_5\varpi_6\varpi_7) &=  \frac{3}{2} \varpi_1 \varpi_2 \varpi_3 \varpi_4\varpi_5\varpi_6\varpi_7\varpi_8,    
 \end{cases}
\end{align}
in $H^*(\Pet_{\Phi_{E_8}})$.
In particular, these formulas~\eqref{eq:E8-1}, \eqref{eq:E8-2}, and \eqref{eq:E8-3} derive the value of $m_{i,K}^J$ when $(\Phi_K,\Phi_J)$ is of type $(A_7,E_8)$, $(D_7,E_8)$, and $(E_7,E_8)$ in Table~$\ref{tab:A list of values}$, respectively. 
\end{lemma}

\smallskip

%%%%%%%%%%%%%%%%%%%%%%%
\section{Integration over $X_\Phi$} \label{sect: Integration}
%%%%%%%%%%%%%%%%%%%%%%%

We give an alternative proof of Proposition~\ref{prop:Top_Class} by using equivariant cohomology of $X_{\Phi}$.
Given the Dynkin diagram of $\Phi$ (in Figure~\ref{pic: Dynkin diagrams}), we define a subset $K=\{\alpha_1,\ldots,\alpha_{n-1} \} \subset \Sigma$ of simple roots.
Recall that $\Phi_K$ is the root subsystem associated with the connected subset $K$. 
Note that the rank of $\Phi_K$ is $n-1$.
Then the toric variety $X_{\Phi_K}$ associated with $\Phi_K$ is a subvariety of $X_{\Phi}$ with complex codimension $1$ since the weight polytope $P_{\Phi_K}$ corresponds to certain facet of $P_\Phi$ (cf. \cite[Corollary~1.3]{Ren} and \cite[Theorem~6.5]{ACEP}).
By abuse of notation, we denote by $\alpha_n$ the first Chern class $c_1(L_{\alpha_n}^*) \in H^2(X_\Phi)$ as mentioned in Remark~\ref{rem:notation_varpi}.

\begin{lemma} \label{lemm:toric_Phi}
The Poincar\'{e} dual of $X_{\Phi_K}$ in $H^*(X_\Phi)$ is given by
\begin{align*} 
[X_{\Phi_K}] = \frac{|W_K|}{|W|} \alpha_n 
\end{align*}
where $W$ and $W_K$ are the Weyl groups associated with $\Phi$ and $\Phi_K$, respectively.
\end{lemma}

\begin{proof}
Let $W^{K}$ denotes minimal length right coset representatives.
We define copies of $X_{\Phi_K}$ by 
\begin{align*}
\widetilde{X_{\Phi_K}} \coloneqq \coprod_{u \in W^{K}} u \cdot X_{\Phi_K}.
\end{align*}
Then we show that the $T$-equivariant cohomology class $[\widetilde{X_{\Phi_K}}]^T \in H^2_T(X_{\Phi})$ is eqaul to
\begin{align} \label{eq:X_Phi_prime_tilde}
[\widetilde{X_{\Phi_K}}]^T = \alpha_n^T
\end{align}
where $\alpha_n^T$ denotes the equivariant first Chern class $c_1^T(L_{\alpha_n}^*) \in H^2_T(X_\Phi)$.
It follows from the GKM condition for $X_\Phi$ (e.g. \cite[Proposition~8.2]{AHMMS}) that 
the $v$-th component of $[X_{\Phi_K}]^T$ is $\alpha_n^T|_v$ if $v \in W_{K}$ otherwise $0$.
By a similar argument, for each $u \in W^{K}$ we have 
\begin{align*} 
[u \cdot X_{\Phi_K}]^T|_w = \begin{cases}
\alpha_n^T|_w & \textrm{if} \ w \in uW_{K}, \\
0 & \textrm{otherwise.}
\end{cases}
\end{align*}
Hence, we conclude that 
\begin{align*} 
[\widetilde{X_{\Phi_K}}]^T|_w = \sum_{u \in W^{K}} [u \cdot X_{\Phi_K}]^T|_w = \alpha_n^T|_w \ \ \ \text{ for all } w \in W,
\end{align*}
which proves \eqref{eq:X_Phi_prime_tilde}.
Taking the image of both sides in \eqref{eq:X_Phi_prime_tilde} under the forgetful map $H^*_T(X_\Phi) \to H^*(X_\Phi)$, one obtains $\alpha_n = |W^{K}| \cdot [X_{\Phi_K}] = \frac{|W|}{|W_{K}|} \cdot [X_{\Phi_K}]$, as desired.
\end{proof}

\begin{proof}[Proof of Proposition~\ref{prop:Top_Class}]
We first note that $\alpha_i \varpi_i=0$ by \cite[Theorem~3]{Kly} (see also \cite[Theorems~1.1 and 1.4]{AHMMS} and Theorem~\ref{theorem:presentation_Pet}).
This implies that if we write $\varpi_n=c_1\alpha_1+\cdots+c_n\alpha_n$ for some $c_1,\ldots,c_n \in \Q$, then 
\begin{align}
\int_{X_{\Phi}} \varpi_1\varpi_2 \cdots \varpi_n &= c_n \int_{X_{\Phi}} \varpi_1\varpi_2 \cdots \varpi_{n-1}\alpha_n \notag \\
&=c_n \cdot \frac{|W|}{|W_K|} \int_{X_{\Phi}} \varpi_1\varpi_2 \cdots \varpi_{n-1} \cdot [X_{\Phi_K}] \ \ \ (\textrm{by Lemma~\ref{lemm:toric_Phi}}) \notag \\
&=c_n \cdot \frac{|W|}{|W_K|} \int_{X_{\Phi_K}} \varpi_1\varpi_2 \cdots \varpi_{n-1} \label{eq:Proof_Top_Class}. 
\end{align}
Here, the coefficient $c_n$ is described in \cite[p.69]{Hum1} and the order of Weyl groups is written in \cite[p.66]{Hum1}.
Therefore, it is straightforward to see that the right hand side of \eqref{eq:Proof_Top_Class} is equal to $\frac{|W|}{\det (C_{\Phi})}$ described in Table~\ref{tab:A list of values mPhi} by induction and some elementary algebraic manipulations, so we omit the details.
\end{proof}


\smallskip

\begin{thebibliography}{99}

\bibitem{ADGH}
H. Abe, L. DeDieu, F. Galetto, and M. Harada, 
\emph{Geometry of Hessenberg varieties with applications to Newton–Okounkov bodies}, 
Selecta Math. (N.S.) \textbf{24} (2018), no. 3, 2129--2163.

\bibitem{AFZ}
H. Abe, N. Fujita, and H. Zeng, 
\emph{Geometry of regular Hessenberg varieties},
Transform. Groups \textbf{25} (2020), no. 2, 305--333.

\bibitem{AHHM}
H. Abe, M. Harada, T. Horiguchi, and M. Masuda,
\emph{The cohomology rings of regular nilpotent Hessenberg varieties in Lie type A}, 
Int. Math. Res. Not. IMRN \textbf{2019} (2019),  5316--5388. 

\bibitem{AHKZ}
H. Abe, T. Horiguchi, H. Kuwata, and H. Zeng, 
\emph{Geometry of Peterson Schubert calculus in type A and left-right diagrams},
arXiv:2104.02914.

\bibitem{AHMMS}
T. Abe, T. Horiguchi, M. Masuda, S. Murai, and T. Sato,
\emph{Hessenberg varieties and hyperplane arrangements},
J. Reine Angew. Math. \textbf{764} (2020), 241--286.

\bibitem{Amd}
T. Amdeberhan, 
\emph{Explicit computations with the divided symmetrization operator}, 
Proc. Amer. Math. Soc. \textbf{144} (2016), no. 7, 2799--2810.

\bibitem{And}
D. Anderson,
\emph{Introduction to Equivariant Cohomology in Algebraic Geometry},
Contributions to algebraic geometry, 71--92,
EMS Ser. Congr. Rep., Eur. Math. Soc., Z\"{u}rich, 2012.

\bibitem{ACEP}
F. Ardila, F. Castillo, C. Eur, and A. Postnikov, 
\emph{Coxeter submodular functions and deformations of Coxeter permutahedra}, 
Adv. Math. \textbf{365} (2020), 107039, 36 pp.

\bibitem{AtBo}
M. F. Atiyah and R. Bott,
\emph{The moment map and equivariant cohomology}, 
Topology. \textbf{23}, no. 1 (1984), 1--28.

\bibitem{BaHa}
D. Bayegan and M. Harada, 
\emph{A Giambelli formula for the $S^1$-equivariant cohomology of type $A$ Peterson varieties}, Involve \textbf{5} (2012), no. 2, 115--132.

\bibitem{BST}
A. Berget, H. Spink, and D. Tseng,
\emph{Log-concavity of matroid $h$-vectors and mixed Eulerian numbers},
arXiv:2005.01937.

\bibitem{BeVe}
N. Berline and M. Vergne, 
\emph{Classes caract\'eristiques \'equivariantes. Formule de localisation en cohomologie \'equivariante}, 
C. R. Acad. Sci. Paris S\'er. I Math. \textbf{295}, no. 9 (1982), 539--541.

\bibitem{BGG}
I. N. Bernstein, I. M. Gelfand, and S. I. Gelfand,  
\emph{Schubert cells, and the cohomology of the spaces $G/P$}, 
Uspehi Mat. Nauk, \textbf{28} , no. 3(1973), 3--26.

\bibitem{Bil}
S. Billey, 
\emph{Kostant polynomials and the cohomology ring for $G/B$}, 
Duke Math. J. 96 (1999), no. 1, 205--224.

\bibitem{Bor}
A. Borel, 
\emph{Sur la cohomologie des espaces fibr\'es principaux et des espaces homog\'enes de groupes de Lie compacts}, 
Ann. of Math. (2) \textbf{57} (1953), 115--207.

\bibitem{Cro}
D. Croitoru, 
\emph{Mixed Volumes of Hypersimplices, Root Systems and Shifted Young Tableaux}, 
Thesis (Ph.D.)–Massachusetts Institute of Technology. 2010.

\bibitem{dMPS}
F. De Mari, C. Procesi, and M. A. Shayman,
\emph{Hessenberg varieties}, 
Trans. Amer. Math. Soc. {\bf 332} (1992), no. 2, 529--534. 

\bibitem{Dre2}
E. Drellich, 
\emph{Monk's rule and Giambelli's formula for Peterson varieties of all Lie types}, 
J. Algebraic Combin. \textbf{41} (2015), no. 2, 539--575.

\bibitem{Fult97}
W. Fulton,
\emph{Young Tableaux},
London Mathematical Society Student Texts, \textbf{35}. Cambridge University Press, Cambridge.

\bibitem{GoGo}
R. Goldin and B. Gorbutt,
\emph{A positive formula for type $A$ Peterson Schubert calculus},
Matematica \textbf{1} (2022), no. 3, 618--665.

\bibitem{GMS}
R. Goldin, L. Mihalcea, and R. Singh,
\emph{Positivity of Peterson Schubert Calculus},
arXiv:2106.10372.

\bibitem{HHM}
M. Harada, T. Horiguchi, and M. Masuda, \emph{The equivariant cohomology rings of Peterson varieties in all Lie types}, 
Canad. Math. Bull. \textbf{58} (2015), no. 1, 80--90. 

\bibitem{HaTy}
M. Harada and J. Tymoczko, \emph{A positive Monk formula in the $S^1$-equivariant cohomology of type $A$ Peterson varieties}, 
Proc. Lond. Math. Soc. (3) \textbf{103} (2011), no. 1, 40--72.

\bibitem{Hum1}
J. E. Humphreys, 
\emph{Introduction to Lie algebras and representation theory.} 
Graduate Texts in Mathematics, Vol. 9. Springer-Verlag, New York-Berlin, 1972.

\bibitem{IY}
E. Insko and A. Yong. 
\emph{Patch ideals and Peterson varieties}, 
Transform. Groups 17, no. 4 (2012): 1011--1036.

\bibitem{Kly}
A. Klyachko, 
\emph{Orbits of a maximal torus on a flag space},
Functional Anal. Appl. 19, no. 2 (1985): 77--78.

\bibitem{Kly2}
A. Klyachko,
\emph{Toric varieties and flag spaces},
Trudy Mat. Inst. Steklov. \textbf{208} (1995), Teor. Chisel, Algebra i Algebr. Geom., 139--162.

\bibitem{Kos}
B. Kostant, 
\emph{Flag manifold quantum cohomology, the toda lattice, and the representation with highest weight $\rho$}, 
Selecta Math. (N.S.) 2, no. 1 (1996): 43--91.

\bibitem{Liu}
G. Liu, 
\emph{Mixed volumes of hypersimplices},
Electron. J. Combin. \textbf{23}, no. 3 (2016), Paper 3.19, 19 pp.

\bibitem{NT1}
P. Nadeau and V. Tewari,
\emph{Divided symmetrization and quasisymmetric functions}, 
Selecta Math. (N.S.) \textbf{27} (2021), no. 4, Paper No. 76, 24 pp.

\bibitem{NT2}
P. Nadeau and V. Tewari,
\emph{The permutahedral variety, mixed Eulerian numbers, and principal specializations of Schubert polynomials}, 
arXiv:2005.12194.

\bibitem{NT3}
P. Nadeau and V. Tewari,
\emph{A $q$-deformation of an algebra of Klyachko and Macdonald's reduced word formula},
arXiv:2106.03828.

\bibitem{Petr}
F. Petrov, 
\emph{Combinatorial and probabilistic formulae for divided symmetrization}, 
Discrete Math. \textbf{341} (2018), no. 2, 336--340.

\bibitem{Pos}
A. Postnikov, 
\emph{Permutohedra, associahedra, and beyond}, 
Int. Math. Res. Not. IMRN \textbf{2009}, no. 6, 1026--1106.

\bibitem{Pre18}
M. Precup,
\emph{The Betti numbers of regular Hessenberg varieties are palindromic}, 
Transform. Groups 23 (2018), no. 2, 491--499. 

\bibitem{Ren}
L. E. Renner, 
\emph{Descent systems for Bruhat posets},
J. Algebraic Combin. 29 (2009), no. 4, 413--435.

\bibitem{Rie}
K. Rietsch, \emph{Totally positive Toeplitz matrices and quantum cohomology of partial flag varieties}, 
J. Amer. Math. Soc. 16 (2003), no. 2, 363--392.

\end{thebibliography}
\end{document}